\documentclass{amsart}
\usepackage[utf8]{inputenc}\usepackage{amsmath}
\usepackage{amsfonts}
\usepackage{comment}
\usepackage[dvipsnames]{xcolor}
\usepackage{amssymb}
\usepackage[english]{babel}
\usepackage{amsmath,amsthm}
\usepackage{mathtools}
\usepackage{amsfonts}
\usepackage{fancyhdr}
\usepackage{setspace}
\usepackage{color}
\usepackage[percent]{overpic}
\usepackage{fix-cm}
\usepackage{graphicx}
\usepackage[all]{xy}
\usepackage{mathrsfs}
\usepackage{mathdots}
\usepackage{amsbsy}
\usepackage{leftidx}
\usepackage[hidelinks]{hyperref}
\usepackage{xcolor}
\usepackage{todonotes}
\usepackage{tikz-cd}
\hypersetup{
    linktocpage,
    colorlinks,
    linkcolor={magenta!100!black},
    citecolor={cyan!100!black},
    urlcolor={yellow!100!black}
}

\usepackage{tikz}

\theoremstyle{plain}
\newtheorem{thm}{Theorem}[section]
\newtheorem*{thm*}{Theorem}
\newtheorem{lem}[thm]{Lemma}
\newtheorem{prop}[thm]{Proposition}
\newtheorem{cor}[thm]{Corollary}
\newtheorem{cla}[thm]{Claim}

\theoremstyle{definition}
\newtheorem{eg}[thm]{Example}

\newtheorem{rem}[thm]{Remark}

\numberwithin{equation}{section}

\definecolor{bluepic}{RGB}{15,109,237}
\definecolor{magentapic}{RGB}{227,66,166}
\definecolor{greenpic}{RGB}{83,214,63}
\definecolor{brownpic}{RGB}{150,110,55}
\definecolor{orangepic}{RGB}{244,176,59}

\newcommand{\liea}{\mathfrak{a}}

\newcommand{\calb}{\mathcal{B}}

\newcommand{\calt}{\mathcal{T}}
\newcommand{\cale}{\mathcal{E}}
\newcommand{\caleor}{\cale^{\tn{or}}}

\newcommand{\folo}{\mathscr{O}}
\newcommand{\folc}{\mathscr{C}}

\newcommand{\ttt}{\mathtt{t}}
\newcommand{\ttc}{\mathtt{c}}
\newcommand{\tts}{\mathtt{s}}
\newcommand{\parinv}{\mathtt{p}}

\newcommand{\g}{\mathsf{G}}
\newcommand{\PGL}{\mathsf{PGL}_3(\rr)}
\newcommand{\aut}{\mathsf{Aut}}

\newcommand{\flags}{\mathfrak{F}}
\newcommand{\T}{\textnormal{\textbf{T}}}
\newcommand{\Text}{\T^{\tn{ext}}}
\newcommand{\Teq}{\T^{\tn{eq}}}
\newcommand{\Ti}{\T^{\tn{id}}}

\newcommand{\Pol}{\textnormal{\textbf{P}}}
\newcommand{\Poli}{\Pol^{\tn{id}}}
\newcommand{\Hil}{\textnormal{\textbf{H}}}
\newcommand{\HP}{\textnormal{\textbf{HP}}}
\newcommand{\Graph}{\textnormal{\textbf{G}}}

\newcommand{\rr}{\mathbb{R}}
\newcommand{\bba}{\mathbb{A}}
\newcommand{\rp}{\mathbb{P}(\rr^3)}
\newcommand{\rpd}{\mathbb{P}((\rr^3)^*)}
\newcommand{\hh}{\mathbb{H}}

\newcommand{\zz}{\mathbb{Z}}

\newcommand{\cp}{\mathfrak{C}(S)}
\newcommand{\cpf}{\mathfrak{C}^{\tn{fr}}(S)}
\newcommand{\shif}{\mathfrak{S}(S)}
\newcommand{\teich}{\mathfrak{T}(S)}

\newcommand{\od}{\operatorname{d}}
\newcommand\norm[1]{\lVert#1\rVert}

\newcommand{\bc}{\begin{center}}
\newcommand{\ec}{\end{center}}

\newcommand{\dev}{\tn{dev}}

\newcommand{\tn}{\textnormal}
\newcommand{\tc}{\textcolor}

\newcommand{\psieq}{\psi^{\tn{eq}}}

\begin{document}
\title[Triangular metric]{The triangular metric on projective surfaces}

\author{Le\'on Carvajales}
\address{
Le\'on Carvajales \newline Universidad de la Rep\'ublica \newline 
Facultad de Ciencias Econ\'omicas y de Administraci\'on \newline
Instituto de Estad\'istica \newline
CNRS IRL IFUMI}
\email{leon.carvajales@fcea.edu.uy}

\author{Xian Dai}
\address{
Xian Dai \newline Universit\'e Côte d'Azur \newline Laboratoire de Mathématiques
}
\email{xian.dai@univ-cotedazur.fr}

\author{Beatrice Pozzetti}
\address{
Beatrice Pozzetti \newline
Dipartimento di Matematica \newline
Universit\`a di Bologna \newline
IT
}
\email{beatrice.pozzetti@unibo.it}

\author{Anna Wienhard}
\address{
Anna Wienhard \newline
Max Planck Institute for Mathematics in the Sciences\\
Inselstr. 22\\04103 Leipzig, 
Germany
}
\email{wienhard@mis.mpg.de}

\thanks{L.C. was partially funded by  Agencia Nacional de Investigación e Innovaci\'on--FCE$\_3\_2020\_1\_162840$, Ministerio de Educación y Cultura--FVF$\_2023\_436$, and CSIC--Grupos I$+$D 2022 ``Geometría y Acciones de Grupos". X.D. acknowledges funding by the European Research Council under ERC-Advanced grant 101095722 and the Swiss State Secretariat for Education, Research and Innovation (SERI): MB-2500031.
The authors acknowledge support of the Institut Henri Poincaré (UAR $839$ CNRS-Sorbonne Universit\'e), and LabEx CARMIN (ANR--$10-\tn{LABX}-59-01$).
A.W. was supported by the European Research Council under ERC-Advanced Grant 101018839. She thanks the Hector Fellow Academy for support.
}

\begin{abstract}
In \cite{ThurstonStretch} Thurston introduced stretch paths and stretch maps between hyperbolic surfaces. In this paper we extend similar constructions to the space of convex real projective structures on surfaces. Using Fock-Goncharov-like parameters, we introduce stretch paths and construct explicit stretch maps between the projective surface. We show that stretch maps are $e^t$--Lipschitz with respect to a triangular metric we introduce on the projective surface. 
This allows us to gain control on several length functions on an open set of the space of convex real projective structures (those where there is no ``local back--tracking"). 
As an application we obtain pointwise upper bounds for the Hilbert entropy of some convex real projective structures in terms of Fock-Goncharov-like parameters.

\end{abstract}

\maketitle

\setcounter{tocdepth}{1}
\tableofcontents

\section{Introduction}

Let $S$ be a compact orientable surface with non--empty boundary and negative Euler characteristic. 
In this paper we consider geometric structures on $S$ locally modeled on $(\PGL,\rp)$ which are moreover \textit{convex}, which means that the developing map is an embedding into a convex subset of an affine chart of $\rp$.
The associated deformation space $\cp$ was originally studied by Goldman \cite{GoldmanProj}, who in particular showed that $\cp$ is homeomorphic to an open cell on a finite dimensional Euclidean space. Later, Fock--Goncharov \cite{FGProjective} introduced a different parametrization. 

The deformation space of convex real projective structures on $S$ contains the deformation space of hyperbolic structures on $S$, that is a realization of the Teichm\"uller space $\teich$. The space $\cp$ behaves in many aspects similar to $\teich$, but also exhibits interesting new features.  It is one example of a higher Teichm\"uller space, and arguably one of the simplest ones. 

The goal of this paper is to study $\cp$ and explore some of the similarities and differences to $\teich$. Our focus lies on the investigation of special paths, so called \emph{stretch-paths} in the deformation space of convex projective structures $\cp$, and corresponding Lipschitz maps between convex projective surfaces. As an application, we obtain precise (as well as asymptotic) control of the length function with respect to the Hilbert metric on a convex real projective structure. 

We now describe our results in more detail. 

\subsection{Stretch paths in $\teich$}
We let $\teich$ be the space of isotopy classes of hyperbolic metrics on $S$ with geodesic boundary, and often identify a point in $\teich$ with its holonomy representation and denote both by $\rho$. 

The Teichmüller space can be parametrized by shear coordinates \cite{ThurstonStretch,BONShear}. For this we fix an ideal triangulation $\calt_S$ of $S$. Each point $\rho\in\teich$ then induces an identification between $\widetilde{S}$ and a geodesically convex subset of $\hh^2$ which is moreover decomposed as a union of ideal triangles. 

Ideal triangles are rigid pieces in hyperbolic geometry: they are all equivalent up to isometries of $\hh^2$. 
For ideal quadrilaterals there is a unique invariant, the cross ratio of the four vertices at infinity. We interpret this cross ratio as a gluing parameter of gluing two ideal triangles (the rigid pieces) along an edge. 
Given an edge of the triangulation we call the corresponding cross ratio associated to the quadrilateral containing this edge the \textit{shear invariant}. 
In this way, if we let $\cale_S$ be the set of edges of $\calt_S$,  we may associate to each $\rho\in\teich$ a function $\cale_S\to\rr_{>0}$. 
For $E\in\cale_S$ we will let $\tts_E(\rho)$ be the associated shear invariant. This procedure induces a parametrization of $\teich$. 

These shear invariants allow us to define stretch paths in $\teich$. Given an ideal triangulation and a point $\rho_0 \in \teich$, the stretch path $\rho_t\in\teich$ is characterized by $\tts_E(\rho_t) = \tts_E(\rho_0)^{e^t}$.  
Thurston showed that for any stretch path, there is an associated family of \emph{stretch maps}. For each $t$, it is given by an $e^t$-Lipschitz homeomorphism 
\begin{equation}\label{eq: stretch map teich intro}
\psi_t:(S,\rho_0)\to(S,\rho_t)    
\end{equation}\noindent isotopic to the identity. This map is defined locally on the rigid pieces, the ideal triangles, using the horocyclic foliation. It is $e^t$-Lipschitz on each rigid piece and globally extends to an $e^t$-Lipschitz homeomorphism on $S$. 

In this article we extend the construction of stretch paths and stretch maps in the context of convex real projective structures. However, the situation there is more complicated: ideal triangles are not rigid and already have an invariant. We have to assemble triangles from smaller rigid pieces. Moreover, the horocyclic foliation in a convex projective surface is not local, but depends on the global structure. We instead consider a different foliation, which in the hyperbolic case gives an \emph{ortho-geodesic foliation}, as in \cite{CalFarr}. In the Appendix \ref{appendix: maps H2} we provide a description of \emph{ortho-geodesic stretch maps} associated to these ortho-geodesic foliations in the hyperbolic case and prove that they are $e^t$-Lipschitz. They also yield Thurston's stretch paths in the hyperbolic case.

Thurston's Lipschitz map can be used to control the geodesic length functions $L_\rho(\gamma)$ by the length function on a metric graph associated to the shear coordinates on a large subset of $\teich$. 
The subset is the set of 
\textit{shifting} structures \begin{equation}\label{eq: shifting teich intro}
    \shif\coloneqq \{\rho\in\teich: \tts_E(\rho)> 1 \tn{ for all } E\in\cale_S\}.
\end{equation}
For any $\rho \in \shif$ we consider the metric graph which is given by considering the tree $\omega_\rho$ dual to the lift of the triangulation to $\widetilde{S}$, with edge weight $\log(\tts_E(\rho))>0$ on the edge associated to $E\in\cale_{\widetilde{S}}$. We denote by $L_\rho^\Graph$ the corresponding length function. 

Then one has
\begin{equation}\label{eq: comparison lenghts geodesic and graph intro}
    L_\rho^\Graph(\gamma)-D\cdot n_\gamma\leq L_\rho(\gamma)\leq L_\rho^\Graph(\gamma)+D\cdot n_\gamma,
\end{equation}\noindent where $L_\rho(\gamma)$ is the length of the closed geodesic associated to $\rho$ and $\gamma$, $n_\gamma\in\zz_{>0}$ is the translation length of $\gamma$ acting on the tree $\omega_{\calt_{\widetilde{S}}}$ dual to $\calt_{\widetilde{S}}$ (equipped with the graph metric), and $L_\rho^\Graph(\gamma)$ is the translation length of $\gamma$ acting on the tree $\omega_\rho$. 
As $n_\gamma$ and $D$ are independent of $\rho$, we conclude that along the  stretch path $\rho_t$ one has \begin{equation}\label{eq: degeneration teich intro}
\displaystyle\lim_{t\to\infty} \frac{e^{-t}\cdot L_{\rho_t}(\gamma)}{L_\rho^\Graph(\gamma)}=1.    
\end{equation}
Using Thurston's stretch map we conclude that $L_\rho(\gamma)\geq L_\rho^\Graph(\gamma)$.

In the context of convex real projective structures, the stretch paths and stretch maps we introduce will allow us to similarly control the Hilbert length metric in a large subset of $\cp$ as well as its asymptotics. 

\subsection{Coordinates and stretch paths in $\cp$}
We now turn to the case of convex real projective structures on $S$. Again we fix an ideal triangulation of $S$. Associated to such an ideal triangulation, Fock and Goncharov \cite{FG,FGProjective} introduced a generalization of shear coordinates (these are more generally defined cluster coordinates, but this is not of importance to us here). 
These coordinates consist of one coordinate for each triangle and two coordinates for each edge. 
Let's first consider the invariants associated to triangles. Any triangle $T\in\calt_S$ gives an ideal triangle in a convex real projective structure $\rho \in \cp$, which determines, and is uniquely determined by 
a triple of pairwise transverse flags (up to the action of $\PGL$).
There is a unique real valued invariant for the $\PGL$--action on the space of such triples called the \textit{triple ratio}, this number $\ttt_T(\rho)$ is the invariant associated to the triangle $T$. The convexity of the projective structure ensures that $\ttt_T(\rho)\in\rr_{>0}$. 

In fact, this triple ratio can be recast as a gluing parameter \cite{wienahrdzhang}. Given a triple of flags with positive triple ratio, we have two triangles  $\Ti_T(\rho)\subset\Text_T(\rho)\subset\rp$. Each flag consists of a point in $\rp$ inside a line in $\rp$. The interior triangle $\Ti_T(\rho)$ is a triangle with vertices the points, and the external triangle $\Text_T(\rho)$ is the one formed by the three lines, see Figure \ref{fig: distinguished points triple}. 
The nesting $\Ti_T(\rho)\subset\Text_T(\rho)$ allows us to endow $\Ti_T(\rho)$ with the Hilbert metric on $\Text_T(\rho)$. 
The triangle $\Ti_T(\rho)$ can then be cut into three rigid pieces (see Figure \ref{fig: distinguished points triple}), which are cut along a small triangle $\Teq_T(\rho) \subset \Ti_T(\rho)$, which is equilateral with respect to this Hilbert metric. The side length of $\Teq_T(\rho)$, which is precisely the triple ratio, determines the gluing of the rigid pieces completely. We will use this picture heavily in the construction of our stretch maps. 

The coordinates associated to the edges of the triangulation govern how such triangles can be glued together. They are given by appropriate cross ratio functions that come in pairs and are associated  to oriented edges. 
For $E\in\caleor_S$, we let $\overline{E}$ be the edge with the same support as $E$, but equipped with the opposite orientation. 
The corresponding \textit{edge invariant} is denoted by $ (\ttc_E(\rho),\ttc_{\overline{E}}(\rho))$.

Triangle and edge invariants define a parametrization of $\cp$ \cite[Theorem 2.5]{FGProjective}. 
If $S_0$ denotes the interior of $S$ and $\widetilde{S}_0$ denotes its universal cover, then the developing image of $\widetilde{S}_0$ is an open convex set $\Omega_\rho\subset\rp$ which is the union of the ideal triangles $\Ti_T(\rho)$, for $T\in\calt_{\widetilde{S}}$.
We have in fact \begin{equation}\label{eq: Omegarho in intro}
\Omega_\rho = \bigcup_{T\in\calt_{\widetilde{S}}} \Ti_T(\rho)\subset\displaystyle\bigcap_{T\in\calt_{\widetilde{S}}}\Text_T(\rho).    
\end{equation}

Given the triangulation $\calt_S$ and a convex projective structure $\rho\in \cp$ we can define the set of stretch paths. A path $\{\rho_t\} \subset \cp$ is a \textit{stretch path} starting at $\rho$if $\rho_0=\rho$, $\ttt_T(\rho_t) = \ttt_T(\rho)^{e^t}$ and $\ttc_E(\rho_t) \cdot \ttc_{\overline{E}}(\rho_t) = 
(\ttc_E(\rho) \cdot \ttc_{\overline{E}}(\rho))^{e^t}$ for every $T$ and $E$.
Note that this condition gives a whole family of stretch paths as we only require the product of the edge invariants to be stretched. 

For any stretch path $\rho_t$ we construct an explicit family of \textit{stretch maps} $\psi_t$ which, similarly to Thurston's stretch maps, are defined first locally on the rigid pieces of a triangle, then further extended to the full triangle and finally defined globally on the entire surface $S$. 

The definition on each rigid piece resembles Thurston's construction \cite[Proposition 2.2]{ThurstonStretch}: we define two transverse foliations $\folo$ and $\folc$ of these pieces and perform an ``$e^t$--stretch of $\folc$ along $\folo$" to define the map. 
The foliation $\folc$ is the same as used in \cite[Proposition 2.2]{ThurstonStretch}. But the transverse foliation $\folo$ is different: while Thurston uses the horocyclic foliation, we  use a different foliation as horocycles are not a local object in projective geometry.
The leaves of our foliation $\folo$ are given by level sets of the nearest point projection on the corresponding side with respect to the Hilbert metric on $\Ti_T(\rho)$ induced by $\Text_T(\rho)$  (c.f. Figure \ref{fig: foliations}).
In other words, $\folo$ is given by geodesics ``orthogonal" to that side. 

Since the construction depends only on the information of the triangle invariant $\ttt_T(\rho)=\ttt_T(\rho_t)^{\frac{1}{e^t}}$, we may define our map  on the equilateral piece $\Teq_T(\rho)$ as an $e^t$--homothety. 

The construction is compatible with the gluing instruction along edges $E\in\cale_{\widetilde{S}}$, thus inducing a natural $(\rho,\rho_t)$--equivariant homeomorphism $$\psi_t:\Omega_\rho\to\Omega_{\rho_t},$$ which descends to a map $\psi_t: (S,\rho) \to (S,\rho_t)$. We call this map the \textit{ortho--geodesic stretch map} (see Section \ref{sec: stretch map}).

A key property of the stretch map is the following. 

\begin{thm}[c.f. Theorem \ref{thm: Klip}]\label{thm_intro:Lipschitz}
 Let $\{\rho_t\} \subset \cp$ be a stretch path. For every triangle $T \in \calt_S$ and every $t\geq 0$, the ortho--geodesic stretch map $\psi_t: S \to S$ restricted to $\Ti_T(\rho)$ is a Lipschitz map with respect to the Hilbert metric of $\Text_T(\rho)$, with Lipschitz constant $e^t$. 
\end{thm} 

Unfortunately, we cannot control the Lipschitz property of $\psi_t$ with respect to the Hilbert distance on the convex domains $\Omega_\rho$ and $\Omega_{\rho_t}$. 
However, we can use  Theorem \ref{thm_intro:Lipschitz} to get control of the length function with respect to the Hilbert metrics in terms of the length function of a particular metric graph. For this we introduce a new metric, the triangular metric on a convex projective surface. 

\subsection{The triangular metric and metric graphs}
For a general convex projective structure $\rho\in \cp$, the main difficulty is that the Hilbert metric of $\Omega_\rho$ depends drastically on the full information of $\rho$, and it is very hard to relate this with a ``local piece of data" on our rigid pieces. 
In the hyperbolic case this is easier, as the convex domain $\Omega_\rho$ which is the hyperbolic disk does not depend on $\rho$. 
To circumvent this problem for general convex projective structures, we introduce a different distance, the \textit{triangular distance} $d_\T(\cdot,\cdot)$, on $S$ associated to $\rho$. 

We explain the triangle distance. For $T\in\calt_{\widetilde{S}}$, we endow each $\Ti_T(\rho)$ with the Hilbert metric induced by the convex set $\Text_T(\rho)$.  
The associated distance on $\Omega_\rho$, by gluing together all these metric spaces, is the triangle distance $d_\T(\cdot,\cdot)$ (see Section \ref{subsec: triangular metric tuples} for precise definition). It is $\rho$--invariant and the corresponding length function is denoted by $L_\rho^\T$.  
We note that (\ref{eq: Omegarho in intro}) crucially implies \begin{equation}\label{eq: domination hilbert triangular intro}
    L_\rho^\T(\gamma)\leq L_\rho^\Hil(\gamma)
\end{equation} for every $\gamma\in\Gamma$.

The stretch map we constructed gives a natural homeomorphism $$\psi_t:(\Ti_T(\rho),d_\T)\to(\Ti_T(\rho_t),d_\T)$$
\noindent for each triangle $T$ of $\calt_S$. We are able to show the following thanks to local properties of $d_\T(\cdot,\cdot)$. 

\begin{thm}[c.f. Theorem \ref{thm: Klip}]\label{thm: Klip intro}
    Let $\rho\in\cp$ be arbitrary and let $t\mapsto \rho_t$ be a stretch path starting at $\rho$. 
    Then for every $t\geq 0$ the corresponding ortho--geodesic stretch map $\psi_t: (S,\rho) \to (S, \rho_t)  $ is Lipschitz for the underlying triangular distances, with Lipschitz constant $e^t$. 
\end{thm} 

We consider a metric graph  $\omega_\rho$ dual to the triangulation $\calt_{\widetilde S}$. The graph consists of a small triangle for each $T \in \calt_{\widetilde{S}}$, whose vertices correspond to the edges of $T$, and all of whose edges are weighted by $\ttt_T(\rho)$. 
For any two adjacent triangles $T$ and $T'$ glued along an edge $E$ we connect the corresponding vertices of the small triangles by an edge, which is weighted by a \textit{shear invariant} $\tts_E(\rho)$ taking into account the product $\ttc_E(\rho)\cdot \ttc_{\overline{E}}(\rho)$ of the shear invariants, as well as the triple ratios $\ttt_T(\rho)$ and $\ttt_{T'}(\rho)$. 

We define an  open set $\shif \subset \cp$ of \textit{$\T$-shifting representations}, precisely by declaring that the weights of $\omega_\rho$ associated to each $E$ are positive. 
Geometrically, this still means that we ``always shift in the same direction" each time we travel from one ideal triangle to the next one (c.f. Figure \ref{fig: Tshifting}).

We denote by $L_\rho^\Graph$ the length function associated to $\omega_\rho$. 
Inspired by the work of Parreau \cite{ParreauInv} we prove the following.

\begin{thm}[see Proposition \ref{prop: degeneration triangular}]\label{thm: degeneration triangular intro}
    Let $\rho\in\cp$ be $\T$--shifting and $t\mapsto\rho_t$ be a stretch path. 
    Then for every $\gamma\in\Gamma$ we have $$\displaystyle\lim_{t\to\infty} \frac{e^{-t}\cdot L_{\rho_t}^\T(\gamma)}{L_\rho^\Graph(\gamma)}=1.$$
\end{thm}

\subsection{The Hilbert length function and its entropy}
The Hilbert metric is the most natural metric on a convex real projective surface, however its length function  $L_\rho^\Hil$ is difficult to control. 
The stretch maps and their Lipschitz property with respect to the triangular metric allow us to control the Hilbert length function in terms of the graph length function $L_\rho^\Graph$, which is explicitly described in terms of Fock-Goncharov coordinates. 

\begin{thm}[see Corollary \ref{cor: domination}]\label{cor: domination intro}
    Let $\rho\in\cp$ be $\T$--shifting. Then 
    $$L_\rho^\Graph(\gamma)\leq L_\rho^\Hil(\gamma)$$ \noindent for every $\gamma\in\Gamma$.
\end{thm}

This allows us further to study the entropy. 
For $\mathbf{X}=\Hil,\Graph$ we define the \textit{entropy}
$$h_\rho^\mathbf{X}\coloneqq \displaystyle\limsup_{t\to\infty} \frac{1}{t}\cdot\log(\#\{[\gamma]\in[\Gamma]: L_\rho^\mathbf{X}(\gamma)\leq t\})\in\rr_{>0}$$ \noindent (here, $[\Gamma]$ denotes the set of conjugacy classes in $\Gamma$). 
Entropies are important invariants of geometric structures and more generally proper actions on metric spaces, see e.g. \cite{PozBourbaki,WieICM} and references therein.
Theorem~\ref{cor: domination intro} implies \begin{equation}\label{eq: entropy inequality intro}
    h_\rho^\Hil\leq h_\rho^\Graph
\end{equation} \noindent for every $\rho\in\shif$.
The upper bound can be explicitly estimated in terms of the Fock--Goncharov parameters of $\rho$.

To the best of our knowledge, (\ref{eq: entropy inequality intro}) provides the first pointwise estimates for the Hilbert entropy in terms of an appropriate parametrization of $\cp$.
Still, there exist related results in the literature that we now discuss (we recall that $\cp$ can be identified with a subset of the space of conjugacy classes of representations $\Gamma\to\PGL$).

First we note that Martone--Zhang \cite[Theorems 1.4 and 6.8]{MZPos} find, in a setting much more general than ours and for more general length functions, pointwise bounds for the entropy in terms of the \textit{systole} and the \textit{pants--systole}.
It remains unclear to us whether these can be explicitly estimated in terms of an appropriate parametrization of the underlying representation space.

In another direction, a number of existing results provide asymptotic results for entropies. 
Notably, \cite{ZhangDeg,NIEEntropy,NIESimplicial,MZPos} give sufficient conditions on families of diverging sequences in some representation spaces along which certain entropies converge to $0$.
We may apply (\ref{eq: entropy inequality intro}) to obtain similar results. Recall $\shif$ consist of $\T$-shifting representations.

\begin{cor}\label{cor: decay entropy intro}
Let $\rho\in\shif$ and let $t\mapsto\rho_t$ be a  stretch path starting at $\rho$.
Then for every $t\geq 0$ one has $$ h_{\rho_t}^\Hil\leq e^{-t}\cdot h_\rho^\Graph.$$
\end{cor}

In view of Corollary \ref{cor: decay entropy intro} one may wonder whether $h_{\rho_t}^\Hil \sim e^{-t}\cdot h_\rho^\Graph$ holds as $t\to\infty$.
We cannot prove this with our method. However, work by Parreau \cite{ParreauInv} suggests that the answer may be affirmative in some cases. 
In fact, \cite[Theorem 3]{ParreauInv} proves that, under some different assumptions on $\rho$,  $$\displaystyle\lim_{t\to \infty} \frac{e^{-t}\cdot L_{\rho_t}^\Hil(\gamma)}{L_{\rho}^\Graph(\gamma)}=1$$ 
\noindent holds for every $\gamma\in\Gamma$, see Subsections \ref{subsec: parreau intro} and \ref{subsec: parreau} for more details. 
Showing that this convergence is uniform in $\gamma$ would imply the claimed asymptotics.

In recent work Lejtreger--Lema \cite[Theorems 1.1 and 1.3]{LejLem} find asymptotics in terms of the parameters along certain sequences $\{\rho_n\}_n$ of representations. 
When restricted to $\cp$, 
they need to assume that the triangle invariants remain bounded along the sequence $\{\rho_n\}_n$, an hypothesis complementary to ours. 
In either case we insist on the fact that Corollary~\ref{cor: decay entropy intro} is a pointwise statement, not an asymptotic one. 

\subsection{Link with Parreau's work}\label{subsec: parreau intro}

Finally, we elaborate on the link between Theorem \ref{thm: degeneration triangular intro} and Parreau \cite[Theorem 3]{ParreauInv}.

In \cite{ParreauInv} the author introduces two conditions on a representation $\rho\in\cp$ called \textit{left--shifting} and \textit{edge--separation}. 
These conditions guarantee that, for a special stretch path $t\mapsto\rho_t$, the re--scaled Hilbert lengths $e^{-t}\cdot L_{\rho_t}^\Hil(\gamma)$ converge to some length function which we denote as $L_{\rho}^\HP(\gamma)$. The new length function $L_{\rho}^\HP(\gamma)$ is explicit in terms of the Fock-Goncharov parameters of $\rho$, similarly to $L_\rho^\Graph(\gamma)$.

The length spectrum $L_{\rho}^\HP:\Gamma\to\rr_{>0}$ can be described as the one associated to the $\Gamma$--action on a $2$--dimensional complex built by Parreau in a way similar to the graph $\omega_\rho$.
This complex is endowed with a geometric structure locally modeled on a Cartan subspace $\liea$ of $\PGL$. Hence it is truly a higher rank object: it is designed to capture the asymptotic behavior of more general notions of lengths (in particular, $\liea$--\textit{lengths}). 
On the other hand, if one is only interested in Hilbert lengths, one may endow this complex with the natural Finsler metric induced by the Hilbert length functional. 
The associated length spectrum is what we have denoted as $L_\rho^\HP$ in the previous paragraph.
As the following corollary suggests, the choice of this Finsler metric makes the underlying geometric object closer to a ``rank one space".

\begin{cor}[see Corollary \ref{cor: degeneration triangular and hilbert}]\label{cor: degeneration triangular and hilbert intro}
    Let $\rho$ be $\T$--shifting, left--shifting and edge--separating. 
    Let $t\mapsto\rho_t$ be the associated stretch path. 
    Then for every $\gamma\in\Gamma$ one has $L_\rho^\Graph(\gamma)=L_\rho^\HP(\gamma)$ and $$ \displaystyle\lim_{t\to\infty}\frac{L_{\rho_t}^\Hil(\gamma)}{L_{\rho_t}^\T(\gamma)}=1.$$
\end{cor}

\subsection{Organization of the paper}
Sections \ref{sec: projective and hilbert} and \ref{sec: positive tuples} cover standard background material on projective geometry and positive $n$--tuples of flags, and are intended  to fix notations and terminology. 
In Section \ref{sec: triangular distance} we introduce the triangular distance and prove some basic properties. 
In Section \ref{sec: estimates triangular} we introduce the $\T$--shifting condition and prove an estimate (Proposition \ref{prop: additve estimate triangular distance}) which is crucial for the proof of Theorem \ref{thm: degeneration triangular intro}.
In Section \ref{sec: stretch map} we introduce the ortho--geodesic stretch map and prove Theorem \ref{thm: Klip intro}.
We then apply these results in Section \ref{sec: applications} to the study of convex projective structures on surfaces with boundary, in particular we prove Theorems \ref{thm: degeneration triangular intro} and \ref{cor: domination intro}, and Corollary \ref{cor: degeneration triangular and hilbert intro}.
Finally, we include Appendix \ref{appendix: maps H2} where we show that the ortho--geodesic stretch map defines an extremal Lipschitz map in the setting of hyperbolic surfaces. 

\subsection{Acknowledgements} We are grateful to Joaqu\'in Lejtreger and Joaqu\'in Lema for very interesting discussions that led to the statement of Theorem~\ref{cor: domination intro}. 
\subsection{AI disclosure} Generative AI tools were not at all used for the work; all mathematical content was
written and verified by the authors.

\section{Projective invariants and Hilbert geometry}\label{sec: projective and hilbert}

We start with some preliminaries, fixing some terminology and reviewing basic facts of real projective geometry.

In the whole paper we let $\g\coloneqq\PGL$.
We identify the dual projective plane $\rpd$ with the space of projective lines $ l\subset \rp$.
We let $p\oplus p'\in\rpd$ be the element spanned by two different points $p,p'\in\rp$.
Similarly, for different $l,l'\in\rpd$ we denote by $l\cap l'\in\rp$ its intersection point.
A point $p\in\rp$ is \textit{transverse} to $l\in\rpd$ if $p\notin l$.

\subsection{Cross ratios}

The basic invariant we will be dealing with is the cross ratio, whose definition and main properties we now recall. 

The \textit{cross ratio} between two different lines $l,l'\in(\rp)^*$ and two points $p,p'\in\rp$ transverse to them is defined by $$C(l,p,p',l')\coloneqq\frac{l(p')}{l(p)}\cdot\frac{l'(p)}{l'(p')}\in\rr_{\neq 0}.$$ \noindent Formally, to make sense of the expression above we should pick non-zero vectors (resp. co--vectors) $w\in p$ and $w'\in p'$ (resp. $\theta\in l$ and $\theta'\in l'$) and compute the quotient $$\frac{\theta(w')}{\theta(w)}\cdot\frac{\theta'(w)}{\theta'(w')},$$
\noindent which is independent of choices.
In order to facilitate the reading, we will usually avoid emphasizing these choices when working with expressions of this type. 

For the purpose of better expressing certain symmetries, it is useful sometimes to define a cross ratio between points in a projective line, or dually between projective lines passing through a given point. 
For instance, if $p_1,p_2,p_3,p_4\in\rp$ are pairwise different points belonging to a given projective line, we let 
  $$C(p_1,p_2,p_3,p_4)\coloneqq C(l_1,p_2,p_3,l_4),$$ \noindent where $l_1$ (resp. $l_4$) is a projective line through $p_1$ (resp. $p_4$) and transverse to $p_2$ and $p_3$.
Dually, if $l_1,l_2,l_3,l_4\in\rpd$ are pairwise different lines passing through a given point in $\rp$, we let $$C(l_1,l_2,l_3,l_4)\coloneqq C(l_1,p_2,p_3,l_4),$$ \noindent where $p_2\in\rp$ (resp. $p_3\in\rp$) belongs to $l_2$ (resp. $l_3$) and is transverse to $l_1$ and $l_4$.
One may check that these are well defined. 
One of the useful symmetries alluded above is for instance $$C(l_1,l_2,l_3,l_4)=1-C(l_2,l_1,l_3,l_4).$$

The cross ratio is a complete invariant in the following sense: if $(p_1,p_2,p_3,p_4)$ and $(p_1',p_2',p_3',p_4')$ are two quadruples of pairwise different points belonging respectively to $l\in\rpd$ and $l'\in\rpd$, then these ordered quadruples are equivalent under the diagonal action of $\g$ if and only if $$C(p_1,p_2,p_3,p_4)=C(p_1',p_2',p_3',p_4')$$
\noindent (and a natural dual statement holds).
Moreover, we record the following properties that we will use throughout and whose proofs follow from definitions; they admit natural dual statements that we omit.

\begin{lem}[Properties of cross ratios]\label{lem: properties c.r.} 
Let $p_1,p_2,p_3,p_4\in\rp$ and $l_1,l_2,l_3,l_4\in\rpd$ be pairwise different with $p_i\notin l_j$
for $i,j=1,\dots,4$.
Then the following properties hold: 
\begin{enumerate} 
\item \tn{Cocycle:} $C(l_1,p_1,p_2,l_2)\cdot C(l_1,p_2,p_3,l_2)=C(l_1,p_1,p_3,l_2)$. 
\item \tn{Projection:} $C(l_1,p_1,p_2,l_2)=C(l_1,p_3,p_2,l_2)$ if and only if $p_3\in (l_1\cap l_2)\oplus p_1$.

\item \tn{Positivity:} $C(l_1,p_1,p_2,l_2)>0$ if and only if $p_1$ and $p_2$ belong to the same connected component of $\rp\setminus(l_1\cup l_2)$.

\item \tn{Cyclicality:} $C(l_1,p_1,p_2,l_2)>1$ if and only if there exists a cyclic order on the projective line $p_1\oplus p_2$ for which one has $$
    l_1\cap (p_1\oplus p_2)< p_1<p_2< l_2\cap (p_1\oplus p_2)  <l_1\cap (p_1\oplus p_2).$$

\item \tn{Nesting:} assume $C(l_1,p_1,p_2,l_2)> 1$ and $C(l_1,p_1,l_3\cap(p_1\oplus p_2),l_2)<0$.
Then $$\min\{C(l_3,p_1,p_2,l_2),C(l_1,p_1,p_2,l_3)\}>1.$$ \noindent In particular, $$C(l_1,p_1,p_2,l_2)>\max\{C(l_3,p_1,p_2,l_2),C(l_1,p_1,p_2,l_3)\}.$$

\end{enumerate}
\end{lem}

\subsection{Convex sets}\label{subsec: convex sets}

An \textit{affine chart} is a set of the form $\bba\coloneqq \rp\setminus l_\bba$ for some $l_\bba\in\rpd$. 
We fix it from now on, and we also fix a non-zero co--vector $\theta_\bba\in l_\bba$ which induces the identification \begin{equation}\label{eq: affine chart}
    \bba\to\{\theta_\bba=1\}: p\mapsto w_p,
\end{equation} \noindent where $w_p\in p$ is the unique vector that satisfies $\theta_\bba(w_p)=1$.
This allows us to select an orientation on $\bba$ which we also fix from now on.

Given two different points $p$, $p'\in\bba$, we denote by $[p,p']$ the closure of the unique connected component of $p\oplus p'\setminus\{p,p'\}$ that is contained in $\bba$. 
Note that $[p,p']$ depends on the choice of $\bba$, but we do not emphasize this in the notation in the hope that it won't create much confusion. 
The segment $[p,p']\subset\bba$ is the inverse image of the affine segment between $w_p$ and $w_{p'}$ under the map (\ref{eq: affine chart}).
We denote $(p,p')\coloneqq [p,p']\setminus\{p,p'\}$, and define $[p,p')$ and $(p,p']$ in a similar way. 
The following useful remark will be used throughout.

\begin{rem}\label{rem: segment ciclicality and nesting}
Let $p$ and $p'$ be different points in $\bba$. 
Lemma \ref{lem: properties c.r.} allows us to describe $[p,p']$ using appropriate cross ratios. 
For instance, we have $$(p,p')=\{q\in (p\oplus p')\setminus\{p,p'\}: C(q_-,p,q,p')>1\}$$ \noindent for any given point $q_-\in(p\oplus p')\setminus[p,p']$ (e.g. $q_-=l_\bba\cap(p\oplus p')$).
Moreover, if $q_+\in(p\oplus p')\setminus [p,p']$ is so that $C(q_-,p,p',q_+)>1$, then $$ (p,p')=\{q\in (p\oplus p')\setminus\{p,p'\}: 1<C(q_-,p,q,q_+)<C(q_-,p,p',q_+)\}.$$  

From the Nesting Property follows: if $C(p,q_1,q_2,p')>1$ for $q_1,q_2 \in [p,p']$,  then $$C(p,q_1,q_2,p')>\max\{c,c'\}\geq \min\{c,c'\}>1,$$ \noindent where $c\coloneqq C(q_-,q_1,q_2,p')$ and $c'\coloneqq C(p,q_1,q_2,q_-)$.
\end{rem}

A set $C\subset \bba$ is \textit{convex} if for every $p\neq p'$ in $ C$ one has $[ p, p']\subset C$.
If also the closure of $ C$ is contained in $\bba$ we will say that $C$ is \textit{properly convex}.
Equivalently, $C\subset\bba$ is properly convex if and only if its image under (\ref{eq: affine chart}) is convex (in the affine space $\{\theta_\bba=1\}$) and bounded.

\subsection{Polygons}\label{subsec: polygons}

A special family of properly convex sets is given by \textit{polygons}.
By definition these are inverse images $\Pol\subset\bba$, under (\ref{eq: affine chart}), of Euclidean polygons in $\{\theta_\bba=1\}$.
We will always assume that its topological boundary is decomposed into a minimal number of projective segments which will be called the \textit{sides} or \textit{edges} of $\Pol$. 
Two different sides of $\Pol$ are either disjoint or \textit{consecutive}, i.e. they meet in a single point, called a \textit{vertex} of $\Pol$.
Dually, two vertices are \textit{consecutive} if they can be joined by a side of $\Pol$.
An \textit{ideal polygon} is a polygon with its vertices removed.

Special cases of polygons are quadrilaterals and triangles.
A \textit{quadrilateral} is a polygon with four sides. 
These are ``rigid pieces" in projective geometry: the action of $\g$ on the space of quadrilaterals is simply transitive (c.f. (\ref{eq: rigid configuration}) below).
A \textit{triangle} is a polygon with three sides. 
Here is a collection of elementary facts that we will use throughout.

\begin{rem}\label{rem: triangles and subtriangles}
Three points $p_1$, $p_2$ and $p_3$ not aligned determine four triangles in $\rp$, the closures of the connected components of $$\rp\setminus((p_1\oplus p_2)\cup(p_1\oplus p_3)\cup(p_2\oplus p_3)).$$ 
\noindent To specify a triangle is enough to pick a line not passing through these points (to consistently select a side between vertices).
In particular, if $p_1$, $p_2$ and $p_3$ belong to $\bba$ then they canonically determine a triangle $\T\subset\bba$: it is the unique one whose sides do not intersect $l_\bba$.
If moreover $p_1$, $p_2$ and $p_3$ belong to a polygon $\Pol\subset\bba$, then $\T\subset\Pol$.
We note that for every $i=1,2,3$ we have $$\T=\displaystyle\bigcup_{q'\in[p_{i+1},p_{i+2}]}[p_i,q']$$ \noindent (in all what follows the arithmetic in the subscripts is taken modulo $3$).
By Remark \ref{rem: segment ciclicality and nesting}, we may then describe $\T$ as the closure of the set of points $q\in\rp$ that satisfy $$C(l_\bba\cap (p_i\oplus q'),p_i,q,q')>1$$ \noindent for some $q'\in p_{i+1}\oplus p_{i+2}$ satisfying $C(l_\bba\cap (p_{i+1}\oplus p_{i+2}),p_{i+1},q',p_{i+2})>1$.

\end{rem}

\begin{rem}\label{rem: intersections triangle}
    Let $\T\subset\bba$ be a triangle and $q_1$, $q_2$, $q_3$ and $q_4$ be points belonging to their sides.    
    Then in either of the following two situations, the set $\T\setminus[q_1,q_2]$ has two connected components and the projective segments $[q_1,q_2]\subset\bba$ and $[q_3,q_4]\subset\bba$ intersect:
    \begin{itemize}
        \item either $q_1$ and $q_2$ are not vertices of $\T$ and belong to different sides of $\T$ that meet at $q_3$, and $q_4$ belongs to the side not containing $q_3$, or
        \item the point $q_1$ is a vertex of $\T$ and $q_2$ is not, but belongs to the unique side not adjacent to $q_1$, and $q_3\neq q_1$ is a vertex of $\T$ and $q_4$ is any point on the side not adjacent to $q_3$.
    \end{itemize}
Moreover, in the first situation the projective segment $[q_1,q_2]$ does not intersect the side containing $q_4$.
\end{rem}

\subsection{Hilbert geometry}\label{subsec: hilbert geometry}

Let $C\subset\rp$ be an open properly convex set.
We let $d_C(\cdot,\cdot)$ be the \textit{Hilbert distance} on $C$, which is defined as follows.
Let $q\neq q'$ in $C$ and consider the projective line $q\oplus q'$, which intersects the topological boundary of $C$ in two points $x$ and $y$.
By Lemma \ref{lem: properties c.r.} we may assume moreover that $x$ and $y$ are named in such a way that $C(x,q,q',y)> 1$.
We define $$d_C(q,q')\coloneqq\log C(x,q,q',y).$$ \noindent This is a well defined distance on $C$, invariant under the action of the group $$\aut(C)\coloneqq\{g\in\g: g\cdot C=C\}.$$

A \textit{Hilbert geodesic} is a curve $\alpha: \rr\to C$ which is moreover an isometric embedding. 
Using the Cocycle Property one may show that a point $q\in C$ and a non--zero $v\in T_qC$ canonically determine a Hilbert geodesic $\alpha_{q,v}:\rr\to C$ whose image is contained in a projective line, and so that $\alpha_{q,v}(0)=q$ and $\left.\frac{\od}{\od t}\right\vert_{t=0} \alpha_{q,v}(t)=v$.
In general, there might be several Hilbert geodesics passing through $q$ with speed $v$ at time $t=0$, but still we call $\alpha_{q,v}$ the \textit{Hilbert geodesic directed by $v$}.

The Hilbert distance is the path metric induced by a symmetric Finsler norm $\Vert\cdot\Vert_C$ on $C$ called the \textit{Hilbert norm}.
Namely, there exists a continuous family of symmetric norms $$q\mapsto \Vert\cdot\Vert_{q,C}=\Vert\cdot\Vert_C:T_qC\to\rr_{\geq 0}$$ \noindent by means of which one can define a \textit{Hilbert length} $\ell_C(\alpha)$ of piecewise smooth curves $\alpha\subset C$. 
The Hilbert distance $d_C(q,q')$ is the minimal value of all such lengths, for curves connecting $q$ with $q'$.
If $v\in T_qC$ we then have \begin{equation}\label{eq: finsler norm and derivative of distance}
\Vert v \Vert_C=\od_q (q'\mapsto d_C(q,q'))(v).
\end{equation}
\noindent A more explicit description is as follows (c.f. \cite[Subsection 3.2.1]{BenoistDivI} for details).
Using (\ref{eq: affine chart}) we may identify the tangent space of $C$ with $C\times \{\theta_\bba=0\}$.
Given such a tangent vector $(q,v)$ we let $\lambda^\pm=\lambda^{\pm}(q,v)>0$ be defined by the equalities $$v=\lambda^+(w_y-w_q) \tn{ and } v=\lambda^-(w_q-w_x), $$ \noindent where $x$ (resp. $y$) is the endpoint of the Hilbert geodesic directed by $-v$ (resp. $v$).
One has $\Vert v\Vert_C=\lambda^++\lambda^-$.

We record the following infinitesimal version of the Nesting Property.

\begin{rem}\label{rem: nested finsler norms}
    Suppose that $C'\subset\bba$ is another open properly convex subset of $\rp$, and assume moreover that $C'$ contains $C$.
    Then $\Vert v\Vert_{C'}\leq \Vert v\Vert_C$ for all $q\in C$ and $v\in T_qC\cong T_qC'$.
\end{rem}

Here is an important example that we will use in Section \ref{sec: stretch map}.

\begin{eg}\label{ex: finsler norm ball for polygons}
 Let $\Pol\subset\bba$ be a polygon with vertices $p_1,\dots,p_n$ and fix some $q$ in the interior of $\Pol$.
 For each $i=1,\dots,n$ let $v_i\in T_q\Pol$ be a unit vector so that $\alpha_{q,v_i}(\infty)=p_i$.
 Then the unit ball $\{v\in T_q\Pol: \Vert v\Vert_\Pol\leq 1\}$ coincides with the convex hull of $\{v_1,\dots,v_n,-v_1,\dots,-v_n\}$.
\end{eg}

\subsection{Flags and triple ratio}\label{subsec: triple ratio}

A \textit{flag} of $\rr^3$ is a pair $(p,l)\in\rp\times \rpd$ with $p\in l$.
The space of flags $\flags$ is $\g$-homogeneous.

Two flags $F_1=(p_1,l_1)$ and $F_2=(p_2,l_2)$ are said to be \textit{transverse} if $p_1$ is transverse to $l_2$ and $p_2$ is transverse to $l_1$.
The $\g$--action on the space of ordered pairs of transverse flags is transitive. 
If we consider triples instead, the following invariant appears.
Let $(F_1,F_2,F_3)$ be an ordered triple of pairwise transverse flags, with $F_3=(p_3,l_3)$.
The \textit{triple ratio} of $(F_1,F_2,F_3)$ is $$T(F_1,F_2,F_3)\coloneqq\frac{l_1(p_2)}{l_1(p_3)}\cdot\frac{l_2(p_3)}{l_3(p_2)}\cdot\frac{l_3(p_1)}{l_2(p_1)}\in\rr_{\neq 0}.$$
\noindent It is invariant under cyclic permutations of the triple and $$T(F_2,F_1,F_3)=T(F_1,F_2,F_3)^{-1}.$$

\begin{eg}
The special triples with $(l_1\cap l_2)\in l_3$ or $p_3\in p_1\oplus p_2$ satisfy $T(F_1,F_2,F_3)=-1$.
As we will be interested exclusively in the case of positive triple ratios, we assume from now on that our triples $(F_1,F_2,F_3)$ are so that $(l_1\cap l_2)\notin l_3$ and $p_3\notin p_1\oplus p_2$.
We call these \textit{generic triples}.
\end{eg}

We finish this section with a well known interpretation of the triple ratio in terms of appropriate cross ratios. 
In the process we introduce some notations that will be used throughout (c.f. Figure \ref{fig: distinguished points triple} below).
For a generic triple $(F_1,F_2,F_3)$ and different $i,j,k\in\{1,2,3\}$ we consider the (pairwise different) points $p_{ij}\coloneqq l_i\cap l_j$.
Note that $p_i\notin \{ p_{12},p_{23},p_{13}\}$.
Also, we let $s_i\coloneqq(p_i\oplus p_{jk})\cap( p_j\oplus p_k)$, $$t_i\coloneqq(p_j\oplus p_{ik})\cap (p_k\oplus p_{ij}), \tn{ and }r_i\coloneqq(t_i \oplus p_{jk})\cap (p_j\oplus p_k).$$
\noindent These are called the \textit{distinguished points} of the triple.
They do not belong to $l_1\cup l_2\cup l_3$.
\begin{figure}[h] 
\scalebox{1.5}{%
\begin{overpic}[scale=1, width=0.65\textwidth, tics=2]{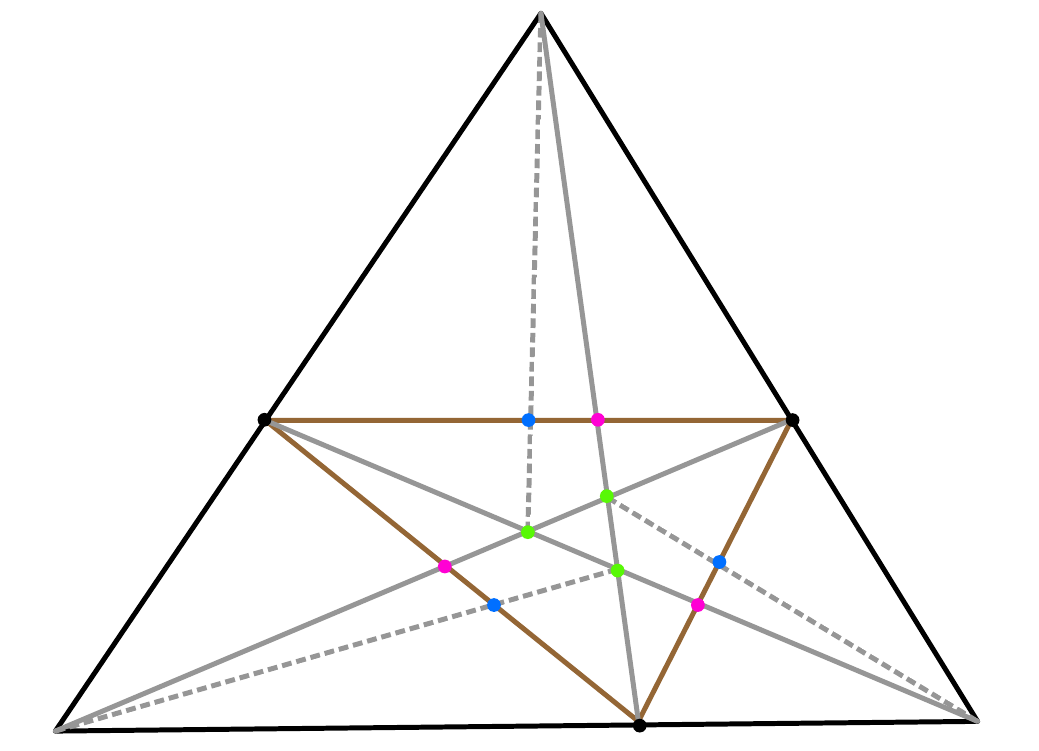}

    \put (21,32)  {\footnotesize$p_1$}
    
    \put (28,42)  {\footnotesize$l_1$}

    \put (78,32)  {\footnotesize$p_3$}
    
    \put (71,42)  {\footnotesize$l_3$}
    
    \put (60,-1)  {\footnotesize$p_2$}
    \put (40,-1)  {\footnotesize$l_2$}

    \put (50,73)  {\footnotesize$p_{13}$}

    \put (95,1)  {\footnotesize$p_{23}$}

    \put (-1,1)  {\footnotesize$p_{12}$}

    \put (47,33)  {\footnotesize\tc{bluepic}{$r_2$}}    
    \put (58,33)  {\footnotesize\tc{magentapic}{$s_2$}}    
    
    \put (45,11)  {\footnotesize\tc{bluepic}{$r_3$}}    
    \put (37.5,17.5)  {\footnotesize\tc{magentapic}{$s_3$}}
    
    \put (71,18)  {\footnotesize\tc{bluepic}{$r_1$}}    
    \put (67,11)  {\footnotesize\tc{magentapic}{$s_1$}}

    \put (47.5,23)  {\footnotesize\tc{greenpic}{$t_{2}$}}  
    \put (56.5,14)  {\footnotesize\tc{greenpic}{$t_{3}$}}
    \put (58.5,26.5)  {\footnotesize\tc{greenpic}{$t_{1}$}}

    \put (53,11)  {\tiny\tc{gray}{$\T_{231}$}}
    \put (37,29)  {\tiny\tc{gray}{$\T_{123}$}}
    \put (63,23)  {\tiny\tc{gray}{$\T_{312}$}}
    
    \put (53.5,20)  {\tiny\tc{greenpic}{$\Teq$}}

   \end{overpic}}	
   \caption{An example with $T(F_1,F_2,F_3)>1$.
   The triangle $\Text$ (resp. $\Ti$) is determined by the black (resp. brown) projective segments.}
   \label{fig: distinguished points triple}
   \end{figure}

\begin{lem}[c.f. Fock--Goncharov {\cite[Section 2.3]{FGProjective}}]\label{lem: t.r. with c.r.}
    Let $(F_1,F_2,F_3)$ be a generic triple of flags.
    Then $$C(l_1,r_2,s_2,l_3)=C(l_3,r_1,s_1,l_2)=C(l_2,r_3,s_3,l_1)=T(F_1,F_2,F_3).$$
    \noindent We also have $$C(l_1,t_2,t_1,l_3)=C(l_3,t_1,t_3,l_2)=C(l_2,t_3,t_2,l_1)=T(F_1,F_2,F_3),$$ \noindent and $$C(l_1,t_2,t_3,l_3)=C(l_3,t_1,t_2,l_2)=C(l_2,t_3,t_1,l_1)=T(F_1,F_2,F_3).$$ 
\end{lem}

A nice consequence of Lemma \ref{lem: t.r. with c.r.} is the following.
Note that given two generic triples $(F_1,F_2,F_3)$ and $(F_1',F_2',F_3')$ there is a unique $g\in\g$ so that 
\begin{equation}\label{eq: rigid configuration}
g\cdot F_1=F_1', \tn{ } g\cdot F_3=F_3', \tn{ and } g\cdot l_2=l_2',
\end{equation}
\noindent where $F_2'=(p_2',l_2')$.
In particular, $g\cdot r_2=r_2'$ (where $r_2'$ is defined for the triple $(F_1',F_2',F_3')$ in the obvious way).
It follows from Lemma \ref{lem: t.r. with c.r.} that the triple ratio is the unique invariant of the $\g$--action on the space of pairwise transverse, ordered, generic triples in $\flags$.
Moreover, the action on each orbit is simply transitive.

\section{Positive tuples of flags}\label{sec: positive tuples}

We now recall well known facts about \textit{positive} $n$--tuples of flags. 
All the material covered in this section is standard (see e.g. Fock--Goncharov \cite[Section 2]{FGProjective}), but we will fix some notations and terminology that will be used in the rest of the paper.

\subsection{Positive triples}\label{subsec: positive triples of flags}

We start with the case $n=3$.
A triple $(F_1,F_2,F_3)$ of pairwise transverse flags is \textit{positive} if $T(F_1,F_2,F_3)>0$.
We let $\flags_3^+$ be the space of positive ordered triples of flags. 
By Lemma \ref{lem: t.r. with c.r.} these triples determine several special triangles that we now analyze, justifying Figure \ref{fig: distinguished points triple} more in detail.

For $(F_1,F_2,F_3)\in\flags_3^+$ the points $p_{12}$, $p_{23}$ and $p_{13}$ are not aligned, because $(F_1,F_2,F_3)$ is generic.
By Remark \ref{rem: triangles and subtriangles} the lines $l_1$, $l_2$ and $l_3$ determine four triangles  of vertices $p_{12}$, $p_{23}$ and $p_{13}$.
Dually, the points $p_1$, $p_2$ and $p_3$ are pairwise different and not aligned. Hence, they also determine four triangles in $\rp$.
We can naturally single out one of these triangles in each case. 
Namely, let $\Text=\Text_{(F_1,F_2,F_3)}$ be the triangle determined by $l_1$, $l_2$ and $l_3$ that contains $t_2$ (in its interior).
Similarly, we let $\Ti=\Ti_{(F_1,F_2,F_3)}$ be the ideal triangle with vertices $p_1$, $p_2$ and $p_3$ that contains $t_2$ (in its interior).

\begin{lem}[c.f. {\cite[Theorem 2.2]{FGProjective}}]\label{lem: nested triangles}
 Let $(F_1,F_2,F_3)\in\flags_3^+$.
    Then $\Ti=\Ti_{(F_1,F_2,F_3)}$ is contained in the interior of $\Text=\Text_{(F_1,F_2,F_3)}$.
    In particular, if $\Text\subset\bba$ then the triangle $\Ti$ is  the unique triangle with vertices $p_1$, $p_2$ and $p_3$ that is contained in $\bba$. 
    As a consequence, for different $i$, $j$ and $k$ in $\{1,2,3\}$ the side $[p_i,p_j]$ of $\Ti$ is the unique projective segment connecting these two points that is disjoint from $l_k$.
    Moreover, the points $t_1$, $t_2$ and $t_3$ belong to the interior of $\Ti$ and $r_1,s_1,r_2,s_2,r_3,s_3$ belong to its sides. 
    Finally, $\Text$ and $\Ti$ are the unique triangles determined by $l_1$, $l_2$, $l_3$, and $p_1$, $p_2$, $p_3$ respectively that satisfy the above inclusion.
    
\end{lem}

Lemma \ref{lem: nested triangles} follows from Lemmas \ref{lem: properties c.r.} and \ref{lem: t.r. with c.r.}. 

From now on we assume $\Text\subset\bba$.
By Lemma \ref{lem: t.r. with c.r.}, if $T(F_1,F_2,F_3)\neq 1$ then the points $t_1$, $t_2$ and $t_3$ are pairwise different and  not aligned. 
By Remark \ref{rem: triangles and subtriangles} they determine a triangle $\Teq=\Teq_{(F_1,F_2,F_3)}\subset\bba$ which is contained in the interior of $\Ti$.
We have the following. 

\begin{lem}[c.f. W.--Zhang {\cite[Corollary 2.13]{wienahrdzhang}}]\label{lem: Teq with cross ratios}
 Let $(F_1,F_2,F_3)$ be a positive triple of flags with $\ttt\coloneqq T(F_1,F_2,F_3)> 1$.
 Then $\Teq$ is equilateral with respect to the Hilbert distance $d_{\T}(\cdot,\cdot)$, with side length equal to $\log (\ttt)$.
 Moreover, one has $$\Teq=\{q\in\rp: \hspace{0,3cm}  1\leq C(l_i,t_{i+1},q,l_{i+1})\leq \ttt \tn{ for all } i=1,2,3\}.$$
\end{lem}

\begin{proof}
The first claim follows directly from Lemma \ref{lem: t.r. with c.r.}, let us then prove that $\Teq$ coincides with $$\T\coloneqq \{q\in\rp: \hspace{0,3cm}  1\leq C(l_i,t_{i+1},q,l_{i+1})\leq \ttt \tn{ for all } i=1,2,3\}.$$

Indeed, by Lemmas \ref{lem: properties c.r.} and \ref{lem: t.r. with c.r.} we have that $\T$ is contained in $\Text$ and is convex.
Applying also Remark \ref{rem: segment ciclicality and nesting} we deduce that the topological boundaries of $\T$ and $\Teq$ coincide.
The lemma follows.
\end{proof}

\begin{rem}\label{rem: Teq when t leq 1}
We have a similar description of $\Teq$ when $\ttt<1$:
$$\Teq=\{q\in\rp: \hspace{0,3cm}  \ttt\leq C(l_i,t_i,q,l_{i+1})\leq 1 \tn{ for all } i=1,2,3\}.$$
\noindent Also, to avoid discussing the degenerate case $T(F_1,F_2,F_3)=1$ separately each time, in that case we will denote $\Teq\coloneqq\{t_1\}=\{t_2\}=\{t_3\}$ and consider it as a ``degenerate'' equilateral triangle of side length equal to zero.

\end{rem}

We introduce one more notation. 
By Remark \ref{rem: triangles and subtriangles} we may consider, for $i=1,2,3$, the triangle $\T_{i(i+1)(i+2)}\subset\Ti$ using the vertices $p_i$, $t_{i+1}$ and $p_{i+2}$ (c.f. Figure \ref{fig: distinguished points triple}).
We have \begin{equation}\label{eq: Tid decompositon}
\Ti=\Teq\cup \T_{123} \cup \T_{231}\cup \T_{312}.
\end{equation}
 
\begin{rem}\label{rem: Tijk intrinsic}
    The triangle $\T_{i(i+1)(i+2)}$ only depends on the flags $F_i$ and $F_{i+2}$, and the hyperplane $l_{i+1}$ (it does not depend on $p_{i+1}$).
    Indeed, this piece of information determines the point $t_{i+1}$ and by Lemma \ref{lem: nested triangles} one may characterize $\T_{i(i+1)(i+2)}$ as the unique triangle of vertices $p_i$, $t_{i+1}$ and $p_{i+2}$ whose sides do not intersect $l_{i+1}$.
    In particular, these triangles are all projectively equivalent to each other in a canonical way (c.f. (\ref{eq: rigid configuration})), and thanks to Lemma \ref{lem: t.r. with c.r.} we may view the triple ratio as a ``gluing parameter" between these pieces (c.f. \cite[Remark 2.14]{wienahrdzhang}).  
\end{rem}

\subsection{Positive quadruples of flags}\label{subsec: positive quadruples}

We now turn our attention to ordered quadruples $(F_1,F_2,F_3,F_4)$ of pairwise transverse flags.
The space of such quadruples has several connected components, one of which corresponds to \textit{positive quadruples} \cite{FG, FGProjective}. 
Positive quadruples are the basic building blocks that allows us to define \textit{positive $n$--tuples} for arbitrary $n\geq 3$ (Subsection \ref{subsec: positive tuples} below).

To define positive quadruples, we restrict ourselves to quadruples $(F_1,F_2,F_3,F_4)$ having generic sub--triples, but there are still several connected components of these. 
Indeed, as $p_2\notin p_1\oplus p_3$ and $p_4\notin p_1\oplus p_3$ we can associate the \textit{Fock-Goncharov-like parameter} in $(\rr_{\neq 0})^2\times (\rr_{\neq 0})^2$ given by\footnote{This parametrization is slightly different to the one used in Fock--Goncharov \cite{FGProjective}, but is better adapted for our purposes, c.f. Lemma \ref{lem: c.r. distinguished points quadruple} below.} \begin{equation}\label{eq: parametrization quadruples}(T(F_1,F_2,F_3),T(F_1,F_3,F_4),C(l_1,p_2,p_4,p_1\oplus p_3),C(l_3,p_4,p_2,p_1\oplus p_3)).
\end{equation}
\noindent The first two parameters are the \textit{triangle invariants} of $(F_1,F_2,F_3,F_4)$, and the remaining two are the \textit{edge invariants}.
Conversely, every $(\ttt_1,\ttt_2,\ttc_1,\ttc_2)\in (\mathbb{R}_{\neq 0})^2\times(\mathbb{R}_{\neq 0})^2$ uniquely determines a $\g$--orbit of quadruples as follows. 
Fix a triple $(F_1,F_2,F_3)$ with $T(F_1,F_2,F_3)=\ttt_1$, this choice is unique up to $\g$--action.
Then, use $\ttc_1$ (resp. $\ttc_2$) to determine a line through $p_1$ (resp. $p_3$), and let $p_4$ be the intersection between these two lines. 
Finally, use $\ttt_2$ to single out a line $l_4$ within the pencil of lines passing through $p_4$, c.f. \cite[Theorem 2.2]{FGProjective}.

A quadruple of flags $(F_1,F_2,F_3,F_4)$ with generic sub--triples is said to be \textit{positive} if its Fock-Goncharov-like parameter belongs to $(\mathbb{R}_{>0})^2\times(\mathbb{R}_{<0})^2$.
We let $\flags_4^+$ be the space of ordered positive quadruples of flags.

The reader may wonder why one selects the specific connected component of parameters $(\mathbb{R}_{>0})^2\times(\mathbb{R}_{<0})^2$ to define positive quadruples, instead of choosing a different one. 
One reason is that, in analogy with Lemma \ref{lem: nested triangles}, these quadruples determine a natural pair of nested quadrilaterals.
The following lemma is well known:

\begin{lem}[c.f. {\cite[Theorem 2.2]{FGProjective}}]\label{lem: nested quadrilaterals}
    Let $(F_1,F_2,F_3,F_4)\in\flags_4^+$.
    Then $\Ti_{(F_1,F_2,F_3)}$ and $\Ti_{(F_1,F_3,F_4)}$ share a common side $[p_1,p_3]$ of extremities $p_1$ and $p_3$, and $\Ti_{(F_1,F_3,F_4)}\setminus[p_1,p_3]$ is contained in the connected component of $\Text_{(F_1,F_2,F_3)}\setminus(p_1\oplus p_3)$ that does not intersect $\Ti_{(F_1,F_2,F_3)}$.
    In particular, $\Ti_{(F_1,F_2,F_3)}\cup\Ti_{(F_1,F_3,F_4)}$ is an ideal quadrilateral.
\end{lem}


Fix some $(F_1,F_2,F_3,F_4)\in\flags_4^+$.
The line $p_1\oplus p_3$ plays a special role in the parametrization (\ref{eq: parametrization quadruples}).
Along this line, we have several distinguished points by Subsection \ref{subsec: positive triples of flags}.
Indeed, the triple $(F_1,F_2,F_3)$ determines the points $r_2,s_2\in p_1\oplus p_3$. 
Similarly, $(F_1,F_3,F_4)$ determines distinguished points in $ p_1\oplus p_3$ that we will denote by $r_4$ and $s_4$, see Figure \ref{fig: product c.r.}.
Lemmas \ref{lem: properties c.r.} and \ref{lem: t.r. with c.r.} imply the following.

\begin{figure}[htbp] 
\scalebox{1}{%
\begin{overpic}[scale=1, width=0.65\textwidth, tics=2]{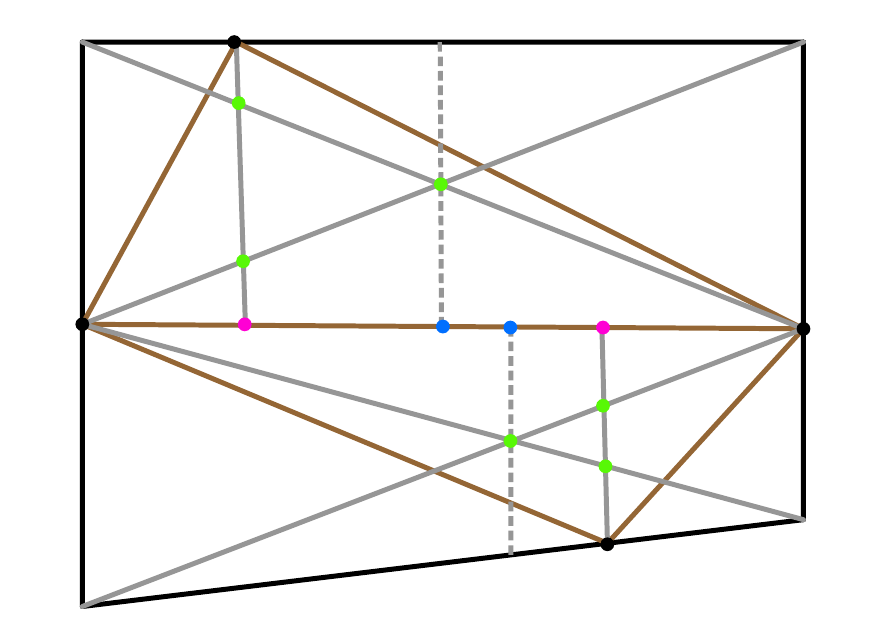}

	\put (4,35)  {$p_1$}
    \put (93,35)  {$p_3$}
    \put (24,70)  {$p_4$}   
    \put (67,7)  {$p_2$}

    
    \put (5,20)  {$l_1$}
    \put (92,50)  {$l_3$}
    \put (69,69)  {$l_4$}   
    \put (22,1)  {$l_2$}

    \put (25,33)  {\tc{magentapic}{$s_4$}}
    \put (66,37.5)  {\tc{magentapic}{$s_2$}}

    \put (48,32)  {\tc{bluepic}{$r_4$}}
    \put (55.5,37.5)  {\tc{bluepic}{$r_2$}}

    \put (31,39)  {\tiny\tc{brownpic}{$\Ti_{(F_1,F_3,F_4)}$}}
    \put (39,30)  {\tiny\tc{brownpic}{$\Ti_{(F_1,F_2,F_3)}$}}
                  
\end{overpic}}	
   \caption{Distinguished points on $p_1\oplus p_3$ associated to a positive quadruple $(F_1,F_2,F_3,F_4)$ (vertical lines pass through $p_{13}=l_1\cap l_3$).}
   \label{fig: product c.r.}
\end{figure}

\begin{lem}\label{lem: c.r. distinguished points quadruple}
    Let $(F_1,F_2,F_3,F_4)$ be a positive quadruple of flags with Fock-Goncharov-like parameter $(\ttt_1,\ttt_2,\ttc_1,\ttc_2)$. 
    Then $C(l_1,s_2,s_4,l_3)=\ttc_1\ttc_2$.
    Moreover, $$C(l_1,r_2,s_4,l_3)=\ttt_1\ttc_1\ttc_2, \hspace{0,3cm} C(l_1,s_2,r_4,l_3)=\ttt_2\ttc_1\ttc_2,  \hspace{0,3cm} \tn{ and } C(l_1,r_2,r_4,l_3)=\ttt_1\ttt_2\ttc_1\ttc_2.$$
    \noindent In particular, $$ \displaystyle\max_{\sigma,\sigma'=r,s}C(l_1,\sigma_2,\sigma'_4,l_3)=\left\lbrace\begin{array}{l}
         \ttc_1\ttc_2 \hspace{1,75cm}\tn{ if } \ttt_1\leq 1 \tn{ and } \ttt_2\leq 1\\
        \ttt_2\ttc_1\ttc_2 \hspace{1,4cm}  \tn{ if } \ttt_1\leq 1 \leq \ttt_2\\
        \ttt_1\ttc_1\ttc_2 \hspace{1,4cm}  \tn{ if } \ttt_2\leq 1 \leq  \ttt_1\\
        \ttt_1\ttt_2\ttc_1\ttc_2 \hspace{1,05cm}  \tn{ if } \ttt_1\geq 1 \tn{ and } \ttt_2\geq 1
    \end{array}\right..$$
\end{lem}

\begin{rem}\label{rem: gluing parameter quadrilateral}
    As in Remark \ref{rem: Tijk intrinsic}, thanks to Lemma \ref{lem: c.r. distinguished points quadruple} we may view the parameters $\ttc_1$, $\ttc_2$ as ``gluing parameters" between the pieces $\Ti_{(F_1,F_2,F_3)}$ and $\Ti_{(F_1,F_3,F_4)}$.
\end{rem}

\subsection{Positive tuples}\label{subsec: positive tuples}

There is an analogue of (\ref{eq: parametrization quadruples}) for arbitrary $n$--tuples, but its statement is more involved as it requires fixing a combinatorial choice (that was made implicitly in the previous section, c.f. Remark \ref{rem: quadruples with triangulated polygons} below).

Let $n\geq 3$ and fix a reference polygon $P$ with vertices $a_1,\dots,a_n$.
The orientation of $\bba$ induces a cyclic ordering on the topological boundary of $P$ and we assume that $$a_1<a_2<\dots<a_n<a_1.$$
We fix moreover a decomposition of $P$ into $n-2$ different triangles whose vertices are vertices of $P$, and so that two of these triangles are either disjoint, they meet in a vertex (of $P$), or they meet along an edge (whose interior is contained in the interior of $P$). 
Let $\calt_P$ be the set of these $n-2$  triangles . 
We also let $\cale_P$ be the set of sides of elements in $\calt_P$ which are not sides of $P$. Note that $\cale_P$ has $n-3$ elements. 
Let $\caleor_P$ be the set of elements of $\cale_P$ equipped moreover with an orientation.
A polygon $P$ equipped with these choices is called a \textit{triangulated polygon}, and we fix it for the rest of the paper.

For $E\in\caleor_P$ we let $i_E\in\{1,\dots,n\}$ (resp. $k_E\in\{1,\dots,n\}$) be so that $a_{i_E}$ (resp. $a_{k_E}$) is the starting (resp. ending) point of $E$.
Let $T_E\in\calt_P$ be the triangle whose vertices are $a_{i_E}$, $a_{j_E}$ and $a_{k_E}$ and such that $$a_{i_E}<a_{j_E}<a_{k_E}<a_{i_E}.$$
\noindent Let $\overline{E}\in\caleor_P$ be the edge with the same support as $E$, but equipped with the opposite orientation.
We have $i_{\overline{E}}=k_E$ and $k_{\overline{E}}=i_E$.
\noindent Note that $T_E\neq T_{\overline{E}}$ are the two triangles of $\calt_P$ meeting at $E$.
We have moreover $$a_{i_E}<a_{j_E}<a_{k_E}<a_{j_{\overline{E}}}<a_{i_E},$$
\noindent see Figure \ref{fig: triangulated polygon}. 
\begin{figure}[h] 
\scalebox{1}{%
\begin{overpic}[scale=1, width=0.65\textwidth, tics=2]{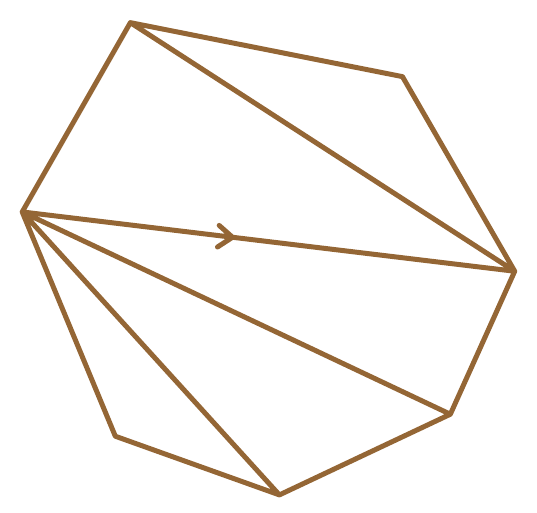}

	\put (-2.5,56)  {\tc{brownpic}{$a_{i_E}$}}
    \put (85,15)  {\tc{brownpic}{$a_{j_E}$}}
    \put (97,44)  {\tc{brownpic}{$a_{k_E}$}}
    \put (21,94)  {\tc{brownpic}{$a_{ j_{\overline{E}}  }$}}

    \put (40,55)  {\tc{brownpic}{$E$}}
    \put (55,40)  {\tc{brownpic}{$T_E$}}
    \put (30,70)  {\tc{brownpic}{$T_{ \overline{E} }$}}
	
    \end{overpic}}
   \caption{A triangulated polygon with seven vertices. 
   The cyclic ordering on the set of vertices is counter clock--wise.
   The oriented edge $E\in\caleor_P$ determines two triangles $T_E$ and $T_{\overline{E}}$, together with a cyclically ordered sequence $a_{i_E}<a_{j_E}<a_{k_E}<a_{j_{\overline{E}}}<a_{i_E}$.}
   \label{fig: triangulated polygon}
\end{figure}

We are now ready to define the analogue of (\ref{eq: parametrization quadruples}).
Let $(F_1,\dots,F_n)$ be an $n$--tuple of pairwise transverse flags. 
Suppose moreover that all its sub--triples are generic. 
We associate to $(F_1,\dots,F_n)$ a \textit{Fock-Goncharov-like parameter} in $$(\rr_{\neq 0})^{\calt_P}\times (\rr_{\neq 0})^{\caleor_P}$$ \noindent as follows.
On the one hand, take $T\in\calt_P$ with vertices $a_i$, $a_j$ and $a_k$ ordered so that $$a_i<a_j<a_k<a_i.$$ \noindent We define the corresponding \textit{triangle invariant} as $\ttt_T\coloneqq T(F_i,F_j,F_k)$.
As the triple ratio is invariant under cyclic permutations, this is well defined. 
For $E\in\caleor_P$, we will usually denote $\ttt_E\coloneqq \ttt_{T_E}$.
On the other hand, we can define an \textit{edge invariant} $\ttc_{E}$ for an element $E\in\caleor_P$ by letting $$\ttc_{E}\coloneqq C(l_{i_E},p_{j_E},p_{j_{\overline{E}}},p_{i_E}\oplus p_{k_E}).$$

As in Subsection \ref{subsec: positive quadruples} and by an inductive argument, a parameter in $(\rr_{\neq 0})^{\calt_P}\times (\rr_{\neq 0})^{\caleor_P}$ uniquely determines an $n$--tuple $(F_1,\dots,F_n)$ up to the action of a well defined $g\in\g$.
The $n$--tuple of flags is said to be \textit{positive} if its triangle invariants are positive and its edge invariants are negative.
If we let $\flags_n^+$ be the space of ordered positive $n$--tuples of flags, we then have the \textit{Fock-Goncharov-like parametrization} 
\begin{equation}\label{eq: parametrization n tuples}
   \flags_n^+/\g\to  (\rr_{>0})^{\calt_P}\times (\rr_{<0})^{\caleor_P}.
\end{equation}

\begin{rem}\label{rem: quadruples with triangulated polygons}
    When $n=4$ then (\ref{eq: parametrization quadruples}) is a special case of (\ref{eq: parametrization n tuples}) as it corresponds to triangulate the quadrilateral $P$ using an edge between $a_1$ and $a_3$.
    If $E\in\caleor_P$ is that edge oriented from $a_1$ to $a_3$, then (\ref{eq: parametrization quadruples}) becomes $$(F_1,F_2,F_3,F_4)\mapsto (\ttt_E,\ttt_{\overline{E}},\ttc_{E},\ttc_{\overline{E}}).$$ \noindent Notice that if we consider the same quadrilateral but triangulated using $a_2$ and $a_4$ instead, we get a different parametrization of $\flags_4^+/\g$ (the change of coordinates can be computed explicitly). 
\end{rem}

The following is analogous to Lemmas \ref{lem: nested triangles} and \ref{lem: nested quadrilaterals}, its proof follows by induction.
Fix  $(F_1,\dots,F_n)\in\flags_n^+$.
By Lemma \ref{lem: nested triangles} each $T\in\calt_P$ with vertices $a_i<a_j<a_k<a_i$ canonically determines a pair of triangles\footnote{Note that these triangles do not depend on the specific ordering $a_i<a_j<a_k<a_i$.}
\begin{equation}\label{eq: ideal triangles for n tuples}
  \Ti_T=\Ti_{T,(F_1,\dots,F_n)}\coloneqq \Ti_{(F_i,F_j,F_k)}\subset\Text_T=\Text_{T,(F_1,\dots,F_n)}\coloneqq\Text_{(F_i,F_j,F_k)}.  
\end{equation}

\begin{lem}[c.f. {\cite[Theorem 2.2]{FGProjective}}]\label{lem: nested polygons}
    Let $(F_1,\dots,F_n)\in\flags_n^+$.
    Then $$ \Poli=\Poli_{(F_1,\dots,F_n)}\coloneqq \bigcup_{T\in\calt_P}\Ti_T$$ \noindent is an ideal polygon contained in the interior of $\bigcap_{T\in\calt_P}\Text_T$ and naturally triangulated in the same way as $P$.
     Moreover the vertices of $\Poli$ are $p_1,\dots,p_n$ and there exists a cyclic ordering on the topological boundary of $\Poli$ so that $$p_1<p_2<\dots<p_{n-1}<p_n<p_1.$$
\end{lem}

We finish this subsection with one last piece of notation.
Let $(F_1,\dots,F_n)\in\flags_n^+$.
For $E\in\caleor_P$ we may consider the positive quadruple $(F_{i_E},F_{j_E},F_{k_E},F_{j_{\overline{E}}})$ which, as in Remark \ref{rem: quadruples with triangulated polygons}, has Fock-Goncharov-like parameter \begin{equation}\label{eq: subparameter}
  (\ttt_E,\ttt_{\overline{E}},\ttc_{E},\ttc_{\overline{E}}).  
\end{equation}
\noindent We observe that, strictly speaking, $(F_{i_E},F_{j_E},F_{k_E},F_{j_{\overline{E}}})$ is not a sub--quadruple of $(F_1,\dots,F_n)$ in the sense that we cannot ensure that $$1\leq i_E<j_E<k_E<j_{\overline{E}}\leq n$$ \noindent (this only holds up to a cyclic permutation of $ (i_E, j_E, k_E, j_{\overline{E}})$). 
We will however ignore this subtlety and call $(F_{i_E},F_{j_E},F_{k_E},F_{j_{\overline{E}}})$ the \textit{sub--quadruple} of $(F_1,\dots,F_n)$ associated with $E\in\caleor_P$, and call  (\ref{eq: subparameter}) the corresponding \textit{sub--parameter}.
Observe that the sub--quadruple associated to $\overline{E}$ is $(F_{k_E},F_{j_{\overline{E}}},F_{i_E},F_{j_E})$ and the corresponding sub--parameter is \begin{equation}\label{eq: subparameter of overline E}
    (\ttt_{\overline{E}},\ttt_E,\ttc_{\overline{E}},\ttc_{E}).
\end{equation}

Finally, for each oriented edge $E$ we have the distinguished points
\begin{equation}\label{eq: distinguished points for E in caleor}
    r_E, s_E\in[p_{i_E},p_{k_E}]\subset\bba,
\end{equation} \noindent coming from the triple $(F_{i_E},F_{j_E},F_{k_E})$.
We emphasize that $r_E$ and $s_E$ depend not only on the support of $E$ but also on its orientation.

\section{Triangular distance}\label{sec: triangular distance}

For the rest of the paper we fix a triangulated polygon $P$ with $n\geq 3$ vertices as in Subsection \ref{subsec: positive tuples}, and consider the induced Fock-Goncharov-like parametrization (\ref{eq: parametrization n tuples}) of $\flags_n^+/\g$.
Our goal now is to introduce a natural distance on the ideal polygon determined by a given $(F_1,\dots,F_n)\in\flags^+_n$, and to explore some of its basic properties.

\subsection{Definition and first properties}\label{subsec: triangular metric tuples}

Recall the notations from Lemma \ref{lem: nested polygons}.
Roughly speaking, the triangular distance is defined by gluing together the Hilbert distances on $\Ti_T$ induced by the properly convex sets $\Text_T$ (recall Subsection \ref{subsec: hilbert geometry}).

More precisely, let $(F_1,\dots,F_n)\in\flags_n^+$ and $\Poli=\Poli_{(F_1,\dots,F_n)}$ be the corresponding ideal polygon.
Recall that this is the union of the ideal triangles $\Ti_T$, for $T\in\calt_P$, and so that each $\Ti_T$ is contained in a unique $\Text_T$.
Let $\alpha\subset\Poli$ be a piecewise smooth curve. 
We may decompose it as a concatenation $$\alpha=\alpha_1\star\dots\star\alpha_M$$ \noindent where, for each $m=1,\dots,M$, the curve $\alpha_m$ is entirely contained in $\Ti_{T_m}$ for some $T_m\in\calt_P$.
We define the \textit{triangular length} of $\alpha$ as $$\ell_{\T}(\alpha)\coloneqq \displaystyle\sum_{m=1}^M\ell_{\Text_{T_m}}(\alpha_m),$$ \noindent where $\ell_{\Text_{T_m}}(\cdot)$ denotes the length function associated to the Hilbert norm of $\Text_{T_m}$.

The triangular length is well defined. 
Indeed, note first that if some $\alpha_m$ is contained in some $\Ti_T$ with $T\neq T_m$, then $\alpha_m$ is contained in the common projective segment $\Ti_T\cap\Ti_{T_m}$.
As the Hilbert norms $\Vert\cdot\Vert_{\Text_T}$ and $\Vert\cdot\Vert_{\Text_{T_m}}$ agree on vectors tangent to this projective segment, we have $\ell_{\Text_{T_m}}(\alpha_m)=\ell_{\Text_T}(\alpha_m)$ in that case.
Moreover, it is also clear that $\ell_{\T}(\alpha)$ does not depend on the chosen decomposition of $\alpha$ as a concatenation of curves.

The \textit{triangular distance} $d_{\T}(\cdot,\cdot)$ is the path metric induced by $\ell_{\T}(\cdot)$.
It is non-negative, symmetric, and satisfies the triangular inequality. 
The following shows that it is a genuine distance on $\Poli$ extending each $d_{\Text_T}(\cdot,\cdot)$.

\begin{lem}\label{lem: triangular distance separates and extends}
    Let $(F_1,\dots,F_n)\in\flags^+_n$ and $T\in\calt_P$. 
    Then for every $q$ and $q'$ in $\Ti_T$ one has $$d_{\T}(q,q')=d_{\Text_T}(q,q').$$
    \noindent In particular, $d_{\T}(\cdot,\cdot)$ separates points. 
\end{lem}

\begin{proof}

We clearly have $ d_{\T}(q,q') \leq d_{\Text_T}(q,q') $.
To prove the equality, we will show that for every piecewise smooth curve $\alpha\subset\Poli$ joining $q$ with $q'$, there exists a piecewise smooth curve $\widehat{\alpha}\subset\Ti_T$ joining $q$ with $q'$ such that \begin{equation}\label{eq: triangular distance separates and extends}
    \ell_{\Text_T}(\widehat{\alpha})\leq \ell_{\T}(\alpha).
\end{equation}

Fix then some $\alpha$, decomposed as a concatenation $\alpha=\alpha_1\star\dots\star\alpha_M$ as in the definition of the triangular length.
As projective lines are geodesics for the Hilbert distance we may assume that $\alpha_m$ is contained in a projective segment for $m=1,\dots,M$.
Note that each time that there are some $1\leq m_0<m_0'\leq M$ with $\alpha_{m_0}\star\dots\star\alpha_{m_0'}$ contained in some $\Ti_{T'}$, then we may replace this sub--polygonal by a projective segment inside $\Ti_{T'}$ without increasing the total triangular length.
Moreover, unless $\alpha\subset\Ti_T$ (in which case we are done with the proof of (\ref{eq: triangular distance separates and extends})), the above sub--polygonal must exist. 
By induction we conclude that (\ref{eq: triangular distance separates and extends}) holds.

To finish, let us prove that if arbitrary points $q$ and $q'$ in $\Poli$ satisfy $d_{\T}(q,q')=0$, then we necessarily have $q=q'$.
We will show that in fact $q$ and $q'$ belong to some $\Ti_T$, and this will finish the proof as $d_{\Text_T}(\cdot,\cdot)$ does separates points.

Suppose for a contradiction that $q$ and $q'$ do not belong to the same ideal triangle. 
As above, for every $\varepsilon>0$ we may find a polygonal $\alpha=\alpha_1\star\dots\star\alpha_M$ connecting $q$ with $q'$, with $M\geq 2$, such that the curves $\alpha_m$ and $\alpha_{m+1}$ are not contained in a single ideal triangle for $m=1,\dots,M-1$, and $\ell_{\T}(\alpha)<\varepsilon$.
In particular, $\alpha_1\subset\Ti_{T_1}$ is a projective segment connecting $q$ with a point in some edge $E\subset\Ti_{T_1}$ not containing $q$ and $$\ell_{\Text_{T_1}}(\alpha_1)=\ell_{\T}(\alpha_1)<\varepsilon.$$
\noindent As $\varepsilon$ is arbitrary, this would show that the $d_{\Text_{T_1}}(\cdot,\cdot)$--distance between $q$ and some edge $E$ not containing $q$ is zero. 
As $d_{\Text_{T_1}}(\cdot,\cdot)$ separates points and satisfies the triangular inequality, this is a contradiction. 
\end{proof}

We record for future use the following consequence of the proof of Lemma \ref{lem: triangular distance separates and extends}.

\begin{rem}\label{rem: minimizing triangular curves}
    Let $q,q'\in\Poli$ and $T_1,\dots,T_M\in\calt_P$ be so that $ \Ti_{T_1}\cup\dots\cup \Ti_{T_M}$ is an ideal polygon containing $q$ and $q'$.
    Then the triangular distance between $q$ and $q'$ may be computed by only looking at polygonal curves $\alpha\subset \Ti_{T_1}\cup\dots\cup \Ti_{T_M}$ which in restriction to each $\Ti_{T_i}$ are given by a projective segment.
\end{rem}

Here is a useful lemma about the triangular distance on a single triangle that we will use repeatedly, recall the decomposition (\ref{eq: Tid decompositon}) and see Figure \ref{fig: bound triangle}.

\begin{figure}[htbp] 
\scalebox{1}{%
\begin{overpic}[scale=1, width=0.65\textwidth, tics=2]{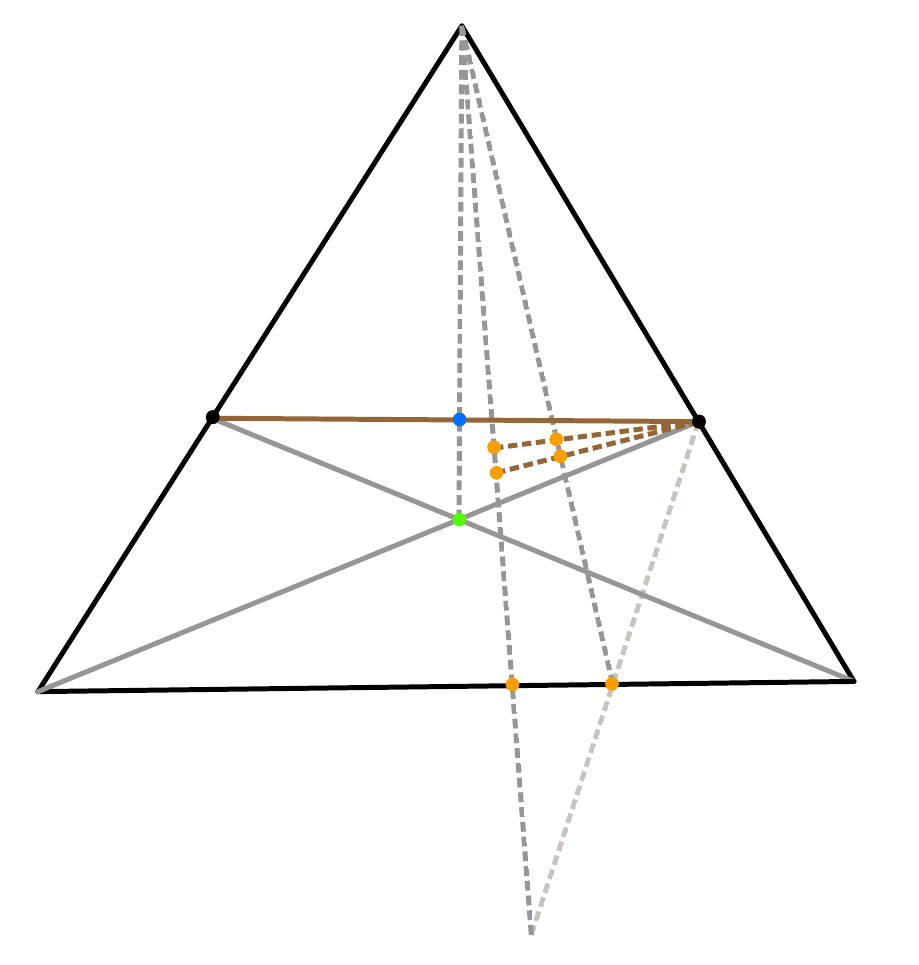}

	\put (17,56)  {$p_1$}
    \put (75,56)  {$p_3$}
    \put (46,99)  {$p_{13}$}
    \put (-2,28)  {$p_{12}$}
    \put (90,28)  {$p_{23}$}
    
    \put (49,25)  {\tc{orangepic}{$\widehat{q}_1$}}
    \put (64,25)  {\tc{orangepic}{$\widehat{q}_2$}}
    \put (53.5,-0.5)  {\tc{orangepic}{$\widetilde{q}_1$}}

    \put (29,75)  {$l_1$}
    \put (64,75)  {$l_3$}
    \put (75,25)  {$l_2$}
   
    \put (46,42)  {\tc{greenpic}{$t_2$}}
    \put (43,58.5)  {\tc{bluepic}{$r_2$}}

    \put (48,49)  {\tc{orangepic}{$q_1'$}}
    \put (47.5,53)  {\tc{orangepic}{$q_1$}}
    \put (59.5,48.5)  {\tc{orangepic}{$q_2'$}}
    \put (57,58)  {\tc{orangepic}{$q_2$}}

    \end{overpic}}
   \caption{The proof of Lemma \ref{lem: bound triangular pencil} using the Projection and Nesting Properties of the cross ratio.}
   \label{fig: bound triangle}

\end{figure}

\begin{lem}\label{lem: bound triangular pencil}
For $(F_1,F_2,F_3)\in\flags_3^+$ let $q_1$ and $q_2$ in $\T_{123}$ with $q_2\in q_1\oplus p_3$ and $$1\leq C(l_1,r_2,q_1,l_3)<C(l_1,r_2,q_2,l_3).$$ \noindent Moreover, for $i=1,2$ let $q_i'\in \T_{123}$ be so that $q_i'\in p_{13}\oplus q_i$ and $q_2'\in q_1'\oplus p_3$.
Then  $$d_{\T}(q_2,q_2')<d_{\T}(q_1,q_1').$$
\noindent In particular, there exists a constant $D>0$ so that for every $(F_1,F_2,F_3)\in\flags_3^+$ and every $q,q'\in\T_{123}$ satisfying $q'\in q\oplus p_{13}$ one has $d_{\T}(q,q')\leq  D$.
\end{lem}

\begin{proof} 
For $i=1,2$ we let $\widehat{q}_i\coloneqq l_2\cap (q_i\oplus q_i')$ and note that by Remark \ref{rem: triangles and subtriangles} we have $\widehat{q}_i\in\Text$.
Moreover, up to exchanging the roles of $q_1$ and $q_1'$ if necessary, we may assume without loss of generality that $C(\widehat{q}_1,q_1',q_1,p_{13})>1$.
Hence $$d_{\T}(q_1,q_1')=\log C(\widehat{q}_1,q_1',q_1,p_{13}).$$
It suffices then to show the following inequalities
$$C(\widehat{q}_1,q_1',q_1,p_{13})>C(\widehat{q}_2,q_2',q_2,p_{13})>1.$$
\noindent This will follow from Lemma \ref{lem: properties c.r.} (c.f. Figure \ref{fig: bound triangle}).

Indeed, letting $\widetilde{q}_1\coloneqq (p_3\oplus \widehat{q}_2)\cap(q_1\oplus q_1')$, by the Projection Property this is equivalent to showing $$C(\widehat{q}_1,q_1',q_1,p_{13})>C(\widetilde{q}_1,q_1',q_1,p_{13})>1.$$ \noindent By Remark \ref{rem: segment ciclicality and nesting} we only need to show that $\widetilde{q}_1\notin[p_{13},\widehat{q}_1]\subset\bba$.
Now, by the Projection Property our hypothesis on $q_1$ and $q_2$ implies that $C(p_{12},\widehat{q}_1,\widehat{q}_2,p_{23})>1$ and so $\widehat{q}_2$ belongs to $(\widehat{q}_1,p_{23})$ (c.f. Remark \ref{rem: segment ciclicality and nesting}).
As $p_3\in(p_{13},p_{23})$, by Remark \ref{rem: intersections triangle} the projective segments $[p_3,\widehat{q}_2]$ and $[p_{13},\widehat{q}_1]$ do not intersect, thus showing $\widetilde{q}_1\notin[p_{13},\widehat{q}_1]$ as desired.

Finally, if in particular $q_1,q_1'\in[r_2,t_2]$, then $$d_{\T}(q_2,q_2')<d_{\T}(q_1,q_1')\leq d_{\T}(r_2,t_2).$$
\noindent But (\ref{eq: rigid configuration}) implies that $d_{\T}(r_2,t_2)$ is a universal constant, thus finishing the proof.   
\end{proof}

\subsection{Comparing triangular distances}

We explicitly state two useful properties comparing triangular distances associated to different $n$--tuples for future use. 
We recall the notation (\ref{eq: ideal triangles for n tuples}) for the following, whose proof is direct (c.f. Remark \ref{rem: minimizing triangular curves}). 

\begin{lem}\label{lem: sufficient condtion for lipshictz triangular}
  Fix two positive $n$--tuples of flags $(F_1,\dots,F_n)$ and $(F_1',\dots,F_n')$ and let $(\Poli)'\coloneqq \Poli_{(F_1',\dots,F_n')}$.
  For $T\in\calt_P$, let  $(\Ti_T)'\coloneqq \Ti_{T,(F_1',\dots,F_n')}$.
  Let $$\psi:\Poli\to(\Poli)'$$ \noindent be a homeomorphism so that $\psi(\Ti_T)=(\Ti_T)'$ for all $T\in\calt_P$.
Suppose that there is some $K>0$ so that for every $T\in\calt_P$ and every projective segment $\alpha\subset\Ti_T$ the curve $\psi\circ\alpha$ is piecewise smooth and satisfies $$\ell_{\T}(\psi\circ\alpha)\leq K\cdot \ell_{\T}(\alpha).$$
\noindent Then $\psi$ is $K$--Lipschitz for the respective triangular distances.   
\end{lem}

The following result states that the triangular distance does not depend on the full information of the $n$--tuple $(F_1,\dots,F_n)$: it only depends on the individual triangle invariants and on the products of the edge invariants.
Hence, each positive $n$--tuple determines an $(n-3)$--dimensional family of positive $n$--tuples with  the same triangular distance, c.f. Remarks \ref{rem: Tijk intrinsic} and \ref{rem: gluing parameter quadrilateral}.
We keep the notations from Lemma \ref{lem: sufficient condtion for lipshictz triangular}.

\begin{lem}\label{lem: isometric triangular metrics}
    Fix two positive $n$--tuples of flags $(F_1,\dots,F_n)$ and $(F_1',\dots,F_n')$.
    Suppose moreover that:
    \begin{enumerate}
        \item for each $T\in\calt_P$, the corresponding triangle invariants for $(F_1,\dots,F_n)$ and $(F_1',\dots,F_n')$ coincide, and
        \item for each $E\in\cale_P$, the product of the associated edge invariants for $(F_1,\dots,F_n)$ coincides with the product of the associated edge invariants for $(F_1',\dots,F_n')$. 
    \end{enumerate}
    \noindent Then there exists a piecewise projective homeomorphism $$\psi:\Poli\to(\Poli)'$$ \noindent preserving the respective triangulations. 
    More precisely, for all $T\in\calt_P$ one has $\psi(\Ti_T)=(\Ti_T)'$ and the restriction $\psi\vert_{\Ti_T}$ extends to an element of $\g$.
    In particular, $\psi$ is an isometry for the respective triangular distances. 
\end{lem}

Lemma \ref{lem: isometric triangular metrics} is a consequence of (\ref{eq: rigid configuration}) and Lemmas \ref{lem: t.r. with c.r.} and \ref{lem: c.r. distinguished points quadruple}.
We omit the details and refer the reader to Lemmas \ref{lem: gluing psii and psieq} and \ref{lem: gluning maps of triples} below where a closely related argument, that is easily adaptable to prove Lemma \ref{lem: isometric triangular metrics}, is done in full detail.

\section{Shifting tuples and estimates}\label{sec: estimates triangular}

We now prove an estimate (Proposition \ref{prop: additve estimate triangular distance} below) for the triangular distance of $n$--tuples satisfying a special condition that we call \textit{$\T$--shifting}.
This estimate is crucial for the proof of Theorem \ref{thm: degeneration triangular intro}.
Throughout this section we assume $n\geq 4$.

\subsection{Shifting tuples}\label{subsec: Tshifting tuples}

The notion of \textit{$\T$--shifting} parameters (and \textit{$\T$--shifting $n$--tuples} of flags) depends on our reference polygon $P$ and intuitively expresses the following.
Let $E\in\caleor_P$ be an oriented edge of the triangulation of $P$.
By (\ref{eq: subparameter}) this determines the sub--parameter $(\ttt_E, \ttt_{\overline{E}},\ttc_{E},\ttc_{\overline{E}})$ and the positive sub--quadruple $(F_{i_E},F_{j_E},F_{k_E},F_{j_{\overline{E}}})$.
Also, the projective line $p_{i_E}\oplus p_{k_E}$ contains the four distinguished points $r_E$, $s_E$, $r_{\overline{E}}$ and $s_{\overline{E}}$, recall (\ref{eq: distinguished points for E in caleor}). 
Depending on the specific $n$--tuple, one may have different relative positions between these four points. 
The $\T$--shifting condition forbids some specific relative positions for all oriented edges $E$ simultaneously in the sense that it implies that, as we travel along $[p_{i_E},p_{k_E}]$, starting from $p_{i_E}$ and heading towards $p_{k_E}$, we first see the points $r_{\overline{E}}$ and $s_{\overline{E}}$ (appearing in some order), and then the points $r_E$ and $s_E$ (appearing possibly in a different order), see Figure \ref{fig: Tshifting}.
Observe that this condition depends only on the support of $E$ and not on its specific orientation, but it is easier to express if we fix an auxiliary orientation for the underlying edge. 
Note as well that the above condition is stronger than just saying that the distinguished points $\{r_E,s_E\}$ and $\{r_{\overline{E}},s_{\overline{E}}\}$ are ``unlinked", as we require the same relative position between the sets $\{r_E,s_E\}$ and $\{r_{\overline{E}},s_{\overline{E}}\}$ for all edges $E\in\caleor_P$: we always require the first set appears after the second one as one travels following the orientation of $E$.

Let us be more precise.
We say that a positive $n$--tuple $(F_1,\dots,F_n)\in\flags_n^+$ (or its Fock-Goncharov-like parameter) is \textit{$\T$--shifting} if for every $E\in\caleor_P$ the \textit{shear invariant}  
\begin{equation}\label{eq:shear_invariant}\tts_E\coloneqq\min\left\{\left(\ttc_E\ttc_{\overline{E}}\right)^{-1},\tn{ } \left(\ttt_E\ttc_{E}\ttc_{\overline{E}}\right)^{-1},\tn{ } \left(\ttt_{\overline{E}}\ttc_{E}\ttc_{\overline{E}}\right)^{-1},\tn{ } \left(\ttt_E\ttt_{\overline{E}}\ttc_{E}\ttc_{\overline{E}}\right)^{-1}\right\}
\end{equation}
\noindent satisfies $\tts_E > 1$.
Note that thanks to (\ref{eq: subparameter of overline E}) we have $\tts_E=\tts_{\overline{E}}$.
Moreover, from Lemma \ref{lem: c.r. distinguished points quadruple} we have 
\begin{equation}\label{eq: shear quadruple}
\tts_E=\displaystyle\min_{\sigma,\sigma'=r,s}C(l_{i_E},\sigma'_{\overline{E}},\sigma_E,l_{k_E})=\left(\displaystyle\max_{\sigma,\sigma'=r,s}C(l_{i_E},\sigma_E,\sigma'_{\overline{E}},l_{k_E})\right)^{-1},
\end{equation} \noindent see Figure \ref{fig: Tshifting}.
\begin{figure}[htbp] 
\scalebox{1}{%
\begin{overpic}[scale=1, width=0.65\textwidth, tics=2]{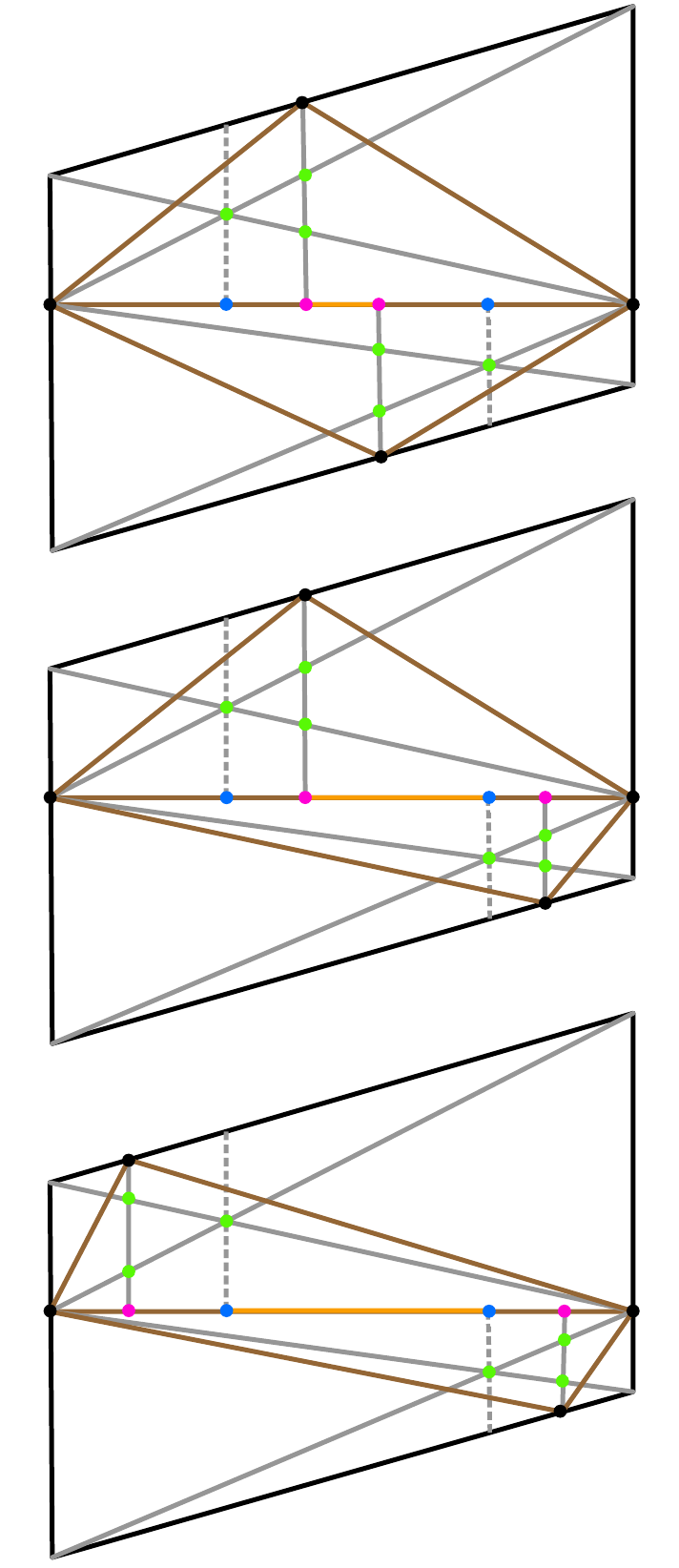}

    \put (0,16)  {$p_{i_E}$}
	\put (-1,49)  { $p_{i_E}$}
    \put (-1,80.5)  { $p_{i_E}$}

    \put (0,10)  { $l_{i_E}$}
	\put (0,42)  { $l_{i_E}$}
    \put (0,73)  { $l_{i_E}$}
   
    \put (41,16)  { $p_{k_E}$}
	\put (41,49)  { $p_{k_E}$}
    \put (41,80)  { $p_{k_E}$}

    \put (41,25)  { $l_{k_E}$}
	\put (41,57)  { $l_{k_E}$}
    \put (41,88)  { $l_{k_E}$}

	\put (34.5,8.5)  { $p_{j_E}$}
	\put (33.5,41)  { $p_{j_E}$}
    \put (23,69)  { $p_{j_E}$}

    \put (20,4)  { $l_{j_E}$}
	\put (20,36.5)  { $l_{j_E}$}
    \put (15,66.5)  { $l_{j_E}$}

	\put (6.5,27.5)  { $p_{j_{\overline{E}}}$}
	\put (17.5,63.5)  { $p_{j_{\overline{E}}}$}
    \put (17.5,95)  { $p_{j_{\overline{E}}}$}

    \put (28,33.5)  { $l_{j_{\overline{E}}}$}
	\put (28,66.5)  { $l_{j_{\overline{E}}}$}
    \put (28,98)  { $l_{j_{\overline{E}}}$}

    \put (30.5,81.5)  {\tc{bluepic}{$r_E$}}
    \put (12,81.5)  {\tc{bluepic}{$r_{\overline{E}}$}}

    \put (27.5,49.7)  {\tc{bluepic}{ $r_E=q_E$}}
    \put (12,50)  {\tc{bluepic}{$r_{\overline{E}}$}}

    \put (28.2,16.9)  {\tc{bluepic}{ $r_E=q_E$}}
    \put (13.8,15.15)  {\tc{bluepic}{$r_{\overline{E}}=q_{\overline{E}}$}}

    \put (23,81.4)  {\tc{magentapic}{ $s_E=q_E$}}
    \put (16.3,79.35)  {\tc{magentapic}{ $s_{\overline{E}}=q_{\overline{E}}$}}

    \put (35,48)  {\tc{magentapic}{$s_E$}}
    \put (16.4,47.8)  {\tc{magentapic}{$s_{\overline{E}}=q_{\overline{E}}$}}

    \put (33.5,15)  {\tc{magentapic}{ $s_E$}}
    \put (8.5,17.5)  {\tc{magentapic}{ $s_{\overline{E}}$}}


    \put (24,86)  {\tc{brownpic}{ $\Ti_{T_{\overline{E}}}$}}
    \put (16,76)  {\tc{brownpic}{$\Ti_{T_E}$}}

    \put (24,55)  {\tc{brownpic}{$\Ti_{T_{\overline{E}}}$}}
    \put (23.5,47)  {\tc{brownpic}{$\Ti_{T_E}$}}

    \put (18,18)  {\tc{brownpic}{$\Ti_{T_{\overline{E}}}$}}
    \put (24,14)  {\tc{brownpic}{$\Ti_{T_E}$}}
    
   \end{overpic}}	
   \caption{The $\T$--shifting condition along $E\in\caleor_P$. 
   From top to bottom, the cases $\ttt_E,\ttt_{\overline{E}}\leq 1$, $\ttt_{\overline{E}}\leq 1\leq \ttt_E$, and $1\leq \ttt_E,\ttt_{\overline{E}}$.
   The orange segments have  triangular length equal to $-\log(\ttc_E\ttc_{\overline{E}})$, $-\log (\ttt_E\ttc_E\ttc_{\overline{E}})$ and $-\log (\ttt_E\ttt_{\overline{E}}\ttc_E\ttc_{\overline{E}})$ respectively.}
   \label{fig: Tshifting}
\end{figure}
In this case we let $q_E\in\{r_E,s_E\}$ and $q_{\overline{E}}\in\{r_{\overline{E}},s_{\overline{E}}\}$ be the points realizing the $\min$ above, so that \begin{equation}\label{eq: triangular distance qE and overlineqE}
    d_{\T}(q_E,q_{\overline{E}})=\log(\tts_E)=\log C(l_{i_E},q_{\overline{E}},q_E,l_{k_E}) > 0.
\end{equation}

The $\T$--shifting condition is inspired by Parreau's \cite{ParreauInv} \textit{left-shifting} condition, see Subsection \ref{subsec: parreau} for detailed comparisons. 
We emphasize however that, in contrast to Parreau's notion, our $\T$--shifting condition does not depend on the full information of the $n$--tuple. 
More precisely, it is a property about the triangle invariants $\ttt_E$ and $ \ttt_{\overline{E}}$ and the product of the edge invariants (it does not involve the individual values $\ttc_E$ and $\ttc_{\overline{E}}$).
This gives us more flexibility to build, from a starting $\T$--shifting $n$--tuple, a large dimensional family of $\T$--shifting $n$--tuples. 
For example, we have the following.

\begin{eg}\label{ex: stretching Tshifting parameters}
Let $(F_1,\dots,F_n)$ and $(F_1',\dots, F_n')$ be positive $n$--tuples with respective parameters $(\ttt,\ttc)\in(\rr_{>0})^{\calt_P}\times(\rr_{<0})^{\caleor_P}$ and $(\ttt',\ttc')\in(\rr_{>0})^{\calt_P}\times(\rr_{<0})^{\caleor_P}$.
Fix some $K>0$ and suppose that for every $E\in\caleor_P$ one has $$\ttt'_E=\left(\ttt_E\right)^K \tn{ and } \ttc'_E\ttc'_{\overline{E}}=\left(\ttc_E\ttc_{\overline{E}} \right)^K.$$ \noindent Then $(F_1,\dots,F_n)$ is $\T$--shifting if and only if $(F_1',\dots, F_n')$ is.

\end{eg}

\subsection{Estimates in terms of the invariants}\label{subsec: estimate triangular}

We now show that for a $\T$--shifting $n$--tuple $(F_1,\dots,F_n)$ the triangular distance between distinguished points on $\Poli=\Poli_{(F_1\dots,F_n)}$ can be explicitly computed out of the triangle and shear invariants, up to an additive universal constant (Proposition \ref{prop: additve estimate triangular distance}).

To state the proposition precisely we need to fix some terminology.
Assume that $n\geq 5$ (we will also cover the case $n=4$ separately, see Remark \ref{rem: estimate n=4} below).
A sequence $E_1,E_2,\dots,E_M\in\cale_P$, with $2\leq M\leq n-3$, is called \textit{admissible} if \begin{itemize}
    \item the unoriented edges $E_1,\dots,E_M$ are pairwise different, 
    \item for each $m=1,\dots, M-1$ the edges $E_m$ and $E_{m+1}$ are sides of a triangle $T_m\in\calt_P$, and 
    \item the triangles $T_1,\dots, T_{M-1}$ are pairwise different.
\end{itemize}\noindent See Figure \ref{fig: additve estimate triangular distance}.
\begin{figure}[htbp] 
\scalebox{1}{%
\begin{overpic}[scale=1, width=0.65\textwidth, tics=2]{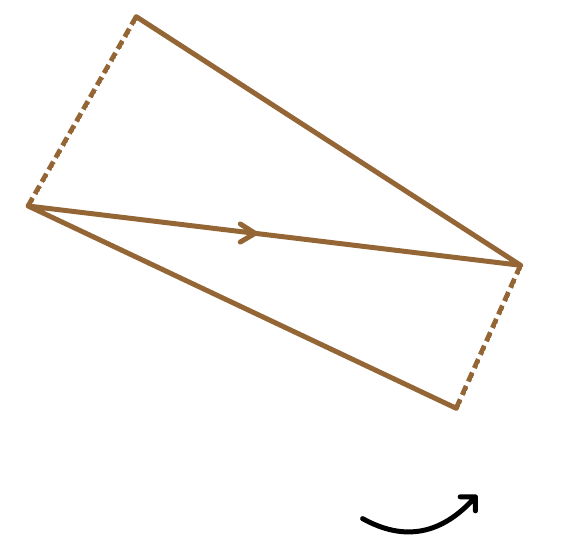}

    \put (30,70)  {\tc{brownpic}{$T_{m}$}}

    \put (60,40)  {\tc{brownpic}{$T_{m-1}$}}

    \put (30,40)  {\tc{brownpic}{$E_{m-1}$}}
    \put (46,80)  {\tc{brownpic}{$E_{m+1}$}}
    \put (46,56)  {\tc{brownpic}{$E_{m}$}}
    
   \end{overpic}}	
   \caption{The local picture of an admissible sequence $E_1,\dots,E_M$. 
   Each $m=1,\dots,M-1$ determines a triangle $T_m\in \calt_P$.
   The triangles $T_{m-1}$ and $T_{m}$ have common side $E_{m}$ but are different.
   The edge $E_{m}$ is naturally oriented (the curved arrow indicates the cyclic order on vertices of $P$).}
   \label{fig: additve estimate triangular distance}
\end{figure}
We may also consider the unique triangle $T_0\in\calt_P$ (resp. $T_M\in\calt_P$) which is different from $T_1$ (resp. $T_{M-1}$) and has edge $E_1$ (resp. $E_M$).
We note that for $m=1,\dots,M$ the edge $E_m$ has a natural orientation: it is the unique one so that $T_{m-1}=T_{E_m}$ and $T_m=T_{\overline{E}_m}$, c.f. Subsection \ref{subsec: positive tuples}.
We fix these orientations throughout.

Fix a $\T$--shifting $n$--tuple $(F_1,\dots,F_n)$ and an admissible sequence $E_1,\dots,E_M$.
For $m=1,\dots,M$ we let $\ttt_m>0$ be the triangle invariant (recall \eqref{eq:shear_invariant}) associated to $T_m$ and $\tts_m > 1$ be the shear invariant associated to $E_m$.
The orientation of $E_1$ and $E_M$ allows us to select the distinguished  points $q_{E_1}$ and $q_{\overline{E}_M}$ as in (\ref{eq: triangular distance qE and overlineqE}).
The reminder of this subsection is dedicated to the proof of the following.

\begin{prop}\label{prop: additve estimate triangular distance}
    There exists a constant $D>0$ so that for every $n\geq 5$ the following holds.
    Let $(F_1,\dots,F_n)$ be a $\T$--shifting $n$--tuple of flags and $d_{\T}(\cdot,\cdot)$ be the corresponding triangular distance on $\Poli=\Poli_{(F_1,\dots,F_n)}$.
    Then for every $2\leq M\leq n-3$ and every admissible sequence $E_1,\dots,E_M$ one has 
    $$\left\vert d_{\T}(q_{E_1},q_{\overline{E}_M}) - \displaystyle\sum_{m=1}^{M-1}\vert \log (\ttt_m)\vert - \displaystyle\sum_{m=1}^{M} \log (\tts_m) \right\vert <D\cdot M.$$ 
\end{prop}

\begin{rem}\label{rem: estimate n=4}

We remark that we could have allowed the case $M=1$ (and also $n=4$) in the definition of an admissible sequence, and an analogue of Proposition \ref{prop: additve estimate triangular distance} holds for that case but no triangle invariants nor error terms are involved (the conclusion is precisely given by (\ref{eq: triangular distance qE and overlineqE})).
We have decided not to include this degenerate case in the statement of Proposition \ref{prop: additve estimate triangular distance} so that we get a cleaner conclusion.
\end{rem}

The proof of Proposition \ref{prop: additve estimate triangular distance} goes by induction on $M\geq 2$.
We now state a slightly more general form of this result which is convenient for implementing that proof (recall the notations from Subsection \ref{subsec: positive tuples} and (\ref{eq: triangular distance qE and overlineqE})).
We refer to Figure \ref{fig: additve estimate triangular distance case M2} for a picture in the case $M=2$.

\begin{prop}\label{prop: additve estimate triangular distance refined}
    There exists a constant $D>0$ so that for every $n\geq 5$ the following holds.
    Let $(F_1,\dots,F_n)$ be a $\T$--shifting $n$--tuple of flags and $d_{\T}(\cdot,\cdot)$ be the corresponding triangular distance on $\Poli=\Poli_{(F_1,\dots,F_n)}$.
    Let $2\leq M\leq n-3$ and fix some admissible sequence $E_1,\dots,E_M\in\cale_P$.
    Let $q_1\in p_{i_{E_1}}\oplus p_{k_{E_1}}$ and $\overline{q}_M\in p_{i_{E_M}}\oplus p_{k_{E_M}}$ be arbitrary points so that 
    \begin{equation}\label{eq: additve estimate triangular distance refined}
        C(l_{i_{E_1}},q_{\overline{E}_1},q_1,l_{k_{E_1}})\geq 1 \tn{ and } C(l_{i_{E_M}},\overline{q}_M,q_{E_M},l_{k_{E_M}})\geq 1.
    \end{equation}
    \noindent Moreover, write $\overline{q}_1\coloneqq q_{\overline{E}_1}$, $ q_M\coloneqq q_{E_M}$ and for $m=2,\dots,M-1$ let $$q_m\coloneqq q_{E_m} \tn{ and } \overline{q}_m\coloneqq q_{\overline{E}_m}.$$
    \noindent Then one has $$ \left\vert d_{\T}(q_1,\overline{q}_M) -\displaystyle\sum_{m=1}^{M-1}\vert \log (\ttt_m)\vert - \displaystyle\sum_{m=1}^{M} d_{\T}(q_m,\overline{q}_m) \right\vert < D\cdot M.$$

\end{prop}

Observe that Proposition \ref{prop: additve estimate triangular distance} follows from Proposition \ref{prop: additve estimate triangular distance refined} and (\ref{eq: triangular distance qE and overlineqE}) by taking $q_1=q_{E_1}$ and $\overline{q}_M=q_{\overline{E}_M}$ (which by the $\T$--shifting condition satisfy (\ref{eq: additve estimate triangular distance refined})).

\begin{proof}[Proof of Proposition \ref{prop: additve estimate triangular distance refined}]

On the one hand, by the triangle inequality we have $$d_{\T}(q_1,\overline{q}_M)\leq \displaystyle\sum_{i=1}^{M}d_{\T}(q_m,\overline{q}_m)+ \displaystyle\sum_{i=1}^{M-1}d_{\T}(\overline{q}_m,q_{m+1}).$$
\noindent Let $D$ be the constant provided by Lemma \ref{lem: bound triangular pencil}.
By Lemma \ref{lem: Teq with cross ratios} we have $$d_{\T}(\overline{q}_m,q_{m+1})\leq \vert\log(\ttt_m)\vert+2D$$
\noindent for all $m=1,\dots,M-1$.
Thus, the upper bound in the statement holds without the $\T$--shifting assumption. 

We now turn to the complementary lower bound, which goes by induction.
The key step is the case $M=2$, which we isolate for convenience (c.f. Figure \ref{fig: additve estimate triangular distance case M2}).

\begin{lem}\label{lem: additve estimate triangular distance refined case M2}
    There exists a constant $D>0$ so that for every $n\geq 5$ the following holds.
    Let $(F_1,\dots,F_n)$ be a $\T$--shifting $n$--tuple of flags and $d_{\T}(\cdot,\cdot)$ be the corresponding triangular distance on $\Poli=\Poli_{(F_1,\dots,F_n)}$.
    Fix some admissible sequence $E_1,E_2\in\cale_P$ and let $q_1\in p_{i_{E_1}}\oplus p_{k_{E_1}}$ and $\overline{q}_2\in p_{i_{E_2}}\oplus p_{k_{E_2}}$ be arbitrary points so that \begin{equation}\label{eq: additve estimate triangular distance refined case M2}
        C(l_{i_{E_1}},q_{\overline{E}_1},q_1,l_{k_{E_1}})\geq 1 \tn{ and } C(l_{i_{E_2}},\overline{q}_2,q_{E_2},l_{k_{E_2}})\geq 1.
    \end{equation}
    \noindent Moreover, write $\overline{q}_1\coloneqq q_{\overline{E}_1}$ and $ q_2\coloneqq q_{E_2}$.
    \noindent Then one has $$ d_{\T}(q_1,\overline{q}_2) -\vert \log (\ttt_1)\vert - d_{\T}(q_1,\overline{q}_1)- d_{\T}(q_2,\overline{q}_2)> -D.$$ 
\end{lem}

\begin{figure}[htbp] 
\scalebox{1}{%
\begin{overpic}[scale=1, width=0.65\textwidth, tics=2]{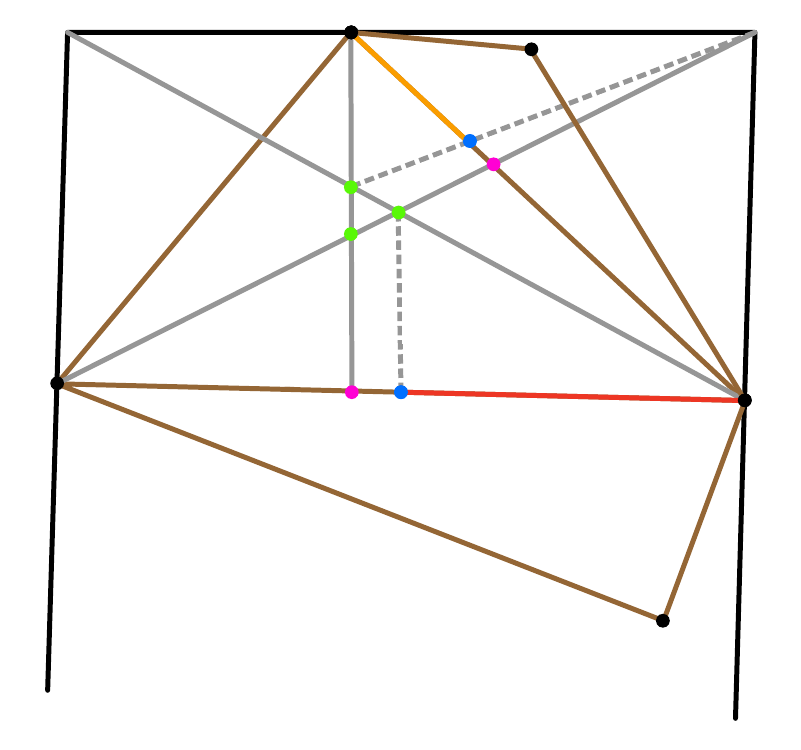}

    \put (70,30)  {\tc{brownpic}{$\Ti_{T_0}$}}
    \put (30,65)  {\tc{brownpic}{$\Ti_{T_1}$}}
    \put (71,65)  {\tc{brownpic}{$\Ti_{T_2}$}}

    \put (49,40.5)  {\tc{bluepic}{$\overline{q}_1$}}
    \put (57,78)  {\tc{bluepic}{$q_2$}}

    \put (-15,46)  {$p_{i_{E_1}}=p_{j_{E_2}}$}
    \put (-13,70)  {$l_{i_{E_1}}=l_{j_{E_2}}$}

    \put (79.5,13)  {$p_{j_{E_1}}$}

    \put (95,43.5)  {$p_{k_{E_1}}=p_{k_{E_2}}$}
    \put (95,70)  {$l_{k_{E_1}}=l_{k_{E_2}}$}

    \put (40,93)  {$p_{j_{\overline{E}_1}}=p_{i_{E_2}}$}
    \put (10,93)  {$l_{j_{\overline{E}_1}}=l_{i_{E_2}}$}

    \put (68,86.5)  {$p_{j_{\overline{E}_2}}$}
    
   \end{overpic}}	
   \caption{A picture for Lemma \ref{lem: additve estimate triangular distance refined case M2} in the case $\ttt_1>1$.
   The admissible sequence is $E_1\coloneqq T_0\cap T_1$, $E_2\coloneqq T_1\cap T_2$.
   These edges are naturally oriented and this defines labels on the relevant vertices. 
   The figure depicts the case $k_{E_1}=k_{E_2}$ (the complementary case would be $i_{E_1}=i_{E_2}$).
   By (\ref{eq: additve estimate triangular distance refined case M2}), $q_1$ is any point in the red segment, while $\overline{q}_2$ is any point in the orange segment.}
   \label{fig: additve estimate triangular distance case M2}
\end{figure}

We conclude the proof of Proposition \ref{prop: additve estimate triangular distance refined} assuming that Lemma \ref{lem: additve estimate triangular distance refined case M2} holds.  
Hence, we let $M>2$ and suppose by induction that the statement of Proposition \ref{prop: additve estimate triangular distance refined} holds for all admissible sequences $E_1,\dots,E_{M-1}$.

Let $\alpha\subset\Poli$ be a piecewise smooth curve connecting $q_1$ with $\overline{q}_M$ and so that $$\vert d_{\T}(q_1,\overline{q}_M)-\ell_{\T}(\alpha) \vert <1.$$
\noindent 
By Lemma \ref{lem: triangular distance separates and extends} and Remark \ref{rem: minimizing triangular curves} we may assume that $\alpha$ decomposes as $\alpha=\alpha_1\star\alpha_2$, where $\alpha_2\subset\Ti_{T_{M-1}}$ is a projective segment connecting some point $q\in \Ti_{T_{M-1}}\cap \Ti_{T_{M-2}}$ with $\overline{q}_M$ and $\alpha_1\subset \Ti_{T_1}\cup\dots\cup\Ti_{T_{M-2}}$ connects $q_1$ with $q$.
Hence we have \begin{equation}\label{eq: additive estimate intermediate decomposition}
       d_{\T}(q_1,\overline{q}_M)-d_{\T}(q_1,q)-d_{\T}(q,\overline{q}_M) >-1.  
    \end{equation}
\noindent The proof is divided in two cases, according to the relative position between $q$ and $\overline{q}_{M-1}$.

Suppose first that $C(l_{i_{E_{M-1}}},q,\overline{q}_{M-1},l_{k_{E_{M-1}}})\geq 1$.
By the Cocycle Property and (\ref{eq: triangular distance qE and overlineqE}) we then have $$C(l_{i_{E_{M-1}}},q,q_{M-1},l_{k_{E_{M-1}}})\geq \tts_{M-1},$$
\noindent which is bigger than $ 1$ by the $\T$--shifting hypothesis.
Hence, $q$ satisfies (\ref{eq: additve estimate triangular distance refined}) with $M$ replaced by $M-1$ and by the induction assumption together with (\ref{eq: additive estimate intermediate decomposition}) we deduce that $$   d_{\T}(q_1,\overline{q}_M) - \displaystyle\sum_{m=1}^{M-2}\vert \log (\ttt_m)\vert - \displaystyle\sum_{m=1}^{M-2} d_{\T}(q_m,\overline{q}_m) - d_{\T}(q_{M-1},q)  -d_{\T}(q,\overline{q}_M)  $$
\noindent is larger than $ -D\cdot (M-1)-1$.
By the triangle inequality, 
$$   d_{\T}(q_1,\overline{q}_M) - \displaystyle\sum_{m=1}^{M-2}\vert \log (\ttt_m)\vert - \displaystyle\sum_{m=1}^{M-2} d_{\T}(q_m,\overline{q}_m) - d_{\T}(q_{M-1},\overline{q}_M) > -D\cdot(M-1)-1. $$

On the other hand, by the $\T$--shifting condition we may apply Lemma \ref{lem: additve estimate triangular distance refined case M2} to the admissible sequence $E_{M-1},E_M$ and the points  $q_{M-1}$ and $\overline{q}_M$ and obtain $$ d_{\T}(q_{M-1},\overline{q}_M)- \vert \log(\ttt_{M-1})\vert -d_{\T}(q_{M-1},\overline{q}_{M-1})-d_{\T}(q_{M},\overline{q}_M)>-D.$$
\noindent Adding up the last two inequalities we obtain the conclusion in this case.

It remains to treat the complementary case.
Namely, we assume that $$C(l_{i_{E_{M-1}}},q,\overline{q}_{M-1},l_{k_{E_{M-1}}})< 1.$$
\noindent By the $\T$--shifting condition, we may apply Lemma \ref{lem: additve estimate triangular distance refined case M2} to the admissible sequence $E_{M-1},E_M$ and the points $q$ and $\overline{q}_M$.
We have $$d_{\T}(q,\overline{q}_M)-\vert\log(\ttt_{M-1})\vert-d_{\T}(q,\overline{q}_{M-1})-d_{\T}(q_M,\overline{q}_M)>-D.$$
\noindent Adding up (\ref{eq: additive estimate intermediate decomposition}) we deduce $$ d_{\T}(q_1,\overline{q}_M)-d_{\T}(q_1,q)- \vert\log(\ttt_{M-1})\vert-d_{\T}(q,\overline{q}_{M-1})-d_{\T}(q_M,\overline{q}_M) >-D-1.$$
\noindent Hence $$ d_{\T}(q_1,\overline{q}_M)- d_{\T}(q_1,\overline{q}_{M-1})-\vert\log(\ttt_{M-1})\vert -
d_{\T}(q_M,\overline{q}_M) >-D-1.$$
\noindent As $(F_1,\dots,F_n)$ is $\T$--shifting we may apply the induction hypothesis to the term $d_{\T}(q_1,\overline{q}_{M-1})$ to finish the proof.
\end{proof}

Let us now prove Lemma \ref{lem: additve estimate triangular distance refined case M2}.

\begin{proof}[Proof of Lemma \ref{lem: additve estimate triangular distance refined case M2}]

Recall that the edge $E_1$ is  oriented, so we get labels $a_{i_{E_1}}$ and $a_{k_{E_1}}$ for its vertices.
There are two possibilities, namely: 

\begin{itemize}
    \item either the edge $E_2$ connects $a_{j_{\overline{E}_1}}$ with $a_{k_{E_1}}$, or
    \item the edge $E_2$ connects $a_{i_{E_1}}$ with $a_{j_{\overline{E}_1}}$.
\end{itemize}
\noindent We only treat explicitly the first case, the second one follows similarly. 
Hence, we assume the $E_2$ connects $a_{j_{\overline{E}_1}}$ with $a_{k_{E_1}}$.
Using the natural orientation of $E_2$ we have $$k_{E_2}=k_{E_1}, \tn{ } i_{E_2}=j_{\overline{E}_1}, \tn{ and } j_{E_2}=i_{E_1}$$
\noindent (c.f. Figure \ref{fig: additve estimate triangular distance case M2}).
We deduce $$a_{i_{E_1}}=a_{j_{E_2}}<a_{j_{E_1}}<a_{k_{E_1}}=a_{k_{E_2}}<a_{j_{\overline{E}_2}}<a_{j_{\overline{E}_1}}=a_{i_{E_2}}<a_{i_{E_1}}.$$
\noindent We ease notations  let $$F_1\coloneqq F_{k_{E_1}}, \tn{ }F_2 \coloneqq F_{j_{\overline{E}_1}} \tn{ and }F_3 \coloneqq F_{i_{E_1}}.$$
\noindent We then have $q_1,\overline{q}_1\in p_3\oplus p_1\tn{ and } q_2,\overline{q}_2\in p_2 \oplus p_1$.

Now, by definition of the Fock-Goncharov-like parametrization (\ref{eq: parametrization n tuples}) we have $\ttt_1=T(F_1,F_2,F_3)$.
We assume that $\ttt_1\geq 1$, the other case is similar. 
This assumption allows us to compute $\overline{q}_1$ and $q_2$.
Indeed, recalling the notations from Subsection \ref{subsec: triple ratio}, (\ref{eq: shear quadruple}) and (\ref{eq: triangular distance qE and overlineqE}), we know that $\overline{q}_1\in\{r_2,s_2\}$ is the unique point in $p_1\oplus p_3$ satisfying $$C(l_3,\overline{q}_1,q_{E_1},l_1)=\tts_1=\displaystyle\min_{\sigma'=r,s} C(l_3,\sigma'_2,q_{E_1},l_1).$$
\noindent Now, as by the Cocycle Property and Lemma \ref{lem: t.r. with c.r.} we have $$C(l_3,s_2,q_{E_1},l_1)=\ttt_1\cdot C(l_3,r_2,q_{E_1},l_1)\geq C(l_3,r_2,q_{E_1},l_1)$$
\noindent we deduce $\overline{q}_1=r_2$.
Similarly, we deduce that $q_2=r_3$.
If we let $D$ be the constant provided by Lemma \ref{lem: bound triangular pencil}, it then suffices to show \begin{equation}\label{eq: lemm additive bound case M2}
    d_{\T}(q_1,\overline{q}_2) - \log (\ttt_1) - d_{\T}(q_1,r_2)- d_{\T}(r_3,\overline{q}_2)> -4D.
\end{equation}

Now, recall that we assume $\Text=\Text_{(F_1,F_2,F_3)}\subset\bba$, and by Lemma \ref{lem: nested triangles} we have $r_2\in[p_3,p_1]$ and $p_3\notin[r_2,p_1]$.
By Remark \ref{rem: segment ciclicality and nesting} and (\ref{eq: additve estimate triangular distance refined case M2}) we have $q_1\in[r_2,p_1]$ and similarly $  \overline{q}_2\in [r_3,p_2]$.

\begin{cla}\label{cla: additve estimate triangular distance refined case M2}
    
The intersection point $q\coloneqq(q_1\oplus \overline{q}_2) \cap (r_3\oplus t_3)$ satisfies $q\in[q_1,\overline{q}_2]\cap[r_3,t_3]$.
\end{cla}

We postpone the proof of Claim \ref{cla: additve estimate triangular distance refined case M2} until the end of the subsection and proceed assuming it holds. 
In particular, as projective segments inside $\Ti_{T_1}$ are geodesics for $d_{\T}(\cdot,\cdot)$ (c.f. Lemma \ref{lem: triangular distance separates and extends}), we have $$d_{\T}(q_1,\overline{q}_2)=d_{\T}(q_1,q)+d_{\T}(q,\overline{q}_2)$$ \noindent and we will bound separately both terms of the right hand side to obtain (\ref{eq: lemm additive bound case M2}).

Now, on the one hand we have $$d_{\T}(q,\overline{q}_2)\geq d_{\T}(\overline{q}_2,r_3)-d_{\T}(r_3,q).$$
\noindent As $q\in[r_3,t_3]$, by Lemma \ref{lem: bound triangular pencil} this is greater or equal than $d_{\T}(\overline{q}_2,r_3)-D$.

On the other hand, by a similar argument we have $$d_{\T}(q_1,q)\geq d_{\T}(q_1,t_3)-D.$$

\begin{cla}\label{cla: additve estimate triangular distance refined case M2 bis}
The intersection point $q'\coloneqq ( q_1\oplus t_3)\cap(r_2\oplus t_2)$ satisfies $q'\in [q_1,t_3]\cap[r_2, t_2]$.
\end{cla}

See below for a proof of Claim \ref{cla: additve estimate triangular distance refined case M2 bis}.
Assuming it holds, we deduce that $d_{\T}(q_1,t_3)=d_{\T}(q_1,q')+d_{\T}(q',t_3)$ and this is greater or equal than $$ d_{\T}(q_1,r_2)-d_{\T}(r_2,q')+d_{\T}(t_3,t_2)-d_{\T}(t_2,q').$$ \noindent As $q'\in[r_2,t_2]$, the proof of (\ref{eq: lemm additive bound case M2}) finishes by applying Lemma \ref{lem: Teq with cross ratios} and Lemma \ref{lem: bound triangular pencil}.
\end{proof}

It still remains to prove Claims \ref{cla: additve estimate triangular distance refined case M2} and \ref{cla: additve estimate triangular distance refined case M2 bis}.

\begin{proof}[Proof of Claims \ref{cla: additve estimate triangular distance refined case M2} and \ref{cla: additve estimate triangular distance refined case M2 bis}]

We begin by proving Claim \ref{cla: additve estimate triangular distance refined case M2}, which is obvious if $\overline{q}_2=r_3$, so we assume that this is not the case. 
By Remark \ref{rem: triangles and subtriangles} we may consider the sub--triangle of $\Ti=\Ti_{(F_1,F_2,F_3)}$ spanned by the vertices $p_2$, $s_2$ and $r_3$.
The point $\overline{q}_2$ belongs then to its side $[p_2,r_3]$ and by Lemma \ref{lem: nested triangles} the point $t_3$ belongs to the side $[p_2,s_2]$.
By Remark \ref{rem: intersections triangle}, the projective segments $[s_2,\overline{q}_2]$ and $[r_3,t_3]$ intersect in a point $\widehat{q}$.
It then suffices to show that $[\widehat{q},r_3]$ and $[q_1,\overline{q_2}]$ intersect. 
Now, to show this we may work on the sub--triangle of $\Ti$ spanned by $\overline{q}_2$, $s_2$ and $p_1$.
It contains $\widehat{q}$ on one side and $r_3$ on other side, so by Remark \ref{rem: intersections triangle} it suffices to show that $q_1$ belongs to the side $[p_1,s_2]$. 
But we know that $q_1\in[r_2,p_1]$, and by Remark \ref{rem: segment ciclicality and nesting} and Lemma \ref{lem: t.r. with c.r.} we deduce $q_1\in[p_1,s_2]$ (recall that we are assuming $\ttt_1\geq 1$).

It remains to prove Claim \ref{cla: additve estimate triangular distance refined case M2 bis}, but the proof follows similarly by applying Remark \ref{rem: intersections triangle} to the triangle spanned by $t_3$, $ r_2$ and $p_1$ (note that, as $\ttt_1\geq 1$, by Remark \ref{rem: segment ciclicality and nesting} and Lemma \ref{lem: t.r. with c.r.} we have that $t_2$ lies on the side between $p_1$ and $t_3$).
\end{proof}

\section{Ortho--geodesic stretch map}\label{sec: stretch map}

In this section we fix two positive $n$--tuples $(F_1,\dots,F_n)$ and $(F_1',\dots,F_n')$ for some $n\geq 3$ with respective Fock-Goncharov-like parameters $(\ttt,\ttc)$ and $(\ttt',\ttc')$. 
For each $E\in\caleor_P$ let $(\ttt_E,\ttt_{\overline{E}},\ttc_E,\ttc_{\overline{E}})$ and $(\ttt_E',\ttt_{\overline{E}}',\ttc_E',\ttc_{\overline{E}}')$ be the corresponding sub--parameters, as in Subsection \ref{subsec: positive tuples}.
We will assume moreover that there exists some $K>1$ so that for all $E\in\caleor_P$ one has $$
    \ttt_E'=\ttt_E^K \tn{ and } \ttc_E'\ttc_{\overline{E}}'=(\ttc_E\ttc_{\overline{E}})^K.$$
\noindent In this situation we will say that $(F_1',\dots,F_n')$ is a \textit{$\T$--stretch of size $K$} of $(F_1,\dots,F_n)$.

If $(F_1',\dots,F_n')$ is a $\T$--stretch of size $K$ of $(F_1,\dots,F_n)$ we may consider the ideal polygons $\Poli=\Poli_{(F_1,\dots,F_n)}$ and $(\Poli)'\coloneqq\Poli_{(F_1',\dots,F_n')}$ of Lemma \ref{lem: nested polygons}.
The goal of this section is to construct a natural geometric homeomorphism $$
    \psi^{(\ttt,\ttc)}=\psi_K^{(\ttt,\ttc)}:\Poli\to(\Poli)',$$ \noindent which we call the associated \textit{ortho--geodesic stretch map}. 
We will show in Theorem \ref{thm: Klip} that it is $K$-Lipschitz for the underlying triangular distances.

\subsection{The case of positive triples}\label{subsec: stretch map on triples}

The construction is inductive, we begin with the case $n=3$.
We fix some $K>1$ and two positive triples with $$T(F_1',F_2',F_3')=T(F_1,F_2,F_3)^K.$$
\noindent Let $\Text\coloneqq \Text_{(F_1,F_2,F_3)}$ and $(\Text)'\coloneqq \Text_{(F_1',F_2',F_3')}$. Then let $\Vert\cdot\Vert_{\Text}$ and $\Vert\cdot\Vert_{(\Text)'}$ be their respective Hilbert norms as in Subsection \ref{subsec: hilbert geometry}. 
We also denote $$\Ti\coloneqq \Ti_{(F_1,F_2,F_3)} \tn{ and } (\Ti)'\coloneqq \Ti_{(F_1',F_2',F_3')}.$$
\noindent Each of these ideal triangles is further decomposed into four sub--triangles as in (\ref{eq: Tid decompositon}) (the decomposition associated to $(F_1',F_2',F_3')$ is denoted $(\Ti)'=(\Teq)'\cup\T_{123}'\cup\T_{231}'\cup\T_{312}'$).
Our map will take each of these sub--triangles for $(F_1,F_2,F_3)$ to the corresponding one for $(F_1',F_2',F_3')$.

\subsubsection{The rigid pieces}\label{subsub: homeo rigid piece}

We start by defining our homeomorphism in restriction to each $\T_{i(i+1)(i+2)}$, 
for $i=1,2,3$.
In the next subsection we explain how to define it on $\Teq$, and then explain how to glue these pieces together. 

Now, let $g_i\in\g$ be the unique element so that $g_i\cdot (F_i,F_{i+2},l_{i+1})=(F_i',F_{i+2}',l_{i+1}')$ (recall (\ref{eq: rigid configuration})).
By Remark \ref{rem: Tijk intrinsic} we know that $$g_i\cdot \T_{i(i+1)(i+2)}=\T_{i(i+1)(i+2)}'.$$ 
\noindent Hence, to define our homeomorphism $\T_{i(i+1)(i+2)}\to\T_{i(i+1)(i+2)}'$ it suffices to define a self--homeomorphism $\T_{i(i+1)(i+2)}\to\T_{i(i+1)(i+2)}$ and then compose it with $g_i$.

Let us then proceed with the definition of the self--homeomorphism alluded above. 
To simplify notations we denote $p\coloneqq p_i$, $t\coloneqq t_{i+1}$, $p'\coloneqq p_{i+2}$, $p_0\coloneqq l_i\cap l_{i+2}$ and $r\coloneqq r_{i+1}$.
We also set $l\coloneqq l_i$, $l'\coloneqq l_{i+2}$ and $l_0\coloneqq l_{i+1}$. 
Set $\T_0\coloneqq \T_{i(i+1)(i+2)}$, see Figure \ref{fig: foliations}.
\begin{figure}[htbp] 
\scalebox{1.5}{%
\begin{overpic}[scale=1, width=0.65\textwidth, tics=2]{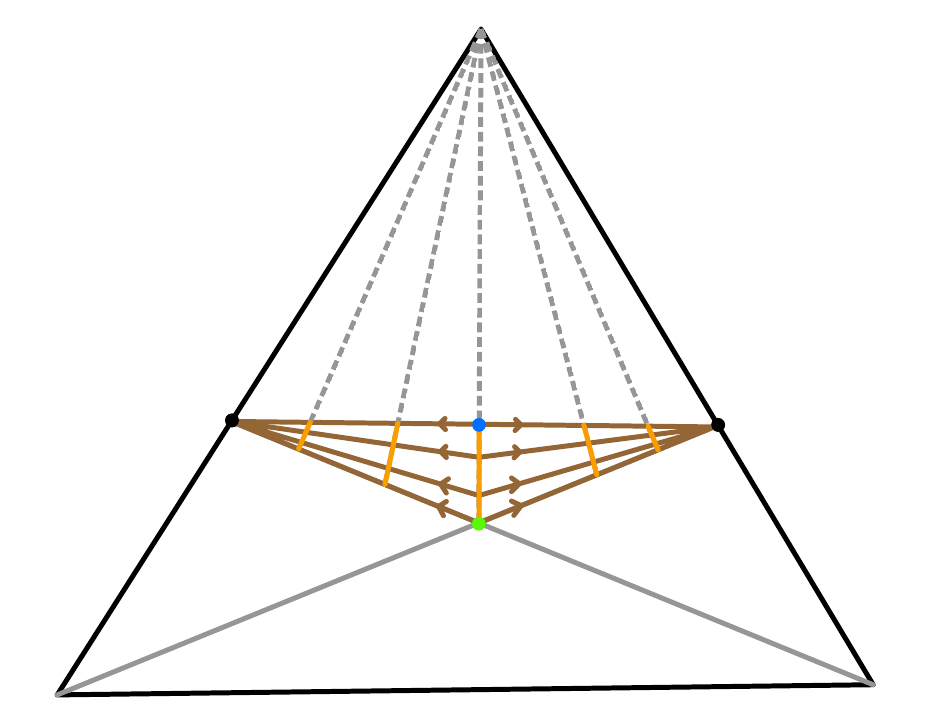}

    \put (21,32)  {\footnotesize$p$}
    
    \put (30,45)  {\footnotesize$l$}

    \put (78,31.5)  {\footnotesize$p'$}
    
    \put (69,45)  {\footnotesize$l'$}

    \put (50,0)  {\footnotesize$l_0$}

    \put (50,75.5)  {\footnotesize$p_0$}

    \put (49,32.5)  {\footnotesize\tc{bluepic}{$r$}}

    \put (50,17)  {\footnotesize\tc{greenpic}{$t$}}

   \end{overpic}}	
   \caption{The rigid piece $\T_0$. The leaves of the ortho--geodesic foliation $\folo$ in orange,  and the leaves of the cusp--geodesic foliation in brown.
   The piece of $\folc$ lying on the left (resp. right) of $\folo_r=[r,t]$ is $\folc^p$ (resp. $\folc^{p'}$).
   Arrows indicate the dynamics of $\psi^0$.}
   \label{fig: foliations}
   \end{figure}
The desired map $\psi^0=\psi^0_K:\T_0\to\T_0$ \noindent will be defined by stretching by a factor of $K$ a specific foliation of $\T_0$, along a specific transverse foliation.

More precisely, on the one hand we consider the \textit{ortho--geodesic foliation} $\folo$ on $\T_0$ induced by restriction of projective lines passing through $p_0$.
The leaf $\folo_q$ through a given $q\in\T_0$ is a projective segment connecting a point in $(p,p')\subset\T_0$ with a point in $(p,t]\cup[t,p')\subset\T_0$.

\begin{rem}
The terminology ``ortho--geodesic" is motivated by the fact that the leaves of $\folo$ are geodesics and the fibers of the nearest point projection (with respect to the triangular distance) onto the projective segment $(p,p')\subset\T_0$.
\end{rem}

On the other hand, for each $q_0\in[t,r]=\folo_r$ we may consider $$\folc_{q_0}\coloneqq (p,q_0]\cup[q_0,p')\subset\T_0.$$
\noindent The foliation $\folc\coloneqq \{\folc_{q_0}\}_{q_0\in[t,r]}$ is called the \textit{cusp--geodesic foliation}. 
The leaf $\folc_q$ through the point $q$ coincides with $\folc_{q_0}$ where $q_0=\folo_r\cap\folc_q$. For convenience, we denote $$\folc_q^{p}\coloneqq (p,q_0]\subset\folc_q \tn{ and } \folc_q^{p'}\coloneqq [q_0,p')\subset\folc_q .$$
By Remarks \ref{rem: segment ciclicality and nesting} and \ref{rem: triangles and subtriangles} we deduce the following useful characterization. 

\begin{rem}\label{rem: pieces of cusp geodesic leaves}
    For every $q\in\T_0$ one has $$\folc_{q}^{p}=\{q'\in\folc_q: C(l,q,q_0,l')\geq 1\} \tn{ and } \folc_{q}^{p'}=\{q'\in\folc_q: C(l,q_0,q,l')\geq 1\}.$$
\end{rem}

We can now define the \textit{ortho--geodesic stretch map} $\psi^0:\T_0\to\T_0$. 
Let $q\in\T_0$ and $q_0$ as above.
We may have either $q\in\folc_{q_0}^{p}$ or $q\in\folc_{q_0}^{p'}$. 
In the first case we define $\psi^0(q)$ to be the unique point in $\folc_{q_0}^{p}$ that satisfies $$d_{\T}(\psi^0(q),q_0)=K\cdot d_{\T}(q,q_0).$$
\noindent As $p\oplus q_0=q\oplus q_0$ intersects the topological boundary of $\Text$ in $p$ and a point of $[l_0\cap l',p_0]$ (c.f. Remark \ref{rem: triangles and subtriangles}), by Remark \ref{rem: pieces of cusp geodesic leaves} the above equality is equivalent to 
\begin{equation}\label{eq: def stretch map rigid}
   C(l,\psi^0(q),q_0,l')=C(l,q,q_0,l')^K.
\end{equation} 
\noindent On the other hand, if instead the point $q$ belongs to $\folc_{q_0}^{p'}$, then $\psi^0(q)$ is defined to be the unique point in $\folc_{q_0}^{p'}$ that satisfies $d_{\T}(\psi^0(q),q_0)=K\cdot d_{\T}(q,q_0)$.

\begin{rem}\label{rem: psi0 differentiable}
    The ortho--geodesic stretch map $\psi^0$ is smooth in the interior of $\T_0$, except possibly at points $q_0\in(r,t)$. 
    Indeed, let $\T_0^p\subset\T_0$ (resp. $\T_0^{p'}\subset\T_0$) be the triangle spanned by $p$ (resp. $p'$), $t$ and $r$. 
    Then $q\mapsto p\oplus q$ and $$ q\mapsto p_0\oplus \psi^0(q)$$ define smooth maps from the interior of $\T_0^p$ to $\rpd$ (note that thanks to (\ref{eq: def stretch map rigid}) the latter is given by $ x\mapsto x^K$ in an appropriate affine chart of the pencil of lines through $p_0$).
    As these two maps are point--wise transverse, then $$q\mapsto  (p\oplus q)\cap (p_0\oplus \psi^0(q))=\psi^0(q)$$ \noindent is also smooth.
    Similarly, $\psi^0$ is smooth in restriction to the interior of $\T_0^{p'}$.
\end{rem}

It is clear that $\psi^0$ preserves the leafs of $\folc$ and restricts to the identity on $\folo_r$.
Moreover we have the following useful properties.

\begin{lem}\label{lem: properties psi0}
    Let $K>1$ and $\psi^0=\psi^0_K:\T_0\to\T_0$ be the associated ortho--geodesic stretch map.  
    For $q,q'\in\T_0$ the following properties hold:
    
    \begin{enumerate}
\item One has $$ C(l, \psi^0(q),\psi^0(q'), l')=   C(l, q,q', l')^K.$$

\item The leaves of $\folc$ are \tn{stretched by} $K$ in the following sense.
Let $q_0\in\folo_r$ and suppose that either $q,q'\in\folc_{q_0}^{p}$ or $q,q'\in\folc_{q_0}^{p'}$.
Then one has $$d_{\T}(\psi^0(q),\psi^0(q'))=K\cdot d_{\T}(q,q').$$

\item The foliation $\folo$ is $\psi^0$--invariant and its leaves are \tn{contracted} by $\psi^0$, in the sense that if $q\notin\folo_r$ and $q'\in\folo_q$ then $$d_{\T}(\psi^0(q),\psi^0(q'))< d_{\T}(q,q').$$

    \end{enumerate}

In particular, if $q\notin\folo_r$ belongs to the interior of $\T_0$, $v\in T_q\folc_q$ and $v'\in T_q\folo_q$ then one has $$\Vert \od_q\psi^0(v)\Vert_{\Text}=K\cdot \Vert v\Vert_{\Text} \tn{ and }  \Vert \od_q\psi^0(v')\Vert_{\Text}\leq \Vert v'\Vert_{\Text}.$$
\end{lem}

\begin{proof}
Lemma \ref{lem: properties c.r.} and (\ref{eq: def stretch map rigid}) directly imply (1) and (2).
Also, for (3) note that by the Projection Property we have $$\folo_q=\{q'\in\T_0: C(l,q,q',l')=1\}.$$ \noindent From (1) we deduce that $\psi^0(q')\in\folo_{\psi^0(q)}$ for all $q'\in\folo_q$.
Moreover, as $K>1$ the contraction property is a direct consequence of Lemma \ref{lem: bound triangular pencil}.
Finally, the stretching and contraction properties of the derivative $\od_q\psi^0$ follow from (2), (3) and (\ref{eq: finsler norm and derivative of distance}).
\end{proof}

We emphasize that the contraction property in (3) breaks down if we do not assume $K>1$, so our construction is really asymmetric. 

\subsubsection{The equilateral piece}\label{subsub: homeo eq piece}
We let $\ttt\coloneqq T(F_1,F_2,F_3)$, so that we have $\ttt^K =T(F_1',F_2',F_3')$.
We now introduce a natural homeomorphism $$\psieq=\psieq_K: \Teq\to(\Teq)'.$$
\noindent For simplicity we assume $\ttt \geq 1$, so that the construction is crucially based on Lemma \ref{lem: Teq with cross ratios} (the case $\ttt<1$ is analogous, c.f. Remark \ref{rem: Teq when t leq 1}).

Let $t_1'$, $t_2'$ and $t_3'$ be the distinguished points for the triple $(F_1',F_2',F_3')$.
By Lemmas \ref{lem: properties c.r.} and \ref{lem: nested triangles}, for every $q\in\Teq$ there is a unique $\psieq(q)\in(\Text)'$ so that $$C(l_1',t_2',\psieq(q),l_2')=C(l_1,t_2,q,l_2)^K \tn{ and } C(l_2',t_3',\psieq(q),l_3')=C(l_2,t_3,q,l_3)^K.$$

\begin{rem}\label{rem: psieq differentiable}
    As in Remark \ref{rem: psi0 differentiable}, the map $\psieq$ is smooth in the interior of $\Teq$. 
\end{rem}

We also have the following.

\begin{cor}\label{cor: properties psieq}
 Let $(F_1',F_2',F_3')$ be a $\T$--stretch of size $K$ of $(F_1,F_2,F_3)$.
 For $q,q'\in\Teq$ the following properties hold:

\begin{enumerate}
    \item For all $i=1,2,3$ one has $$C(l_i',t_{i+1}',\psieq(q),l_{i+1}')=C(l_i,t_{i+1},q,l_{i+1})^K.$$ \noindent In particular, $\psieq(q)\in(\Teq)'$.
  \item For all $i=1,2,3$ one has $$
C(l_i',\psieq(q),\psieq(q'),l_{i+1}')=C(l_i,q,q', l_{i+1})^K.
$$ \noindent In particular, if $l\in\rpd$ is a projective line passing through $p_{i(i+1)}$ and intersecting $\Teq$, then $\psieq (l\cap \Teq)$ is a projective segment supported by a line passing through $p_{i(i+1)}'$.
\item Let $l\in\rpd$ be as in \tn{(2)}. If  $q,q'\in l\cap\Teq$ then $$d_{\T}(\psieq(q),\psieq(q'))=K\cdot d_{\T}(q,q').$$ 
\noindent As a consequence, if $q$ belongs to the interior of $\Teq$ and $v\in T_q l$, then $$\Vert \od_q\psieq(v)\Vert_{(\Text)'}=K\cdot\Vert v\Vert_{\Text}.$$
 
\end{enumerate}
 
\end{cor}

\begin{proof}
By our assumption on triple ratios, (1) is a consequence of Lemmas \ref{lem: properties c.r.}, \ref{lem: t.r. with c.r.} and \ref{lem: Teq with cross ratios}.
Now, (2) follows from (1) and Lemma \ref{lem: properties c.r.}.
Finally, (3) follows from (2) and (\ref{eq: finsler norm and derivative of distance}).  
\end{proof}

\subsubsection{Assembling the pieces }\label{subsub: assembling homeo triples}

For each $i=1,2,3$ we let $g_i\in\g$ be the unique element so that $g_i\cdot \T_{i(i+1)(i+2)}=\T_{i(i+1)(i+2)}'$. 
We let $$\psi^i=\psi^i_K:\T_{i(i+1)(i+2)}\to\T_{i(i+1)(i+2)}': \psi^i(q)\coloneqq g_i\cdot \psi^0(q),$$ \noindent where $\psi^0$ is the homeomorphism of Subsection \ref{subsub: homeo rigid piece}.
Let $$\psi^\ttt=\psi^\ttt_K:\Ti\to(\Ti)': \psi^\ttt(q)\coloneqq \left\lbrace\begin{array}{l}
     \psi^1(q) \hspace{0,45cm}\tn{ if } q\in \T_{123}\\
     \psi^2(q) \hspace{0,45cm}\tn{ if } q\in \T_{231}\\
     \psi^3(q) \hspace{0,45cm}\tn{ if } q\in \T_{312}\\
     \psieq(q) \hspace{0,31cm}\tn{ if } q\in \Teq
\end{array},\right.$$ \noindent where $\psieq$ is as in Subsection \ref{subsub: homeo eq piece}.
We observe that $\psi^\ttt$ is well defined: 

\begin{lem}\label{lem: gluing psii and psieq}
    Let $i=1,2,3$.
    Then for all $q\in\T_{i(i+1)(i+2)}\cap \T_{(i+1)(i+2)i}$ and $q'\in\Teq\cap \T_{i(i+1)(i+2)}$ one has $$\psi^i(q)=\psi^{i+1}(q) \tn{ and } \psi^i(q')=\psieq(q').$$
\end{lem}

\begin{proof}

We assume for simplicity that $\ttt \geq  1$ (the case $\ttt<1$ is analogous).
By Lemma \ref{lem: t.r. with c.r.} and Remark \ref{rem: pieces of cusp geodesic leaves} we have $$t_{i+1}\in \folc_{t_{i+2}}^{p_{i}} \subset t_{i+2}\oplus p_{i} \tn{ and }t_{i+1}'\in \folc_{t_{i+2}'}^{p_{i}'}  \subset t_{i+2}'\oplus p_{i}'.$$
\noindent Moreover, by definition $\psi^0(t_{i+1})$ is the unique point in $\folc_{t_{i+2}}^{p_{i}}$ so that $$C(l_{i+1},t_{i+2},\psi^0(t_{i+1}),l_{i})=C(l_{i+1},t_{i+2},t_{i+1},l_{i})^K=\ttt^K.$$
\noindent Thus applying $g_{i+1}$ we have $$C(l_{i+1}',t_{i+2}',\psi^{i+1}(t_{i+1}),l_{i}')=C(l_{i+1}',t_{i+2}',t_{i+1}',l_{i}').$$
\noindent As both $\psi^{i+1}(t_{i+1})$ and $t_{i+1}'$ belong to $t_{i+2}' \oplus p_{i}'$, we conclude that $$\psi^{i+1}(t_{i+1})=t_{i+1}'=\psi^i(t_{i+1}).$$
\noindent Moreover, as by Lemma \ref{lem: properties psi0} both $\psi^{i+1}$ and $\psi^{i}$ stretch $\T_{i(i+1)(i+2)}\cap \T_{(i+1)(i+1)i}$ by $K$, the same reasoning shows that these maps agree on this common part.

Analogously, applying Lemma \ref{lem: properties psi0} and Corollary \ref{cor: properties psieq} we see that $\psi^i$ and $\psieq$ coincide along $\T_{i(i+1)(i+2)}\cap \Teq$.
\end{proof}

\subsection{The case of positive tuples} \label{subsec: stretch map on tuples}

We now turn to the general case, namely we fix some $n\geq 3$ and a $\T$--stretch $(F_1',\dots,F_n')$ of size $K$ of $(F_1,\dots,F_n)\in\flags_n^+$.
For each $E\in\caleor_P$ we have the respective sub--quadruples $(F_{i_E},F_{j_E},F_{k_E},F_{j_{\overline{E}}})$ and $(F_{i_E}',F_{j_E}',F_{k_E}',F_{j_{\overline{E}}}')$ as in Subsection \ref{subsec: positive tuples}. Subsection \ref{subsec: stretch map on triples} yields maps $$\psi^E=\psi^E_K:\Ti_{(F_{i_E},F_{j_E},F_{k_E})}\to\Ti_{(F_{i_E}',F_{j_E}',F_{k_E}')}$$ \noindent and $$\psi^{\overline{E}}=\psi^{\overline{E}}_K:\Ti_{(F_{k_E},F_{j_{\overline{E}}},F_{i_E})}\to\Ti_{(F_{k_E}',F_{j_{\overline{E}}}',F_{i_E}')}.$$
\noindent As different ideal triangles of $\Poli$ can meet at most on a common edge (c.f Lemma \ref{lem: nested polygons}), the following shows that these maps glue together to produce the map $\psi^{(\ttt,\ttc)}$ that we were looking for. 

\begin{lem}\label{lem: gluning maps of triples}
For every $E\in\caleor_P$ the homeomorphisms $\psi^E$ and $\psi^{\overline{E}}$ agree on $$[p_{i_E},p_{k_E}]=\Ti_{(F_{i_E},F_{j_E},F_{k_E})}\cap \Ti_{(F_{k_E},F_{j_{\overline{E}}},F_{i_E})}.$$
\end{lem}

The proof of Lemma \ref{lem: gluning maps of triples} is analogous to that of Lemma \ref{lem: gluing psii and psieq}, so we omit it. 
Thanks to Lemma \ref{lem: isometric triangular metrics}, the ortho--geodesic stretch map $\psi^{(\ttt,\ttc)}$  depends only on specific sub--data of $(\ttt,\ttc)$, namely, it depends on the full information of $\ttt$ and on the products $\ttc_E\ttc_{\overline{E}}$ for all $E\in\cale_P$.

\subsection{Stretch properties}

The goal of this section is to prove the following.

\begin{thm}\label{thm: Klip}
 Let $n\geq 3$ and $(F_1',\dots,F_n')$ be a $\T$--stretch of size $K$ of $(F_1,\dots,F_n)\in\flags_n^+$.
 Then the ortho--geodesic stretch map $$\psi^{(\ttt,\ttc)}=\psi^{(\ttt,\ttc)}_K:\Poli\to(\Poli)'$$ \noindent is a Lipschitz map with Lipschitz constant $K$ for the respective triangular distances. 
\end{thm}

To prove Theorem \ref{thm: Klip} we first make several reductions. 
Thanks to Lemma \ref{lem: sufficient condtion for lipshictz triangular} it is enough to assume that 
$n=3$.
Hence, from now on we work with a $\T$--stretch $(F_1',F_2',F_3')$ of size $K$ of some $(F_1,F_2,F_3)$ with triple ratio equal to $\ttt>0$.
We need to show that for a projective segment $\alpha\subset\Ti=\Ti_{(F_1,F_2,F_3)}$ the curve $\psi^\ttt\circ\alpha$ is piecewise smooth and satisfies $$\ell_{(\Text)'}(\psi^\ttt\circ\alpha)\leq K\cdot \ell_{\Text}(\alpha).$$

Lemma \ref{lem: properties psi0} readily controls the stretch of some  special projective segments inside $\Ti$.
Namely, if for some $i=1,2,3$ we have (recall Figure \ref{fig: distinguished points triple}): 
\begin{itemize}
    
    \item either $\alpha$ is contained in some side of $\T_{i(i+1)(i+2)}$, or
    \item the segment $\alpha$ is contained in some $[t_{i},r_i]$,
\end{itemize} \noindent then the desired inequality as satisfied. 
Hence, we assume from now on that $\alpha$ does not satisfy the conditions listed above. 
Remarks \ref{rem: psi0 differentiable} and \ref{rem: psieq differentiable} imply that $\psi^\ttt\circ\alpha$ is piecewise smooth. 
Since the $g_i$'s of Subsection \ref{subsub: homeo rigid piece} induce an isometry for the respective triangular distances,  it suffices to work out separately three cases: the case $\alpha\subset\Teq$, the cases $\alpha\subset\T_0^p$ and $\alpha\subset\T_0^{p'}$ (recall Remark \ref{rem: psi0 differentiable} and Figure \ref{fig: foliations}). 
Moreover, we assume, without loss of generality, that $\alpha$ is contained in the interior of that triangle.

We fix  a triangle $\T\in\{\Teq,\T_0^p,\T_0^{p'}\}$ and a point $q$ in the interior of $\T$. 
We will show that for every unit vector $v\in T_q\T\cong T_q \Ti$ one has $$\Vert \od_q\psi^\ttt(v)\Vert_{(\Text)'}\leq K.$$
\noindent It  suffices to check the above inequality on three specific vectors. 
Namely, for $i=1,2,3$ we let $v_i\in T_q\Ti$ be the unit vector so that the directed Hilbert geodesic $\alpha_{q,v_i}$ of $\Text$ satisfies $\alpha_{q,v_i}(\infty)=p_{(i+1)(i+2)}$.
We will show that \begin{equation}\label{eq: stretch special vectors}
    \Vert \od_q\psi^{\ttt}(v_i)\Vert_{(\Text)'}\leq K
\end{equation} \noindent holds for $i=1,2,3$.
As $\od_q\psi^\ttt$ is linear and the Hilbert norm is symmetric, this will prove the desired bound in view of Example \ref{ex: finsler norm ball for polygons}.

The rest of this subsection is dedicated to the proof of (\ref{eq: stretch special vectors}).
The simplest case is when $q$ is in the interior of $\Teq$, because thanks to Corollary \ref{cor: properties psieq} we have $$\Vert \od _q\psi^\ttt(v_i)\Vert_{(\Text)'}=\Vert \od _q\psieq(v_i)\Vert_{(\Text)'}=K$$ \noindent for all $i=1,2,3$ (in fact, in this case $\od_q\psi^\ttt$ stretches every vector by a factor of $K$).

Let us now turn our attention to the pieces $\T_0^p$ and $\T_0^{p'}$.
Assume that $q$ belongs to the interior of $\T_0^p$ (the other case is symmetric). 
We may restrict our analysis to the homeomorphism $\psi^0:\T_0\to\T_0$ of Subsection \ref{subsub: homeo rigid piece}. 
Recalling the notations introduced in that subsection, if we denote by $v_0$ the unit vector in $T_q\T_0^p\cong T_q\T_0$ pointing towards $p_0$, then thanks to Lemma \ref{lem: properties psi0} we have $$\Vert \od _q\psi^0(v_0)\Vert_{(\Text)'}\leq 1<K.$$
\noindent The proof of (\ref{eq: stretch special vectors}) (and hence of Theorem \ref{thm: Klip}) finishes with the following. 

\begin{prop}\label{prop: stretching special vector psi0}
Let $q$ be a point in the interior of $\T_0^p$, and $v$ and $v'$ be the unit vectors in $T_q\T_0^p\cong T_q\T_0$ so that the Hilbert geodesics $\alpha_{q,v}$ and $\alpha_{q,v'}$ satisfy $$\alpha_{q,v}(\infty)=l\cap l_0 \tn{ and } \alpha_{q,v'}(\infty)=l'\cap l_0.$$
\noindent Then $$\Vert \od _q\psi^0(v)\Vert_{\Text}= K \tn{ and } \Vert \od _q\psi^0(v')\Vert_{\Text}= K.$$
\end{prop}

\begin{proof}
Let $\tn{int}(\T_0^p)$ denote the interior of $\T_0^p$.
For $q\in\tn{int}(\T_0^p)$ let $$g(q)\coloneqq \log C(l,r,q,p\oplus t) \tn{ and } o(q)\coloneqq  \log C(l,q,r,l') .$$ \noindent Then $$(g,o): \tn{int}(\T_0^p)\to (0,\infty)\times (0,\infty)$$ \noindent is a smooth coordinate system on which $\psi^0$ is given by $(g,o)\mapsto\left(g,K\cdot o\right)$.

Now, let $\widetilde{v}$ be either $-v$ or $v'$ and $\widetilde{\alpha}$ be $\alpha_{q,-v}$ or $\alpha_{q,v'}$ accordingly. 
In either case, by (\ref{eq: finsler norm and derivative of distance}) and the Cocycle Property we have \begin{equation}\label{eq: derivative of o}
1=\Vert \widetilde{v}\Vert_{\Text}= \left.\frac{\od}{\od t} \right\vert_{t=0^+}d_{\Text}(q,\widetilde{\alpha}(t))= \left.\frac{\od}{\od t} \right\vert_{t=0^+}o(q)-o(\widetilde{\alpha}(t)).    
\end{equation}
\noindent We claim the following.

\begin{cla}\label{claim: stretch psi0}
For every $K>1$ and every $t>0$ sufficiently small one has $$d_{(\Text)'}(\psi^0(q),\psi^0(\widetilde{\alpha}(t)))=\log C(l,\psi^0(q),\psi^0(\widetilde{\alpha}(t)),l').$$
\end{cla}

Claim \ref{claim: stretch psi0} finishes the proof of the proposition, as we then have $$ d_{(\Text)'}(\psi^0(q),\psi^0(\widetilde{\alpha}(t)))=o(\psi^0(q))-o(\psi^0(\widetilde{\alpha}(t)))=o(\psi^0(q))-K\cdot o(\widetilde{\alpha}(t))$$
\noindent for every small enough $t>0$, and to conclude it suffices to apply (\ref{eq: finsler norm and derivative of distance}) and (\ref{eq: derivative of o}).
\end{proof}

\begin{proof}[Proof of Claim \ref{claim: stretch psi0}]

We only treat explicitly the case $\widetilde{v}=-v$.
We denote $\alpha\coloneqq \alpha_{q,-v}$. For $t>0$ small we have $$C(l,q,\alpha(t),l')=e^t>1.$$
\noindent By Remark \ref{rem: segment ciclicality and nesting} we deduce $q\in[l\cap l_0,\alpha(t)]$.
Let $$\widehat{\alpha}(t)\coloneqq (p\oplus \alpha(t))\cap (\psi^0(q)\oplus(l\cap l_0))\in\tn{int}(\T_0^p).$$
As $\psi^0(q)$ belongs to the sub--triangle of $\Text$ spanned by $\alpha(t)$, $p$ and $l\cap l_0$, then by Remark \ref{rem: triangles and subtriangles} we know $\widehat{\alpha}(t)\in[p,\alpha(t)]$. 
We will prove that \begin{equation} \label{eq: claim stretching psi0}
    \psi^0(\alpha(t))\in [\widehat{\alpha}(t),\alpha(t)],
\end{equation} \noindent and this will be enough to conclude (c.f. Figure \ref{fig: claim stretching psi0}).
\begin{figure}[htbp] 
\scalebox{1}{%
\begin{overpic}[scale=1, width=0.65\textwidth, tics=2]{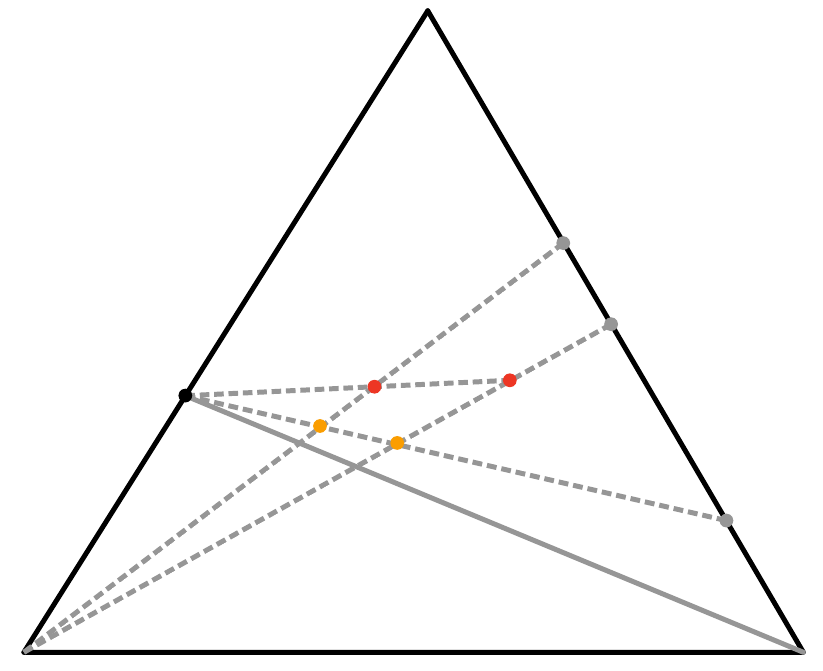}

    \put (19,33)  {$p$}
    \put (40,65)  {$l$}
    \put (62,65)  {$l'$}

    \put (50,-1.5)  {$l_0$}
    \put (50,81)  {$p_0$}

    \put (47,24)  {\tc{orangepic}{$q$}}
    \put (24,27)  {\tc{orangepic}{$\psi^0(q)$}}

    \put (55,36.5)  {\tc{red}{$\alpha(t)$}}
    \put (36.5,36)  {\tc{red}{$\widehat{\alpha}(t)$}}

    \put (69,53)  {\tc{gray}{$\widehat{q}'$}}
    \put (76,42)  {\tc{gray}{$\widehat{q}$}}
    \put (89.5,18)  {\tc{gray}{$\widetilde{q}$}}
   \end{overpic}}	
   \caption{The proof of Claim \ref{claim: stretch psi0}: for any point $q'$ in $[\widehat{\alpha}(t),\alpha(t)]$ the distance $d_{\Text}(\psi^0(q),q')$ is computed as $\log C(l,\psi^0(q),q',l')$.}
   \label{fig: claim stretching psi0}
   \end{figure}

Now, to show that (\ref{eq: claim stretching psi0}) holds observe that $\psi^0(\alpha(t))\in(p,\alpha(t))$ for all $t\geq 0$ (c.f. Remark \ref{rem: pieces of cusp geodesic leaves}).
Hence, by Remark \ref{rem: segment ciclicality and nesting} we need to show $$ C(l,\widehat{\alpha}(t),\psi^0(\alpha(t)),l')\geq 1.$$
\noindent Equivalently (c.f. Lemmas \ref{lem: properties c.r.} and \ref{lem: properties psi0}), $$ C(l,\psi^0(q),\widehat{\alpha}(t),l')\leq C(l,q,\alpha(t),l')^K.$$ 
\noindent Moreover, as $K>1$ it actually suffices to show $$C(l,\psi^0(q),\widehat{\alpha}(t),l')\leq C(l,q,\alpha(t),l')$$ \noindent and this follows from the Nesting Property. 
Indeed, let $$\widehat{q}\coloneqq (q\oplus \alpha(t))\cap l' \tn{ and } \widehat{q}'\coloneqq (\psi^0(q)\oplus \widehat{\alpha}(t))\cap l'. $$
\noindent By Remark \ref{rem: triangles and subtriangles} we have $\widehat{q},\widehat{q}'\in\Text$ and, as $\psi^0(q)\in(p,q)$, we also have $\widehat{q}'\in(p_0,\widehat{q})$.
Hence, Remark \ref{rem: intersections triangle} implies that $p\oplus \widehat{q}'$ does not intersect the segment $[l\cap l_0,\widehat{q}]$.
By the Projection Property and Remark \ref{rem: segment ciclicality and nesting} we deduce $$1<C(l,\psi^0(q),\widehat{\alpha}(t),l')<C(l,q,\alpha(t),l')$$ \noindent thus finishing the proof of (\ref{eq: claim stretching psi0}).
Observe that moreover we have $\widehat{\alpha}(t)\in[\psi^0(q),\widehat{q}']$.

Now that we know that (\ref{eq: claim stretching psi0}) holds we can finish the proof of the claim. 
Indeed, let $\widetilde{q}\coloneqq (p\oplus q)\cap l'$, which belongs to the side $[p_0,l'\cap l_0]$ of $\Text$.
As $q\in(\psi^0(q),\widetilde{q})$, Remark \ref{rem: triangles and subtriangles} applied to the triangle spanned by $l\cap l_0$, $\widehat{q}'$ and $\widetilde{q}$ implies that $\widehat{q}\in(\widehat{q}',\widetilde{q})$.
Let $\T\subset\Text$ be the triangle spanned by $\psi^0(q)$, $\widehat{q}'$ and $\widetilde{q}$.
Note that $\T$ contains $q$ and $\widehat{q}$, and therefore we also have $\alpha(t)\in \T$.
As $\widehat{\alpha}(t)\in[\psi^0(q),\widehat{q}']\subset\T$, by (\ref{eq: claim stretching psi0}) we deduce $ \psi^0(\alpha(t))\in\T.$
\noindent By Remark  \ref{rem: triangles and subtriangles} we conclude that $\psi^0(q)\oplus \psi^0(\alpha(t))$ intersects $[\widehat{q}',\widetilde{q}]$ and $$C(l,\psi^0(q),\psi^0(\alpha(t)),l')\geq 1.$$
Similarly, by (\ref{eq: claim stretching psi0}) we know that $\widehat{\alpha}(t)$ belongs to the sub--triangle of $\Text$ spanned by $\psi^0(\alpha(t))$, $p$ and $l\cap l_0$.
Hence, so does $\psi^0(q)$ and by Remark \ref{rem: triangles and subtriangles} we conclude that $\psi^0(q)\oplus\psi^0(\alpha(t))$ intersects $[l\cap l_0,p]$.
This shows that $\psi^0(q)\oplus\psi^0(\alpha(t))$ intersects the topological boundary of $\Text$ in $l$ and $l'$ and this finishes the proof. 
\end{proof}

\section{Applications: length estimates}\label{sec: applications}

We now apply our results to the study of convex projective surfaces. 
In particular, we prove Theorems \ref{thm: degeneration triangular intro} and \ref{cor: domination intro}, and Corollary \ref{cor: degeneration triangular and hilbert intro}.

\subsection{Surfaces and ideal triangulations}\label{subsec: surfaces and triangulations}

Throughout we fix a compact, connected, oriented surface $S$ with non-empty boundary and negative Euler characteristic, and let $\Gamma$ denote its fundamental group.

Denote by $S_0\subset S$ the interior of $S$ and by $\widetilde{S}$ (resp. $\widetilde{S}_0$) the universal cover of $S$ (resp. $S_0$).
We let $\calb_S$ (resp. $\calb_{\widetilde{S}}$) be the set of boundary components of $S$ (resp. $\widetilde{S}$).
The orientation of $S$ induces a cyclic ordering on $\calb_{\widetilde{S}}$.

Shrinking the boundary components of $S$ to punctures we obtain a punctured surface $S'$, whose universal cover is denoted by $\widetilde{S}'$.
We fix an ideal triangulation $\calt_S$ of $S'$, with vertices at the punctures. 
We can lift $\calt_{S}$ to an ideal triangulation $\calt_{\widetilde{S}}$ of $\widetilde{S}'$, note that the set of vertices of elements in $\calt_{\widetilde{S}}$
identifies naturally with $\calb_{\widetilde{S}}$.
The set of edges of elements in $\calt_S$ (resp. $\calt_{\widetilde{S}}$) is denoted by $\cale_S$ (resp. $\cale_{\widetilde{S}}$).
Let $\caleor_S$ (resp. $\caleor_{\widetilde{S}}$) be the set of elements in $\cale_S$ (resp. $\cale_{\widetilde{S}}$) which are moreover equipped with an orientation.

\subsection{Convex projective surfaces}\label{subsec: projective structures}

We now briefly recall basic facts about convex projective structures on $S$, for which we follow closely \cite{FGProjective} and \cite{CasellaTateTillmann} which contains more details and detailed proofs.

A (\textit{real}) \textit{projective structure} on $S$ is a maximal atlas of charts $U\subset S_0\to\rp$ with transition maps locally in $\g=\PGL$.
Such a structure induces a \textit{developing map} $$\dev:\widetilde{S}_0\to\rp$$ \noindent which is equivariant with respect to a representation $\rho:\Gamma\to\g$, the \emph{holonomy} of the structure.
The projective structure is said to be \textit{convex} if $\dev$ is an embedding onto an open convex subset $\Omega_\rho\subset\rp$ and extends to $\widetilde{S}$ sending boundary components to projective segments. 
The boundary component is said to be \textit{geodesic} if it develops to a non-trivial segment, and \textit{cuspidal} if it develops to a point.

The possible  holonomies  of  convex projective structures on $S$ form a subset $\cp$ of the space of conjugacy classes of faithful and discrete representations $\Gamma\to\g$. 
Isotopic convex projective structures have the same holonomy.

We now describe a parametrization of $\cp$ extending that of Subsection \ref{subsec: positive tuples}, which is very similar to the one introduced by  Fock--Goncharov \cite{FGProjective}. We refer for details on their work to the expository article \cite{CasellaTateTillmann}. 
To introduce this parametrization, we have to work with an extension of $\cp$, the space $\cpf$ of \emph{framed convex projective structures}.

If $\rho\in\cp$, for every boundary component $\gamma$ the image $\rho(\gamma)\in \g$ is either \textit{totally hyperbolic}, namely admits distinct positive eigenvalues, \textit{quasi hyperbolic}, namely it has an eigenvalue different from one of algebraic multiplicity 2 and geometric multiplicity one, or is \textit{totally unipotent}, namely is conjugated to a positive Jordan block of length 3 (see \cite[Section 3.1]{CasellaTateTillmann}). In each case the boundary holonomy fixes at least one full flag. 

A \textit{framing} for $\rho\in\cp$ is an equivariant choice of a full flag for each boundary component $\gamma$ of $S$, that is fixed by $\rho(\gamma)$. 
If this holonomy is totally unipotent, there is a unique choice for this flag,  if it is quasi hyperbolic there are three, and if it is totally hyperbolic there are six choices. 
One way of making such choice is given by equipping each non--totally unipotent boundary component with an orientation, and choosing the flag constituted by the unique attracting fixed point of $\rho(\gamma)$ in projective plane, and the projective line to which $\gamma$ develops as the two-dimensional piece of the flag, c.f. \cite{FGProjective}. 
In particular, the orientation of $S$ induces an embedding 
\begin{equation}\label{eq: cp to cpf}
\cp\to\cpf.
\end{equation} 
\noindent More generally, the space $\cpf$ naturally covers $\cp$, ramifying over structures having some boundary that is not totally hyperbolic.

\begin{thm}[Fock--Goncharov {\cite[Theorem 2.5]{FGProjective}}, {\cite[Theorem~15]{CasellaTateTillmann}}]\label{thm: FG}
    Fix an ideal triangulation $\calt_S$ as above. 
    Each $\rho\in\cpf$ induces a continuous $\rho$--equivariant \tn{flag map} $\calb_{\widetilde{S}}\to\flags$ sending cyclically ordered $n$--tuples in $\calb_{\widetilde{S}}$ to positive $n$--tuples of flags. 
    The flag map is unique up to post--composition with an element of $\g$ and induces a homeomorphism $$\cpf\to(\rr_{>0})^{\calt_S}\times(\rr_{<0})^{\caleor_S}.$$
    \noindent Also, the convex set $\dev(\widetilde{S}_0)=\Omega_\rho$ satisfies $$\Omega_\rho\subset\displaystyle\bigcup_{T\in\calt_{\widetilde{S}}} \Ti_T(\rho)\subset\displaystyle\bigcap_{T\in\calt_{\widetilde{S}}} \Text_T(\rho),$$
    \noindent where $\Ti_T(\rho)\subset\Text_T(\rho)$ denote the pair of nested triangles associated to the image of vertices of $T\in\calt_{\widetilde{S}}$ under the flag map. 
When we embed $\cp$ into $\cpf$ using the orientation of $S$, we further have that $$\Omega_\rho=\displaystyle\bigcup_{T\in\calt_{\widetilde{S}}} \Ti_T(\rho)\subset\displaystyle\bigcap_{T\in\calt_{\widetilde{S}}} \Text_T(\rho).$$
\end{thm}

We emphasize that every ideal polygon with $n$ vertices made out from elements of $\calt_{\widetilde{S}}$ naturally determines a triangulated polygon in the sense of Subsection \ref{subsec: positive tuples}.
In this sense, the parametrization of Theorem \ref{thm: FG} extends the Fock-Goncharov-like parametrization of that subsection.
For $\rho\in\cpf$ and $E\in\caleor_S$ we let $\ttc_E(\rho)$ be the corresponding edge invariant of $\rho$.
Also, as in Subsection \ref{subsec: positive tuples} there is a natural triangle $T_E\in\calt_S$ associated to $E$. 
We let $\ttt_E(\rho)$ be the corresponding triangle invariant.

\subsection{Length functions}\label{subsec: length functions}
From now on we think of $\cp$ as a subspace of $\cpf$, embedded using the orientation on $S$ as in (\ref{eq: cp to cpf}). 
There are various length functions associated to a given $\rho\in\cp$ that we now discuss. 
Recall that for an isometry $g$ of a general metric space $(X,d)$ we may define the \textit{translation length} of $g$ by $$L(g)\coloneqq \displaystyle\inf_{x\in X}d(x,g\cdot x). $$

Let $\rho\in\cp$ and $\gamma\in\Gamma$.
The first length function to consider is the \textit{Hilbert length} $L_{\rho}^\Hil(\gamma)$, which is just the translation length of $\rho(\gamma)$ acting on the open convex set $\Omega_\rho$ equipped with its Hilbert distance (c.f. Subsection \ref{subsec: hilbert geometry}).

Second, we consider the \textit{triangular length}. 
Indeed, by Theorem \ref{thm: FG} we  regard $\Omega_\rho$ as a triangulated polygon with ``infinitely many sides". 
As in Section \ref{sec: triangular distance} we  endow $\Omega_\rho$ with the triangular distance and let $L_\rho^{\T}(\gamma)$ be the corresponding translation length for $\rho(\gamma)$.
Note that $\Omega_\rho$ is contained in the interior of $\Text_T(\rho)$ for each $T\in\calt_{\widetilde{S}}$ (c.f. Lemma \ref{lem: nested polygons}). 
Remark \ref{rem: nested finsler norms} then implies \begin{equation}\label{eq: domination triangular hilbert}
    L_\rho^{\T}(\cdot)\leq L_\rho^{\Hil}(\cdot).
\end{equation}

\noindent By Lemma \ref{lem: isometric triangular metrics} there exist large families of representations in $\cp$ with the same triangular length spectrum, hence this is ``more crude" than the Hilbert length (c.f. \cite[Corollary 8.3]{CDPWasy}).
Nevertheless, the triangular length is ``more local" as, at least infinitesimally, is  directly related to the Fock-Goncharov-like parameters, as opposed to the Hilbert length. 

As a consequence, computing the triangular length directly from the parameters  
is a difficult task.

We therefore introduce a third length function, easier to compute, that will help us understanding how triangular lengths degenerate (Proposition \ref{prop: degeneration triangular} below). 

We restrict to a special class of representations, the subset $\shif\subset\cp$ of representations whose Fock-Goncharov-like parameters are $\T$--shifting as in Subsection \ref{subsec: Tshifting tuples}. 
These are called \textit{$\T$--shifting representations} and have \textit{shear invariants} $\tts_E(\rho)=\tts_{\overline{E}}(\rho)> 1$ for each $E\in\caleor_S$.
To each $\rho\in\shif$ we  associate a metric graph $\omega_\rho$ as follows.
Consider the graph $\omega=\omega_{\calt_{\widetilde{S}}}$ \textit{dual} to $\calt_{\widetilde{S}}$, namely, the one with vertices given by elements of $\calt_{\widetilde{S}}$ and edges between adjacent triangles. 
It is a trivalent tree endowed with a free and co--compact $\Gamma$--action. 
Now, given $\rho\in\shif$ we may blow up the vertices of $\omega$ replacing them with the boundary of a triangle whenever the corresponding triangle invariant for $\rho$ is different from $1$.
This gives us a topological graph endowed with a free and co--compact $\Gamma$--action by automorphisms lifting the original $\Gamma$--action on $\omega$.
Moreover, thanks to the $\T$--shifting assumption we may turn this into a metric graph $\omega_\rho$, by declaring that the edges associated to $E\in\cale_{\widetilde{S}}$ have length equal to $\log(\tts_E(\rho))>0$, and the new edges (coming from the blown up triangles) have length given by the absolute value of the logarithm of the corresponding triangle invariant. 
We let $L^\Graph_\rho(\gamma)$ be the translation length of $\gamma$ for the action on $\omega_\rho$, and call it the \textit{$\T$--graph length}.

Using Theorem \ref{thm: FG} we may define an extension, to the context of representations, of the notion of a $\T$--stretch of $n$--tuples of flags (c.f. Section \ref{sec: stretch map}).
We can now prove Theorem \ref{thm: degeneration triangular intro}.

\begin{prop}\label{prop: degeneration triangular}
    Let $\rho\in\shif$ and for each $t> 0$ let $\rho_t\in\cp$ be a $\T$--stretch of $\rho$ of size $e^t$.
    Then $$\displaystyle\lim_{t\to\infty} \frac{e^{-t}\cdot L_{\rho_t}^{\T}(\cdot)}{L_\rho^\Graph(\cdot)}=1,$$ \noindent and this convergence is uniform in $\gamma\in\Gamma$.
\end{prop}

\begin{proof}
Endow the tree $\omega=\omega_{\calt_{\widetilde{S}}}$ with the graph distance whose edges are isometric to unit intervals. 
Let $n_\gamma\in\zz_{>0}$ be the translation length of $\gamma\in\Gamma$ acting on $\omega$.

Now, recall that there is a natural $\Gamma$--equivariant projection $\omega_\rho\to\omega$ collapsing the triangles and restricting to an embedding on the remaining edges. 
The unique $\gamma$--invariant geodesic in $\omega$ lifts to a $\gamma$--invariant geodesic in $\omega_\rho$.
In particular, for any vertex $v$ on it we have $$L_\rho^\Graph(\gamma)=d_\Graph(v,\gamma\cdot v),$$ \noindent where $d_\Graph(\cdot,\cdot)$ is the graph distance on $\omega_\rho$.
Moreover, the choice of $v$ determines an admissible sequence $E_1,\dots,E_{n_\gamma},E_{n_\gamma+1}=\gamma\cdot E_1$ along which we have the shear and triangle invariants $\tts_1(\rho),\dots\tts_{n_\gamma+1}(\rho)$ and $\ttt_1(\rho),\dots,\ttt_{n_\gamma}(\rho)$ (recall the notations from Subsection \ref{subsec: estimate triangular}).
We then have $$L_\rho^\Graph(\gamma)=d_\Graph(v,\gamma\cdot v)=\displaystyle\sum_{i=1}^{n_\gamma}\log(\tts_i(\rho))+\displaystyle\sum_{i=1}^{n_\gamma}\vert\log(\ttt_i(\rho))\vert.$$ 
\noindent \noindent Also, recall from Example \ref{ex: stretching Tshifting parameters} that the representation $\rho_t$ is $\T$--shifting for every $t > 0$. 
Moreover, there is a $\Gamma$--equivariant graph isomorphism $\omega_\rho\to\omega_{\rho_t}$ multiplying all distances by $e^t$. 
In particular, for every $\gamma\in\Gamma$ we have \begin{equation}\label{eq: grpah lengths multiples}
L_{\rho_t}^\Graph(\gamma)=e^t\cdot\left( \displaystyle\sum_{i=1}^{n_\gamma}\log(\tts_i(\rho))+\displaystyle\sum_{i=1}^{n_\gamma}\vert\log(\ttt_i(\rho))\vert \right)=e^t\cdot L_\rho^\Graph(\gamma).
\end{equation}
\noindent As a consequence, we only need to show that $L_{\rho_t}^\T(\gamma)/ L_{\rho_t}^\Graph(\gamma)$ \noindent goes to $1$ uniformly, for which we apply Proposition \ref{prop: additve estimate triangular distance refined}.

Indeed, let $q=q_{E_1}$ be as in that proposition applied to $\Omega_{\rho_t}$, so that we have $\rho_t(\gamma)\cdot q=q_{E_{n_\gamma+1}}$.
By Proposition \ref{prop: additve estimate triangular distance refined} and (\ref{eq: grpah lengths multiples}) we deduce $$ L_{\rho_t}^{\T}(\gamma)\leq d_{\T}(q,\rho_t(\gamma)\cdot q)\leq  L_{\rho_t}^\Graph(\gamma) +D\cdot n_\gamma ,$$ \noindent for some universal constant $D>1$.
Complementary, note that $$L_{\rho_t}^{\T}(\gamma)\geq \displaystyle\lim_{k\to\infty} \frac{d_{\T}(q,\rho_t(\gamma^k)\cdot  q)}{k}.$$
\noindent By Proposition \ref{prop: additve estimate triangular distance refined} and (\ref{eq: grpah lengths multiples}) we deduce $L_{\rho_t}^{\T}(\gamma)\geq  L_{\rho_t}^\Graph(\gamma)- D\cdot n_\gamma$ and therefore $$1-\frac{D\cdot n_\gamma}{L_{\rho_t}^\Graph(\gamma)}\leq \frac{L_{\rho_t}^\T(\gamma)}{L_{\rho_t}^\Graph(\gamma)}\leq 1+\frac{D\cdot n_\gamma}{L_{\rho_t}^\Graph(\gamma)}$$
\noindent for every $t>0$.
But thanks to (\ref{eq: grpah lengths multiples}) the quantity $D\cdot n_\gamma / L_{\rho_t}^\Graph(\gamma)$ converges to $0$ uniformly on $\gamma$, and this completes the proof.\end{proof}

We can now deduce Theorem \ref{cor: domination intro}.

\begin{cor}\label{cor: domination}
    For every $\rho\in\shif$ and every $\gamma\in\Gamma$ one has $$L_\rho^\Graph(\gamma)\leq L_\rho^{\T}(\gamma)\leq L_\rho^{\Hil}(\gamma).$$
\end{cor}

\begin{proof}
By (\ref{eq: domination triangular hilbert}) we only need to prove the first inequality.
But this follows by considering a $\T$--stretch of $\rho$ of size $e^t$ and applying Proposition \ref{prop: degeneration triangular} and Theorem \ref{thm: Klip}.
\end{proof}

\subsection{Link with Parreau's work}\label{subsec: parreau}
Proposition \ref{prop: degeneration triangular} takes inspiration from Parreau \cite{ParreauInv}, where the author finds a similar statement (c.f. Theorem \ref{thm: parreau} below) but for the Hilbert length (and more generally the $\liea$--\textit{length}).
The goal now is to discuss the link with that work in detail and to prove Corollary \ref{cor: degeneration triangular and hilbert intro}.
We mention that Parreau's results apply in a setting more general than ours. 
Instead of discussing that work in that level of generality, we restrict to our setting of interest. 

To find the alluded degeneration Parreau needs, as we do, to impose certain conditions on the path $t\mapsto\rho_t$ that are called \textit{left--shifting} and \textit{edge--separation}.
We now recall what these notions are in our setting and, along the way, provide comparisons with our hypotheses.

In order to do that, we first need to relate Fock--Goncharov's Theorem \ref{thm: FG} with Parreau's choice of parameters, as there are small differences.
Concretely, for each triangle Parreau considers the same invariant as we do.
In contrast, the parameters associated to oriented edges are different. 
Namely, for each $\rho\in\cp$ and each $E\in\caleor_S$ Parreau associates a parameter that we will denote here by $ \parinv_E(\rho) $\footnote{This is denoted by $S_e$ in \cite[p. 2312]{ParreauInv}.}. 
A computation gives: $$\parinv_E(\rho)=\frac{-1}{(1+\ttt_E(\rho))\ttc_{\overline{E}}(\rho)}=((1+\ttt_E(\rho))\vert\ttc_{\overline{E}}(\rho)\vert)^{-1}.$$

Now, for arbitrary $\rho\in\cp$ consider the associated \textit{stretch path}, namely, for $t\geq 0$ we let $\rho_t\in\cp$ be the representation given by 
    $$\ttt_E(\rho_t)\coloneqq \ttt_E(\rho)^{e^t} \tn{ and } \ttc_E(\rho_t)\coloneqq  -(-\ttc_E(\rho))^{e^t}$$ \noindent for all $E\in\caleor_S$.
\noindent Note that, in particular, $\rho_t$ is a $\T$--stretch in the sense of Proposition \ref{prop: degeneration triangular}.
Moreover, associated to this deformation one may consider the following parameters ``at infinity": for $E\in\caleor_S$ let $$\ttt_E(\rho_\infty)\coloneqq \displaystyle\lim_{t\to\infty} \vert\ttt_E(\rho_t)\vert^{\frac{1}{e^t}}=\ttt_E(\rho)$$ \noindent and more importantly $$\parinv_E(\rho_\infty)\coloneqq 
\displaystyle\lim_{t\to\infty}\left\vert \parinv_E(\rho_t) \right\vert^{\frac{1}{e^t}}=\left\lbrace\begin{array}{l}
  \vert \ttc_{\overline{E}}(\rho)\vert^{-1}  \hspace{1,83cm}\tn{ if } \ttt_E(\rho)\leq 1  \\
   (\ttt_E(\rho)\vert \ttc_{\overline{E}}(\rho)\vert)^{-1} \hspace{0,7cm}\tn{ if } \ttt_E(\rho)\geq 1
\end{array}\right..$$
\noindent These are indeed the ``Fock-Goncharov-like parameters" of a representation $\rho_\infty:\Gamma\to\mathsf{PGL}_3(\mathbb{K})$ for an appropriate field $\mathbb{K}$ depending on the scaling power $e^t$, see \cite[Section 5]{ParreauInv}.

The representation $\rho$ is said to be \textit{left--shifting} if the associated stretch path satisfies \begin{equation}\label{eq: leftshifting I}
  \parinv_E(\rho_\infty)> \max\left\lbrace\displaystyle\lim_{t\to\infty}\left(1+ \ttt_{\overline{E}}(\rho_t)^{-1}\right)^{-\frac{1}{e^t}} ,  \displaystyle\lim_{t\to\infty}\left(1+ \ttt_E(\rho_t)\right)^{-\frac{1}{e^t}}\right\rbrace  
\end{equation}
\noindent and \begin{equation}\label{eq: leftshifting II}
  \parinv_{\overline{E}}(\rho_\infty)> \max\left\lbrace\displaystyle\lim_{t\to\infty}\left(1+ \ttt_{\overline{E}}(\rho_t)\right)^{-\frac{1}{e^t}} ,  \displaystyle\lim_{t\to\infty}\left(1+ \ttt_E(\rho_t)^{-1}\right)^{-\frac{1}{e^t}}\right\rbrace  
\end{equation} \noindent for every $E\in\caleor_S$, c.f. \cite[Definition 3.3]{ParreauInv}\footnote{We note that the involved limits are computed as $$ \displaystyle\lim_{K\to\infty} (1+\eta^K)^{1/K}=\max\{\eta,1\}$$ \noindent for every $\eta\in\rr_{>0}$.}.
\noindent Here is a first comparison with our work.

\begin{rem}\label{rem: Tshifting with parreau}
Let $\rho$ be an arbitrary convex projective structure on $S$ and $t\mapsto \rho_t$ be the associated stretch path. 
We may reinterpret the $\T$--shifting condition of Subsection \ref{subsec: length functions} for $\rho$ as a condition on products of Parreau's edge parameters at infinity. 
Indeed, by Lemma \ref{lem: c.r. distinguished points quadruple} and (\ref{eq: shear quadruple}) we have $$\parinv_E(\rho_\infty)\cdot \parinv_{\overline{E}}(\rho_\infty)=\tts_E(\rho)$$ \noindent  for every $E\in\caleor_S$.
In particular, $\rho$ is $\T$--shifting if and only if for every $E\in\caleor_S$ one has $$\parinv_E(\rho_\infty)\cdot \parinv_{\overline{E}}(\rho_\infty)> 1.$$
\noindent As a consequence, a given $\rho$ may be left--shifting without being $\T$--shifting (and conversely, a $\T$--shifting representation is not necessarily left--shifting). 
A simple situation where all conditions are fulfilled is when $\rho$ satisfies  $\parinv_E(\rho_\infty)>1$ for all $E\in\caleor_S$.
However, if $\rho$ is left--shifting and (say) $\parinv_E(\rho_\infty)\leq 1$, then we necessarily have $\parinv_{\overline{E}}(\rho_\infty)>1$ but their product could be either $> 1$ or $\leq 1$.
\end{rem}

The left--shifting hypothesis is one of the ingredients in Parreau's degeneracy Theorem \ref{thm: parreau}.  
It allows the construction of a metric $\Gamma$--space (similar to our $\omega_\rho$) whose associated length spectrum describes the asymptotic behavior of the re--scaled limits of the Hilbert length.
Roughly speaking, this space is a $2$--dimensional complex homotopically equivalent to $S$ which is moreover equipped with a geometric structure locally modeled on a Cartan subspace of $\g$. 
In particular, one may consider on it the Finsler norm induced from the Hilbert length functional. 
The associated (path) metric space will be called the \textit{Hilbert--Parreau complex}, see \cite[Section 3.3]{ParreauInv} and the proof of Proposition \ref{prop: hilbert parreau length spectrum} below for more details.
As the natural $\Gamma$--action preserves this distance, we have an associated \textit{Hilbert--Parreau length spectrum} which we will denote by $L_\rho^\HP(\cdot)$. 
The first link to our work is the following (see the end of this subsection for a proof).

\begin{prop}\label{prop: hilbert parreau length spectrum}
    Let $\rho$ be a $\T$--shifting and left--shifting representation. 
    Then the $\T$--graph length dominates the  Hilbert--Parreau length: $L_\rho^\HP(\gamma)\leq L_\rho^\Graph(\gamma)$ for every $\gamma\in\Gamma$.
\end{prop}

As outlined before, the Hilbert--Parreau length spectrum captures the degeneracy of re--scaled Hilbert lengths along stretch paths. 
The precise result requires in fact an extra assumption on $\rho$.
Namely, we say that $\rho$ is \textit{edge--separating} if for every triangle $T\in\calt_S$ and every pair $E_1,E_2$ of edges of $T$, oriented so that $T=T_{\overline{E}_1}$ and $T=T_{\overline{E}_2}$ (c.f. Subsection \ref{subsec: positive tuples}), one has $$\parinv_{E_1}(\rho_\infty)\cdot \parinv_{E_2}(\rho_\infty)> \displaystyle\lim_{t\to\infty}\left(1+ \ttt_{\overline{E}_1}(\rho_t)^{-1}\right)^{-\frac{1}{e^t}}= \displaystyle\lim_{t\to\infty}\left(1+ \ttt_{\overline{E}_2}(\rho_t)^{-1}\right)^{-\frac{1}{e^t}} $$ \noindent and $$\parinv_{\overline{E}_1}(\rho_\infty)\cdot \parinv_{\overline{E}_2}(\rho_\infty)> \displaystyle\lim_{t\to\infty}\left(1+ \ttt_{\overline{E}_1}(\rho_t)\right)^{-\frac{1}{e^t}}= \displaystyle\lim_{t\to\infty}\left(1+ \ttt_{\overline{E}_2}(\rho_t)\right)^{-\frac{1}{e^t}}, $$ \noindent c.f. \cite[Theorem 4.3]{ParreauInv}.

\begin{thm}[Parreau {\cite[Theorem 3]{ParreauInv}}]\label{thm: parreau}
    Let $\rho\in\cp$ be a left--shifting and edge--separating representation. 
    Let $t\mapsto\rho_t$ be the associated stretch path. 
    Then for every $\gamma\in\Gamma$ one has $$\displaystyle\lim_{t\to\infty} \frac{e^{-t}\cdot L_{\rho_t}^\Hil(\gamma)}{L_\rho^\HP(\gamma)}=1.$$
\end{thm}

We can now prove Corollary \ref{cor: degeneration triangular and hilbert intro}.

\begin{cor}\label{cor: degeneration triangular and hilbert}
    Let $\rho\in\cp$ be a $\T$--shifting, left--shifting and edge--separating representation. 
    Let $t\mapsto\rho_t$ be the associated stretch path. 
    Then for every $\gamma\in\Gamma$ we have $$\lim_{t\to\infty} \frac{L_{\rho_t}^{\Hil}(\gamma)}{ L_{\rho_t}^\T(\gamma)}=1.$$ 
    \noindent In particular, $L_\rho^\HP(\gamma)=L_\rho^\Graph(\gamma)$.
\end{cor}

\begin{proof}
On the one hand, by (\ref{eq: domination triangular hilbert}) we have $$\displaystyle\liminf_{t\to\infty}\frac{L_{\rho_t}^{\Hil}(\gamma)}{ L_{\rho_t}^\T(\gamma)}\geq 1.$$
\noindent On the other hand, fix $\varepsilon>0$ and $\gamma\in\Gamma$.
By Theorem \ref{thm: parreau} we have $$e^{-t}\cdot L_{\rho_t}^\Hil(\gamma)\leq (1+\varepsilon)\cdot L_\rho^\HP(\gamma)$$ \noindent for every $t$ large enough. 
Moreover, by Proposition \ref{prop: degeneration triangular} we also have $$e^{-t}\cdot L_{\rho_t}^\T(\gamma)\geq (1-\varepsilon)\cdot L_\rho^\Graph(\gamma)$$ \noindent for every $t$ large enough. 
Hence, $$\displaystyle\limsup_{t\to\infty} \frac{L_{\rho_t}^\Hil(\gamma)}{L_{\rho_t}^\Graph(\gamma)}\leq \frac{L_\rho^\HP(\gamma)}{L_\rho^\Graph(\gamma)}$$ \noindent and by Proposition \ref{prop: hilbert parreau length spectrum} the proof is complete. \end{proof}

We finish this subsection with a proof of Proposition \ref{prop: hilbert parreau length spectrum}.

\begin{proof}[Proof of Proposition \ref{prop: hilbert parreau length spectrum}]

Let $d_{\HP}(\cdot,\cdot)$ be the Hilbert--Parreau distance. 
The proof goes by finding a natural $\Gamma$--equivariant embedding $\tau_\rho$ of $\omega_\rho$ into the Hilbert--Parreau complex which, restricted to edges of $\omega_\rho$, is isometric. 
We note that this is enough: for every $\gamma\in\Gamma$ we may pick a vertex $v\in\omega_\rho$ so that $ L_\rho^\Graph(\gamma)=d_\Graph(v,\gamma\cdot v)$ and therefore by the triangle inequality we have $$  L_\rho^\HP(\gamma)\leq d_\HP(\tau_\rho(v),\gamma\cdot \tau_\rho(v))\leq d_\Graph(v,\gamma\cdot v)=L_\rho^\Graph(\gamma)$$ \noindent as desired. 

Let us then prove the existence of $\tau_\rho$, for which we first need to setup some notations. 
Let $\liea$ be a Cartan subspace of $\g$ and $\liea^+\subset\liea$ be a closed Weyl chamber. 
Let $\alpha_1\in\liea^*$ and $\alpha_2\in\liea^*$ be the associated simple roots, which determine the isomorphism $\liea\to\rr^2$ given by $x\mapsto (\alpha_1(x),\alpha_2(x))$. 
In these coordinates, the Hilbert length functional $\tn{H}\in\liea^*$ reads as $\tn{H}(x_1,x_2) = x_1+x_2$.
Recall that a direction in $\liea$ is \textit{singular} if it is parallel to either $\ker(\alpha_1)$, $\ker(\alpha_2)$, or $\ker(\tn{H})$.

Now, the Hilbert--Parreau complex is roughly defined as follows (see \cite[Section 3.3]{ParreauInv} for more details).
On the one hand, for each $T\in\calt_{\widetilde{S}}$ we take an equilateral triangle $P^T\subset\liea$ with singular sides of length equal to the absolute value of the logarithm of the corresponding triangle invariant.
The piece of $\omega_\rho$ associated to $T$ admits then a natural embedding as the booundary of $P^T$.
The different triangles $P^T\subset\liea$ are glued together as follows.  
Let $E\in\cale_{\widetilde{S}}$ be an edge which is the intersection of two adjacent triangles $T,T'\in\calt_{\widetilde{S}}$. 
We orient $E$ in such a way that $T'=T_E$ and $T=T_{\overline{E}}$ (c.f. Subsection \ref{subsec: positive tuples}).
To situations may appear. 

First, if both parameters $\parinv_{E}(\rho_\infty)$ and $\parinv_{\overline{E}}(\rho_\infty)$ are bigger than $ 1$ then, by definition, the corresponding piece $P^E\subset\liea$ of the Hilbert--Parreau complex is the segment between $(0,0)$ and $$(x_{\overline{E}},x_E)\coloneqq (\log (\parinv_{\overline{E}}(\rho_\infty)),\log (\parinv_{E}(\rho_\infty))).$$
\noindent Observe that the Hilbert length of this segment is $$x_{\overline{E}}+x_E=\log(\parinv_{\overline{E}}(\rho_\infty))+\log (\parinv_{E}(\rho_\infty))=\log(\tts_E(\rho))$$ \noindent (c.f. Remark \ref{rem: Tshifting with parreau}), and therefore it is isometric to the edge of $\omega_\rho$ associated to $E$.

The second case is if one of the parameters at infinity, say $\parinv_E(\rho_\infty)$, is smaller than $1$ (the case 
$\parinv_{\overline{E}}(\rho_\infty)\leq 1$ is similar).
In this case,
by (\ref{eq: leftshifting I}), we have $$\ttt_{\overline{E}}(\rho) <1<\ttt_E(\rho)$$ \noindent and by (\ref{eq: leftshifting II}) we also have $ \parinv_{\overline{E}}(\rho_\infty)>1$.
The piece $P^E\subset\liea$ is defined in this case as the parallelogram of vertices $$(0,0), \tn{ }(0,x_E), \tn{ }(x_{\overline{E}},0) \tn{ and } (x_{\overline{E}},x_E).$$
\noindent Note that thanks to (\ref{eq: leftshifting I}) the side of $P^E$ between $(0,0)$ and $(0,x_E)$ (resp. $(x_{\overline{E}},x_E)$ and $(x_{\overline{E}},0)$) has length smaller than the side length of $P^{T_{\overline{E}}}$ (resp. $P^{T_E}$). 
Hence, the parallelogram $P^E$ may be glued to the triangles $P^{T_{\overline{E}}}$ and $P^{T_E}$ in such a way that the vertex $(0,0)$ (resp. $(x_{\overline{E}},x_E)$) is identified with a vertex of $P^{T_{\overline{E}}}$ (resp. $P^{T_E}$), see Figure \ref{fig: hilbertparreau and omegarho}. 
\begin{figure}[h] 
\scalebox{1}{%
\begin{overpic}[scale=1, width=1\textwidth, tics=2]{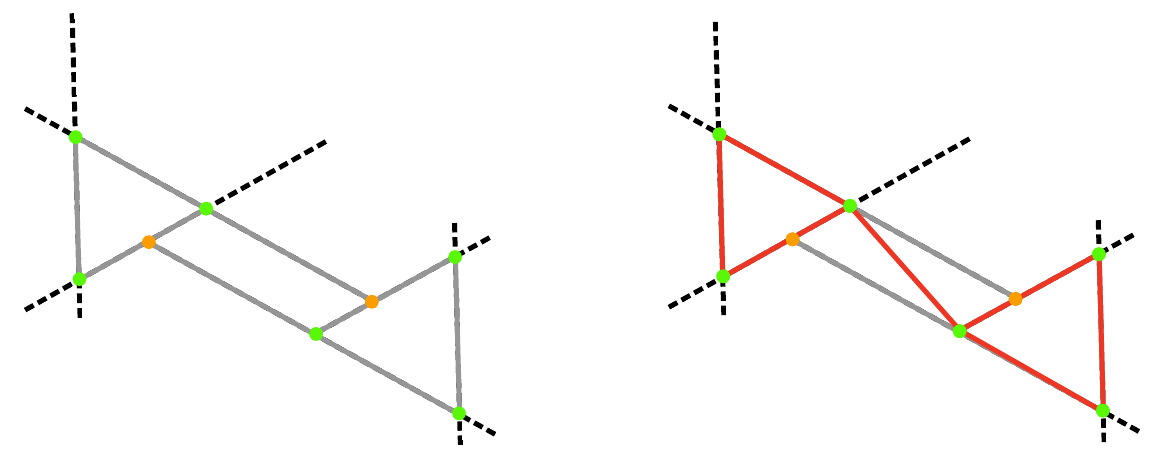}

    \put (8,21)  {\tc{gray}{$P^{T_{\overline{E}}}$}}
    \put (32,10)  {\tc{gray}{$P^{T_E}$}}
    \put (20,15.5)  {\tc{gray}{$P^E$}}

    \put (7,37)  {$\tn{H}>0$}
    \put (26,21)  {$\liea^+$}
    \put (28,27.5)  {$\alpha_1=0$}
    \put (-2,31)  {$\alpha_2=0$}

    \put (14.5,24)  {\tc{green}{$(0,0)$}}
    \put (19,7.5)  {\tc{green}{$(x_{\overline{E}},x_E)$}}
    
    \put (28,17)  {\tc{orangepic}{$(x_{\overline{E}},0)$}}
    \put (8.5,15)  {\tc{orangepic}{$(0,x_E)$}}
   \end{overpic}}	
   \caption{Left: three polygons $P^{T_E}$, $P^E$ and $P^{T_{\overline{E}}}$ glued along sides, giving the local development of the Hilbert--Parreau complex in the case $\parinv_E(\rho_\infty)\leq 1$.
   Right: in red, the local image of $\omega_\rho$ under $\tau_\rho$.}
   \label{fig: hilbertparreau and omegarho}
   \end{figure}
\noindent The embedding of the piece of $\omega_\rho$ corresponding to $E$ is given by choosing the diagonal between $(0,0)$ and $(x_{\overline{E}},x_E)$, and this finishes the proof.\end{proof}

\appendix

\section{Ortho-geodesic stretch maps in the hyperbolic plane}\label{appendix: maps H2}

In this appendix, we study two ortho-geodesic stretch maps in the hyperbolic plane $\mathbb{H}^2$ and we compute their precise Lipschitz constants. The first computation serves as a model for our construction of ortho-geodesic stretch maps on $\mathbb{P}(\mathbb{R}^3)$ in Section \ref{sec: stretch map}. The major difference between our proposal here with Thurston's original work \cite{ThurstonStretch} for hyperbolic surfaces is to replace his horocyclic foliation by some ortho-geodesic foliation. The ortho-geodesic foliation allow us to define, similar to Thurston's work, $K$-Lipschitz maps on the hyperbolic ideal triangle, and have easier generalizations in projective geometry.

We work in an affine chart of the Klein model of the hyperbolic plane $\mathbb{H}^2$. Consider an ideal triangle $\Ti$ in $\mathbb{H}^2$ with vertices $p_1, p_2, p_3$. Orthogonally project $p_1, p_2, p_3$ to the opposite sides and denote the projection points as $s_1, s_2, s_3$ respectively. The geodesic segments $[p_i,s_i]$ meet at the barycenter $t$ of the triangle $\Ti$.  Equivalently, we can define $s_i$ as in Subsetion \ref{subsec: triple ratio}. The point $p_i$ is associated with a plane $l_i$ and the pair $(p_i,l_i)$ yields a flag $F_i$. The intersections between $l_{i}$ and $l_{j}$ are denoted as $p_{ij}=l_i\cap l_j$ for $1\leq i<j\leq 3$. Then as in figure \ref{fig: HyperbolicF1} below,

$$s_i\coloneqq(p_i\oplus p_{jk})\cap( p_j\oplus p_k).$$

 For the Fuchsian case, the triple ratio $T(F_1,F_2,F_3)=1$. Therefore when $i\in\{1,2,3\}$,
$$t= t_i\coloneqq(p_j\oplus p_{ik})\cap (p_k\oplus p_{ij}).$$

We obtain three symmetric triangles in $\T_{p_1p_2p_3}$ divided by the geodesic $p_1\oplus t$, $p_2\oplus t$ and $p_3\oplus t$: these are $\T_{p_1p_2t}$, $\T_{p_2p_3t}$ and $\T_{p_3p_1t}$\footnote{They correspond to $\T_{123}, \T_{231}, \T_{312}$ in Subsection \ref{subsec: triple ratio} respectively.}. Moreover, in the Fuchsian case $\Teq$ (recall Subsection \ref{subsec: triple ratio}) degenerates to a point $t$.  
\begin{equation*}
\Ti=\T_{p_1p_2p_3}=\T_{p_1p_2t} \cup \T_{p_2p_3t}\cup \T_{p_3p_1t}.
\end{equation*}

We will define our maps on one of these three triangles, say triangle $\T_{p_1p_2t}$, and then extend the map to the whole ideal triangle $\Ti$ using symmetry.

\begin{figure}[!htb]
   \centering
    \includegraphics[width=.9\textwidth]{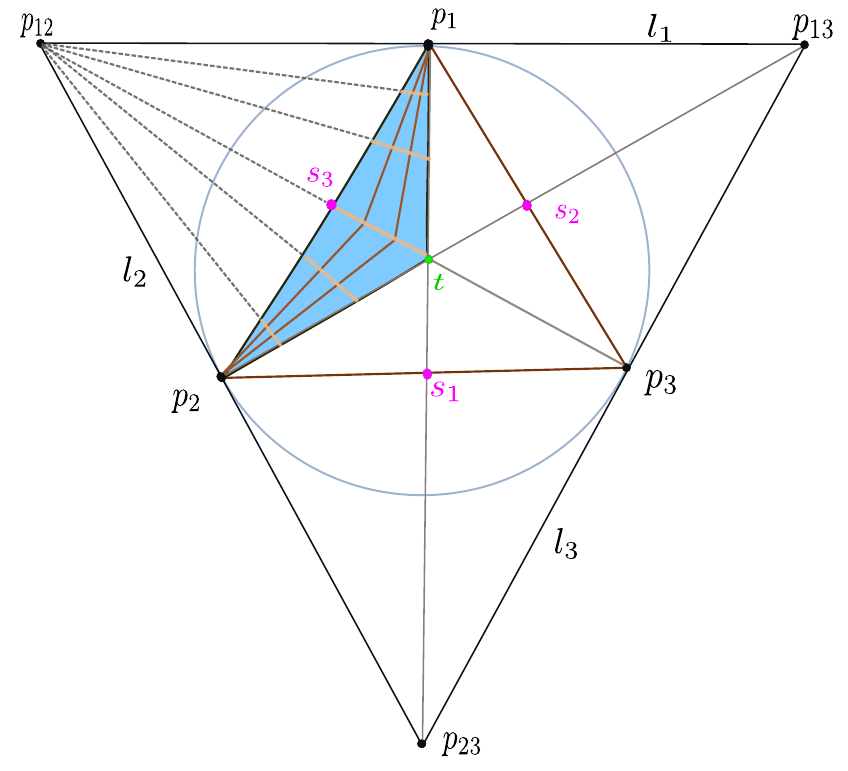}
    \caption{Geodesic foliations in the projective model of $\mathbb{H}^2$. The ortho-geodesic foliation $\folo$ for the $\T_{p_1p_2t}$ is drawn in gray dashed and yellow lines. The cusp geodesic foliations $\folc^{p_1}$ and $\folc^{p_2}$ are drawn in brown lines. They are expained in Subsection \ref{subsection, orthegeodesicFoliation}}. 
    \label{fig: HyperbolicF1}
\end{figure}

\subsection{Ortho--geodesic foliation} \label{subsection, orthegeodesicFoliation}
The triangle $\T_{p_1p_2t}$ is divided by the geodesic segment $[t,s_3]$ in the pieces $\T_{p_1}$ and $\T_{p_2}$ . For the explicit calculations, it is convenient to switch to the upper half plane model of $\mathbb{H}^2$. We denote the unique geodesic between $x$ and $y$ there by $\overline{xy}$. Up to a Mobius transformation, we assume $p_1=\infty$, $p_2=0$, $s_3=i$. These three points determine the coordinates of every point of $\T_{p_1p_2t}$.

 Let  $\{(x,y) \text{ }| \text{ } x\in\mathbb{R} \text{ and } y> 0 \}$ be the standard Euclidean system of the upper half plane model of $\mathbb{H}^2$. We additionally introduce the  coordinate system $(x,w)$ with
$$ w=\sqrt{x^2+y^2}.$$ 

The equation $w \equiv \text{constant}>0$ describes a half circle centered at 0 and orthogonal to the imaginary axis in the upper half plane. Therefore, we call this coordinate system the \emph{ortho-geodesic coordinate}. 

 Next we explain the counterparts in the upper half plane of the geodesic foliations 
 (see Figure \ref{fig: HyperbolicFigure2} and compare Subsection \ref{sec: stretch map}). 
 First, the \emph{ortho-geodesic foliation} $\folo$ is the foliation given by the geodesics orthogonal to the geodesic $\overline{p_1p_2}$. Equivalently in the ortho-geodesic coordinate $(x,w)$, these are ortho-geodeiscs given by the equation $w \equiv \text{constant}$. Then we have the \emph{cusp geodesic foliation}. They are divided into two cases: the cusp geodesic foliation through $p_1$, denoted as $\folc^{p_1}$, is the foliation given by all vertical geodesics which are all geodesics through $p_1=\infty$; the cusp geodesic foliations through $p_2$, denoted as $\folc^{p_2}$, is the foliation given by all geodesics through the endpoint $p_2$ (see Figure \ref{fig: HyperbolicFigure2}). If we draw these foliations in projective model, then we are back to Figure \ref{fig: HyperbolicF1}.

\begin{figure}[!htb]
   \centering
    \includegraphics[width=\textwidth]{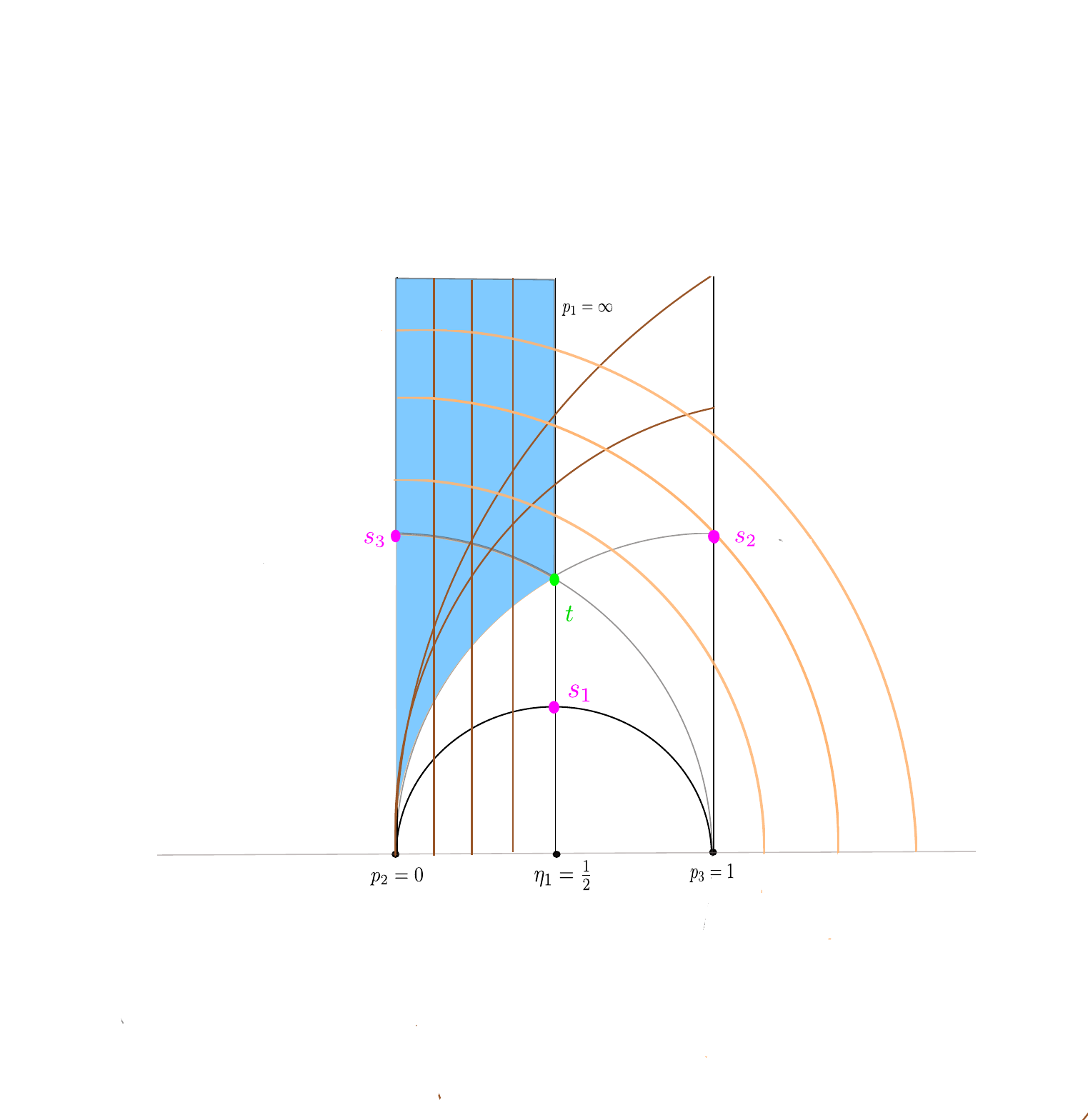}
    \caption{ The geodesic foliations in upper half plane model of $\mathbb{H}^2$. The ortho-geodesic foliation $\folo$ for the $\T_{p_1p_2t}$ is drawn in yellow lines. The cusp geodesic foliation $\folc^{p_1}$ and $\folc^{p_2}$ are drawn in brown lines. } 
    \label{fig: HyperbolicFigure2}
\end{figure}

\subsection{Two ortho-geodesic stretch maps} 

We introduce two stretch maps on the ideal triangle of $\mathbb{H}^2$ in Subsections \ref{subsection, firstmap} and \ref{subsection, secondmap}. Both maps satisfy the requirements stated at the beginning of Section \ref{sec: stretch map}. The first map, presented in Subsection \ref{subsection, firstmap},  generalizes to the ortho-geodesic stretch map $\psi^0=\psi^0_K$ of Subsection \ref{subsub: homeo rigid piece}.
Subsection \ref{subsection, secondmap} presents a second map, whose generalization to the projective setting is not discussed in the paper. We include it nevertheless, as the explicit  computation may be of independent interest.

\subsubsection{First map on the ideal triangle} \label{subsection, firstmap}

We define our map on the triangle $\T_{p_1}$, whose edges are given by the geodesics $\overline{p_1s_3}$, $\overline{s_3t}$ and $\overline{tp_1}$ (recall Figure \ref{fig: HyperbolicFigure2}). In the ortho-geodesic coordinate system of the upper half plane model of $\mathbb{H}^2$, the triangle $\T_{p_1}$ is 
$$\T_{p_1}=\{(x,w)\in \mathbb{H}^2 \text{ }| \text{ }0\leq x \leq \frac{1}{2}, w \geq 1\}.$$

Let $K>1$. We define a Lipschitz map $ \psi_K $ in $\T_{p_1}$ using the ortho-geodesic coordinate as follows,
  \begin{align*}
      \psi_K  : \text{ }& \T_{p_1} \text{  } \xrightarrow{\hspace*{0.6cm}} \text{  } \T_{p_1},\\
       &\psi_K(x,w)= (x, w^K).
   \end{align*}

The geometric features of this map is described in the following proposition.  Here, we denote by $\folo_{s_3}$ the leaf of $\folo$ passing through $s_3$ and by   $\folo_Z$ the leaf of $\folo$ passing through $Z\in \overline{p_1p_2}\cap \T_{p_1}$. We denote also by $\eta_1$ the intersection point of the geodesic $\overline{ts_1}$ with the boundary of the hyperbolic plane $\mathbb{H}^2$. 
\begin{prop}
 The map $\psi_K$ preserves the foliation $\folc^{p_1}$ leafwise and preserves the foliation $\folo$ (taking leaves to leaves). Furthermore 
\begin{equation} \label{eqtn:distanceEqtn1}
    d_{\mathbb{H}^2}(\folo_{s_3} ,\psi_K (\folo_{Z}))=K  d_{\mathbb{H}^2}(\folo_{s_3},\folo_Z).
\end{equation}

\end{prop}

\begin{proof}

 The fact $\psi_K$ preserves the geodesic foliation $\folc^{p_1}$ leafwise follows since $\psi_K$ does not change the $x$-coordinate. Since the equation $w \equiv \text{constant}> 0 $ parametrizes a half circle in the upper half plane, the map $\psi_K$ mapping $w$ to $w^K$ implies that $\psi_K$ preserves $\folo$ and takes one leaf to another. Lastly, since $\psi_K(0,w)= (0, w^K)$, a simple hyperbolic geometry computation yields Equation \ref{eqtn:distanceEqtn1}.
\end{proof}
We compute the Lipschitz constant of our Lipschitz map $\psi_K$. Throughout the appendix we fix $K>1$. 
\begin{prop} \label{prop,ortho-geodesicsExtreme}
 The map  $\psi_K: \T_{p_1}\text{  } \xrightarrow{\hspace*{0.6cm}} \text{  } \T_{p_1}$ 
   is a Lipschitz map with Lipschitz constant $\text{Lip}(\psi_K)=K$. 
\end{prop}

\begin{proof}
We give a computation using the ortho-geodesic coordinates system $(x,w)$. Recall our map $\psi_K: \T_{p_1} \text{  } \xrightarrow{\hspace*{0.6cm}} \text{  } \T_{p_1}$ in the coordinate is
\begin{equation*}
    \psi_K(x,w)=(x, w^K).
\end{equation*}

The differential of $\psi_K$ is thus
\begin{equation*}
D\psi_K|_{(x,w)}=
\begin{pmatrix}
1 & 0 \\
0 & K w^{K-1}
\end{pmatrix}.
\end{equation*}

A computation shows that the hyperbolic metric in the ortho-geodesic coordinates system is

$$g_{\scriptscriptstyle\mathbb{H}^2}|_{(x,w)}=\frac{d x^2 }{w^2-x^2} + \frac{w^2 d w^2 }{x^2(w^2-x^2)^2}+\frac{2w}{x(w^2-x^2)} dw dx.$$

With this information, the norm of the differential at a point $Z$ with coefficient $(x,w)$ is computed as
\begin{align*}
    \norm{D\psi_K}_{Z}^2 &=\max_{V\in T_{\scriptscriptstyle Z}\mathbb{H}^2} \frac{\norm{D\psi_K(V)}^2_{g_{\scriptscriptstyle\mathbb{H}^2}}}{\norm{V}^2_{g_{\scriptscriptstyle\mathbb{H}^2}}}\\
    &=\max_{V\in T_{\scriptscriptstyle Z}\mathbb{H}^2}\frac{\frac{V_1^2}{w^{2K}-x^2}+\frac{K^2w^{4K-2}V_2^2}{(w^{2K}-x^2)^2}+\frac{2Kw^{2K-1}V1V2}{x(w^{2K}-x^2)}}{\frac{V_1^2}{w^{2}-x^2}+\frac{V^2_2w^2}{(w^{2}-x^2)^2}+ \frac{2w V_1V_2}{x(w^2-x^2)}}.
\end{align*}
Here $(V_1,V_2)$ is the ortho-geodesic coordinate of the tangent vector $V$. 
With the assumptions $w\geq 1$, $x\in[0,\frac{1}{2}]$ and $K>1$ the following estimates hold: 
\begin{equation}
    \frac{w^{2}-x^2}{w^{2K}-x^2}\leq 1,
\end{equation}
and
\begin{equation}
   \frac{K^2w^{4(K-1)}(w^2-x^2)^2}{(w^{2K}-x^2)^2}\leq K^2,
\end{equation}

and 
\begin{equation}
   \frac{K(w^{2K}-w^{2K-2}x^2)}{w^{2K}-x^2}\leq K,
\end{equation}

It is an easy fact that $
\frac{a_1+a_2+a_3}{b_1+b_2+b_3}
\leq
\max\limits_{1\leq i\leq 3}\big\{\frac{a_i}{b_i}\big\}
$, provided $b_1,b_2,b_3>0$.
Therefore, we conclude $ \norm{D\psi_K}_{Z} \leq K$ while equality $ \norm{D\psi_K}_{
Z} = K$ holds precisely when $x=0$ or $w=1$. This shows that $\psi_K$ is $K$-Lipschitz on $\T_{p_1p_2t}$.
\end{proof}

We first extend the map $\psi_K$ from $\T_{p_2}$ to $\T_{p_1p_2t}$ by symmetry, then from $\T_{p_1p_2t}$ to $\T_{p_2p_3t}$ and $\T_{p_1 p_3t}$ by symmetry. This yields the following corollary. 

\begin{cor}\label{cor,ExtremeLip}
  The map  $\psi_K: \T_{p_1p_2p_3} \text{  } \xrightarrow{\hspace*{0.6cm}} \text{  }\T_{p_1p_2p_3}$
   is a Lipschitz map with Lipschitz constant $\text{Lip}(\psi_K)=K$. 
\end{cor}

\subsubsection{Second map on the ideal triangle}  \label{subsection, secondmap}
We define our second map again on the triangle $\T_{p_1}$. The triangle $\T_{p_1}$ in the standard Euclidean coordinate is

 $$\T_{p_1}=\{(x,y)\in \mathbb{H}^2 \text{ }| \text{ }0\leq x \leq \frac{1}{2}, x^2+y^2 \geq 1\}.$$

It is more convenient to define and do compuations for this new map using polar coordinates $(R,\theta)$ given by 
\begin{align*}
    x&=R\cos\theta,\\
    y&=R\sin\theta.
\end{align*}
Again, we fix $K>1$. Our second map $\varphi=\varphi_K$ on $\T_{p_1}$ in polar coordinate is
\begin{align*}
&\varphi_K: \T_{p_1} \longrightarrow \T_{p_1}\\
    &\varphi_K(R,\theta)=(R^K, \alpha),
\end{align*}
where $\cos \alpha =2^{K-1}(\cos \theta)^K$. In terms of the Euclidean coordinate, suppose a point $Z$ in $\T_{p_1}$ has Euclidean coordinate $(x_Z,y_Z)$, then the angle equation is a scaling equation of the $x$-coordinate : $x_{\varphi_K(Z)}= 2^{K-1} x_{Z}^K$.

The geometric feature of this map is described in the following proposition. 
\begin{prop}
The map $\varphi_K$ preserves the foliations $\folo$, $\folc^{p_1}$ and $\folc^{p_2}$ (taking leaves to leaves). Moreover, the following hold. 
\begin{itemize} 
    \item Recall that $\folo_Z$ is the leaf of $\folo$  passing through the point $Z$. 
    \begin{equation} \label{eqtn:distanceEqtn2}
    d_{\mathbb{H}^2}(\folo_{s_3} ,\varphi_K(\folo_{Z}))=K  d_{\mathbb{H}^2}(\folo_{s_3},\folo_Z).
\end{equation}
\item In the Klein model of $\mathbb{H}^2$, the cross ratio of $Z$ and its image $\varphi_K(Z)$ satisfies
\begin{equation*}
    [p_2,\varphi_K(Z),\eta_1, p_1]=[p_2,Z,\eta_1, p_1]^K
\end{equation*}
where $\eta_1$ is the intersection point of the geodesic $\overline{ts_1}$ and the boundary of the hyperbolic plane $\mathbb{H}^2$.
\end{itemize}
\end{prop}

\begin{proof}
The fact that $\varphi_K$ preserves the foliation $\folo$, and $d_{\mathbb{H}^2}(\folo_{s_3} ,\varphi_K(\folo_{Z}))=K  d_{\mathbb{H}^2}(\folo_{s_3},\folo_Z)$ is clear from the map $R\to R^K$. 

We denote the $x$ coordinate of  $Z$ and  $\varphi_K(Z)$ respectively by $x_{Z}$ and $x_{\varphi_K(Z)}$. We denote similarly for $y$ coordinates. It is straightforward that the map $\theta\to \alpha$ with $\cos \alpha =2^{K-1}(\cos \theta)^K$ yields
\begin{equation*}
    x_{\varphi_K(Z)}= 2^{K-1} x_{Z}^K.
\end{equation*}

It then follows immediately in terms of cross ratio that 
$$
    [p_2,\varphi_K(Z),\eta_1, p_1]
    =[p_2,Z,\eta_1, p_1]^K.
$$
Morevoer, this implies that $\varphi_K$ preserves the foliation $\folc^{p_1}$, taking one leaf of $\folc^{p_1}$ to another.

Finally, let us prove that the map $\varphi_K$ preserves the foliation $\folc^{p_2}$. We notice that the half circle $C_Z$ passing through $p_2=0$ and a point $Z$ is of radius 
$$r_Z=\frac{x^2_Z+y^2_Z}{2x_Z}.$$

So its image $\varphi_K(Z)$ lies on the circle $C_{\varphi_K(Z)}$ passing through $p_2$ of radius
$$r_{\varphi_K(Z)}= \frac{x^2_{\varphi_K(Z)}+y^2_{\varphi_K(Z)}}{2x_{\varphi_K(Z)}}.$$
From the map $R\to R^K$, it is clear $x^2_{\varphi_K(Z)}+y^2_{\varphi_K(Z)} = (x^2_Z+y^2_Z)^K$. Together with the fact $x_{\varphi_K(Z)}= 2^{K-1} x_{Z}^K$, we conclude $r_{\varphi_K(Z)}=r_Z^{K}$. Since this holds for an arbitrary point $Z$ in the triangle  $\T_{p_1}$, the map $\varphi_K$ preserves $\folc^{p_2}$ as well.

\end{proof}

We extend our map $\varphi$ from the triangle  $\T_{p_1}$ by symmetry to $\T_{p_1p_2t}$.

\begin{prop} \label{prop,ortho-geodesicsExtreme}
  The map  $\varphi_K: \T_{p_1p_2t} \longrightarrow \T_{p_1p_2t}$ given by
  \begin{equation*}
    \varphi_K(R,\theta)=(R^K, \alpha),
\end{equation*}
 is a Lipschitz map with Lipschitz constant $\text{Lip}(\varphi_K)=K$. 
\end{prop}

\begin{proof}

The differential of $\varphi_K$ in the polar coordinate is
\begin{equation*}
D\varphi_K|_{(R,\theta)}=
\begin{pmatrix}
KR^{K-1} & 0 \\
0 & f(\theta) 
\end{pmatrix},
\end{equation*}

where 
\begin{align*}
    f(\theta)= \frac{\partial \arccos(2^{K-1}(\cos \theta)^K)}{\partial \theta}=\frac{2^{K-1}K(\cos \theta)^{K-1}\sin \theta}{\sqrt{1-2^{2(K-1)}(\cos \theta)^{2K}}}.
\end{align*}

Recall the hyperbolic metric in the polar coordinate is $g_{\scriptscriptstyle\mathbb{H}^2}|_{(R,\theta)}=\frac{d R^2 +R^2 d \theta^2}{(R\sin \theta)^2}$. With these information, the norm of the differential is computed as,
\begin{align*}
    \norm{D\varphi_K}_{Z} &=\max_{v\in T_{\scriptscriptstyle Z}\mathbb{H}^2} \frac{\norm{D\varphi_K(v)}_{g_{\scriptscriptstyle\mathbb{H}^2}}}{\norm{v}_{g_{\scriptscriptstyle\mathbb{H}^2}}}\\
    &=\max_{v\in T_z\mathbb{H}^2}\bigg\{\frac{K[{v_r}^2+(\frac{f(\theta)}{K} R v_{\theta})^2]^{\frac{1}{2}}}{({v_{r}}^2+R^2 {v_{\theta}}^2)^{\frac{1}{2}}} \cdot \frac{\sin \theta}{\sin \alpha}\bigg\}.
\end{align*}
Here $v=(v_r,v_{\theta})$ is the polar coordinate of the tangent vector $v$. 
Observe  $\frac{\pi}{2} \geq \alpha \geq \theta\geq 0$. So
\begin{equation}
    \frac{\sin \theta}{\sin\alpha}\leq 1 \text{  
 and }\frac{\cos \alpha}{\cos \theta} \leq 1,
\end{equation}
In addition,
\begin{equation}
    \left(\frac{f(\theta)}{K}\right)^{2}\leq 1 	\Longleftrightarrow \left(\frac{\cos \alpha}{\cos \theta}\right)^2\leq 1.\\
\end{equation}

We conclude $ \norm{D\varphi_K}_{Z} \leq K$ while equality $ \norm{D\varphi_K}_{Z} = K$ holds precisely when $\theta=\alpha$. This shows that $\varphi_K$ is $K$-Lipschitz on $\T_{p_1p_2t}$.
\end{proof}

By further extension using symmetry, we obtain 

\begin{cor}\label{cor,ExtremeLip}
  The map  $\varphi_K: \T_{p_1p_2p_3} \text{  } \xrightarrow{\hspace*{0.6cm}} \text{  }\T_{p_1p_2p_3}$
   is a Lipschitz map with Lipschitz constant $\text{Lip}(\varphi_K)=K$. 
\end{cor}

Extending the two ortho-geodesic stretch maps contructed above from ideal triangles to the entire surface yield the same stretch path as defined in Thurston's work \cite{ThurstonStretch}.

\bibliography{ref}
\bibliographystyle{alpha}

\end{document}